\documentclass{amsart}

\usepackage{amsrefs}
\usepackage{euscript}
\usepackage{epsf,color}
\usepackage{epsfig}
\usepackage{amsmath}
\usepackage{amscd}
\usepackage{amssymb}
\usepackage{enumerate}

\usepackage{amsfonts}
\usepackage{graphicx}
\usepackage[all]{xy}
\usepackage{setspace}

\numberwithin{equation}{section}

\newcommand{\bC}{{\mathbb C}}
\newcommand{\bR}{{\mathbb R}}
\newcommand{\bZ}{{\mathbb Z}}

\newcommand{\cA}{{\mathcal A}}
\newcommand{\cC}{{\mathcal C}}
\newcommand{\cE}{{\mathcal E}}
\newcommand{\cF}{{\mathcal F}}
\newcommand{\cG}{{\mathcal G}}

\newcommand{\cM}{{\mathcal M}}
\newcommand{\cP}{{\mathcal P}}
\newcommand{\cR}{{\mathcal R}}
\newcommand{\cS}{{\mathcal S}}
\newcommand{\cT}{{\mathcal T}}

\newcommand{\ro}{{\mathrm o}}

\newcommand{\sD}{{\EuScript D}}
\newcommand{\sH}{{\EuScript H}}
\newcommand{\sJ}{{\EuScript J}}
\newcommand{\sL}{{\EuScript L}}
\newcommand{\sN}{{\EuScript N}}
\newcommand{\sQ}{{\EuScript Q}}
\newcommand{\sT}{{\EuScript T}}

\newcommand{\vp}[1][\!]{\vec{p}^{\, #1}}

\newcommand{\vx}[1][\!]{\vec{x}^{\, #1}}
\newcommand{\vy}[1][\!]{\vec{y}^{\, #1}}
\newcommand{\vsig}[1][\!]{\vec{\check{\sigma}}^{\, #1}}
\newcommand{\vpsi}[1][\!]{\vec{\psi}^{\, #1}}
\newcommand{\vPsi}[1][\!]{\vec{\Psi}^{\, #1}}
\newcommand{\vI}[1][\!]{\vec{I}^{\, #1}}
\newcommand{\vQ}[1][\!]{\vec{Q}^{\, #1}}
\newcommand{\vP}[1][\!]{\vec{P}^{\, #1}}
\newcommand{\vR}[1][\!]{\vec{R}^{\, #1}}

\newcommand{\vPil}[1][\!]{\vec{\Pil}^{\, #1}}

\newcommand{\Hom}{\operatorname{Hom}}
\newcommand{\id}{\operatorname{id}}

\newcommand{\into}{\hookrightarrow}

\newcommand{\res}{\operatorname{res}}
\newcommand{\ev}{\operatorname{ev}}
\newcommand{\sol}{\operatorname{sol}}

\newcommand{\C}{C}
\newcommand{\bfX}{\mathbf{X}}
\newcommand{\bfJ}{\mathbf{J}}
\newcommand{\bfK}{\mathbf{K}}
\newcommand{\Rbar}{\overline{\cR}}
\renewcommand{\dbar}{\bar{\partial}}
\newcommand{\Stasheff}{\cT}
\newcommand{\Stasheffbar}{\overline{\cT}}
\newcommand{\Shrub}{\cS}
\newcommand{\Shrubbar}{\overline{\cS}}
\newcommand{\Shrubbarext}{\Shrubbar^{\mathrm{ext}}}
\newcommand{\Shrubhat}{\hat{\cS}}
\newcommand{\Champ}{\cC}
\newcommand{\Champbar}{\overline{\cC}}
\newcommand{\Grad}{\cM}
\newcommand{\Gradbar}{\overline{\cM}}
\newcommand{\Champext}{\overline{\cC}^{\mathrm{ext}}}
\newcommand{\Pil}{\cP}
\newcommand{\Pilbar}{\overline{\cP}}
\newcommand{\inte}{\operatorname{int}}
\newcommand{\Simp}{\operatorname{Simp}}
\newcommand{\Morse}{\operatorname{Morse}}
\newcommand{\M}{\mathrm{M}}
\renewcommand{\S}{\mathrm{S}}
\newcommand{\F}{\mathrm{F}}
\newcommand{\Fuk}{\operatorname{Fuk}}

\newcommand{\ins}{\mathrm{in}}
\renewcommand{\mid}{\mathrm{mi}}

\newcommand{\Crit}{\operatorname{Crit}}
\def\co{\colon\thinspace}

\newtheorem{thm}{Theorem}[section]
\newtheorem{cor}[thm]{Corollary}
\newtheorem{lem}[thm]{Lemma}
\newtheorem{prop}[thm]{Proposition}
\newtheorem{defin}[thm]{Definition}
\newtheorem{def-lem}[thm]{Definition-Lemma}

\theoremstyle{remark}

\newtheorem{rem}[thm]{Remark}

\newtheorem*{claim}{Claim}

\newcommand{\superscript}[1]{\ensuremath{^{\textrm{#1}}} }

\renewcommand{\th}[0]{\superscript{th}}
\newcommand{\st}[0]{\superscript{st}}

\newcommand{\noproof}{\qed}

\title[Fukaya categories of plumbings]{A topological model for the Fukaya categories of plumbings}

\author[M.~Abouzaid]{Mohammed Abouzaid} \date{\today}
\thanks{ This research was conducted during the period the author served as a Clay Research Fellow. }

\begin{document}

\begin{abstract}
We prove that the algebra of singular cochains on a smooth manifold, equipped with the cup product, is equivalent to the $A_{\infty}$ structure on the Lagrangian Floer cochain group associated to the zero section in the cotangent bundle.  More generally, given embeddings with isomorphic normal bundles of a closed manifold $B$ into manifolds $Q_1$ and $Q_2$, we construct a differential graded category from the singular cochains of these spaces, and prove that it is equivalent to the $A_{\infty}$ category obtained by considering exact Lagrangian embeddings of $Q_1$ and $Q_2$ which intersect cleanly along $B$.
\end{abstract}

\maketitle
\setcounter{tocdepth}{1}
\tableofcontents
\parskip1em

\section{Introduction}
In \cite{FO}, Fukaya and Oh proved that counts of holomorphic discs in a cotangent bundle with boundary conditions along exact Lagrangian sections agree, in a certain degeneration, with a count of gradient trees.  This has been widely expected to lead to an equivalence between the $A_{\infty}$ structure defined on Floer cochains of a Lagrangian $Q$, and the chain level cup product on the classical (\v{C}ech, simplicial, singular) models for its ordinary cohomology.  Over $\bR$, Kontsevich and Soibelman in \cite{KS} described an argument using the result of Fukaya and Oh and Homological perturbation theory on de Rham cohomology, which proves this result.  A different proof, also over $\bR$, is given by Fukaya, Oh, Ohta, and Ono's in their book on Lagrangian Floer cohomology \cite{FOOO}*{Section 33}.  One corollary of this paper is that such an equivalence holds over the integers:
\begin{thm} \label{thm:cotangent_bundle}
There is an $A_{\infty}$ equivalence
\begin{equation}  C^{*}(Q)   \to  CF^{*}(Q,Q) . \end{equation}
\end{thm}
\begin{rem}
In order to define  $CF^{*}(Q,Q)$ as an honest $A_{\infty}$ algebra (rather than work with partially defined algebraic structures as in \cite{KS}) we follow the approach used by Seidel in \cite{seidel-book}.  The main idea is to make choices of perturbations to ensure the genericity of all moduli spaces that enter into the definition of the operations.  Seidel proves that the resulting structure is independent, up to $A_{\infty}$ equivalence, of these choices.  
\end{rem}
As in \cite{KS} we pass through a Morse model in order to prove this equivalence. First, we imitate the construction of the Fukaya category to obtain an $A_{\infty}$ structure on Morse cochains by counting perturbed gradient trees.  Writing $CM^{*}(Q)$ for this $A_{\infty}$ algebra, we prove the existence of $A_{\infty}$ functors
\begin{equation}  C^{*}(\sQ) \to  CM^{*}(Q) \leftarrow CF^{*}(Q,Q) \end{equation}
inducing  isomorphisms on cohomology.  Here,  $C^{*}(\sQ)$ are the cochains of a simplicial triangulation of $Q$.  Theorem \ref{thm:cotangent_bundle} follows from the fact that $A_{\infty}$ maps inducing an isomorphism on cohomology admit (quasi)-inverses whenever the underlying groups are free abelian.

We do not rely on the degeneration technique of \cite{FO} in order to prove the equivalence of the $A_{\infty}$ structures coming from Floer and Morse theory.  Rather, we consider moduli spaces that are built from gradient trees and holomorphic discs with one Lagrangian boundary condition along an arbitrary leaf of the foliation of the cotangent bundle by fibres.    The corresponding abstract moduli spaces give a new realisation of the multiplihedra  controlling $A_{\infty}$ functors foreseen by Stasheff in \cite{stasheff}.

Theorem \ref{thm:cotangent_bundle} is the special case of a more general result:   within an exact symplectic manifold $W$ with vanishing first chern class (and a choice of a complex volume form), we define a brane to be an exact Lagrangian $Q$ which is relatively spin, and such that the restriction of the complex volume form to some Weinstein neighbourhood of $Q$ is isotopic to the complexification of a (real) volume form on $Q$ (more precisely, we fix this isotopy).  These extra choices of data are necessary in order for the Fukaya category to be defined over $\bZ$ and admit  natural gradings.

Given a pair of exact Lagrangian branes $Q_1$ and $Q_2$ intersecting cleanly along a submanifold $B$,  the Lagrangian condition implies that the normal bundles of $Q_1 \cap Q_2$  in the two sheets are isomorphic.   In Section \ref{sec:simp}, we construct a differential graded category
\begin{equation} \Simp(Q_1,Q_2)  \end{equation}
with morphism spaces given by simplicial cochains on $Q_i$, $NB$, and $(NB, \partial NB)$ (here $N B$ is the unit normal bundle of $B$ in either $Q_1$ or $Q_2$).   The main result of this paper relates this category  to the full subcategory of the Fukaya category of $W$ with objects $Q_1$ and $Q_2$, and to a Morse theoretic model $\Morse(Q_1,Q_2)$:
\begin{thm} \label{thm:general_result}
If $\omega$ vanishes on $\pi_{2}(M, Q_1 \cup Q_2)$, there are $A_{\infty}$ equivalences
\begin{equation} \Simp(Q_1,Q_2) \to \Morse(Q_1,Q_2) \leftarrow \Fuk(Q_1,Q_2). \end{equation}
\end{thm}
\begin{rem}
 For clarity of exposition,  we have chosen to write all arguments first whenever $B$ is a point, with the case of clean intersection relegated to Appendix \ref{sec:case-clean-inters}.  The proofs we give all extend straightforwardly to the clean intersection case, so the main point of the appendix is to explain the proper definitions and constructions in the general situation. 
\end{rem}

At the level of cohomology, Theorem \ref{thm:general_result} is not too difficult to prove from Po{\'z}niak's thesis \cite{pozniak}.   A standard application of Weinstein's neighbourhood theorem shows that a neighbourhood of $Q_{1} \cup Q_2$ in $W$ is symplectomorphic to the result of gluing the cotangent bundles of $Q_1$ and $Q_2$ along the cotangent bundle of $B$.  The diffeomorphism type of the resulting manifold is called the \emph{plumbing} of the two cotangent bundles \cite{milnor}.  Up to an appropriate notion of deformation equivalence which does not affect the Fukaya category, there is a canonical symplectic form on this plumbing.

The experts will observe that the condition that $\omega$ vanish on $\pi_{2}(M, Q_1 \cup Q_2)$ is stronger than the mere exactness of $Q_1$ and $Q_2$.  One may think of it as a strict notion of exactness for the union of $Q_1$ and $Q_2$ as an immersed Lagrangian submanifold of $M$.  For example, as we allow for $B$ to be disconnected, our result gives a description of the category generated by a ``cycle'' of Lagrangians  $\{ Q_i \}_{i=1}^{d}$ such that $Q_i$ intersects only $Q_{i-1}$ and  $Q_{i+1}$ with the understanding that the labels are computed $\mod d$.  However, our description is only valid in the case we know in addition that a loop running around this cycle cannot bound any holomorphic discs (for energy reasons).

We recall that, in  \cite{ruan}, Ruan extended the degeneration results of Fukaya and Oh to the case of an immersed Lagrangian $Q \in M$ with double points such that $\pi_{2}(M,Q)$ vanishes, and proved that counts of rigid holomorphic discs can be made to agree with counts of gradient trees.  This indicates a straightforward generalisation of  Theorem \ref{thm:general_result}: given two embeddings $B \to Q $  whose images do not intersect and with isomorphic normal bundles $N B$, it is easy to extend the construction of $ \Simp(Q_1,Q_2)  $  to produce an associative product on the direct sum
\begin{equation}  \label{eq:immersed_floer}  C^{*}(Q) \oplus C^{*}(NB)  \oplus   C^{*}(NB, \partial NB) \end{equation}
which satisfies the Leibniz rule with respect to the obvious differential.  This algebra is then equivalent, as an $A_{\infty}$ algebra, to the Lagrangian Floer cochains of $Q$ considered as an immersed Lagrangian in the ``self-plumbing'' of its cotangent bundle along $B$.  Floer cohomology for immersed Lagrangians was described in great generality by Akaho and Joyce in \cite{akaho}, producing an $A_{\infty}$ structure on a vector space which, up to choosing different model for a cochain-level theory, is equivalent to the direct sum \eqref{eq:immersed_floer}.  One may interpret our result as a computation, in the exact case, of the $A_{\infty}$ structure assigned by Akaho and Joyce to an immersed Lagrangian.  Indeed, their $A_{\infty}$ structure  comes from the count of perturbed holomorphic discs; the perturbations are necessary because of the abundance of constant discs with marked points which do not form regular moduli spaces  as soon as the number of marked points is greater than $3$.  As is often the case with such approaches, it seems impossible to give an explicit direct counts of the number of discs after perturbations.

Finally, we note that while every $A_{\infty}$ category is abstractly equivalent to a differential graded category, there is no construction which produces a finite dimensional differential graded category when the morphism spaces are finite rank abelian groups.  One consequence of Theorem \ref{thm:general_result} is the existence of such a finite dimensional differential graded model for  Fukaya categories in the case of plumbings.  We insist, however, that our construction requires working with cochains on $Q_1$ and $Q_2$ even if these manifolds happen to be formal (i.e have the property that their cohomology are quasi-isomorphic to their cochain algebras).  As explained in Chapter 20 of \cite{seidel-book}, it is already true that one can find $5$ exact Lagrangians in a twice punctured genus $2$ curve, satisfying our exactness condition, such that the subcategory of the Fukaya category they generate is not formal (despite the fact that the circle is formal). 

\subsection*{Acknowledgments}
I would like to thank Paul Seidel for answering my many questions about Chapter 20 of \cite{seidel-book}.

\section{Three categories}

For the remainder of this paper, we fix a pair of closed smooth manifolds $Q_1$ and $Q_2$ with base points $b_i$, and open sets $U_i$ containing $b_i$.  In the most basic situation, $U_i$ is the interior of a closed (embedded) ball.  We shall write $Q_{1,2}$ for the closure of $U_1$ and $Q_{2,1}$ for the closure of $U_{2}$. 

In addition, we equip $Q_1$ and $Q_2$ with simplicial triangulations $\sQ_1$ and $\sQ_2$, and choose top dimensional simplices $\Delta_{1}$ and $\Delta_{2}$ respectively containing $b_1$ and $b_2$.   Recall that the condition that a triangulation be simplicial implies that we have an ordering of all vertices; in particular, $\Delta_i$ has a unique identification $\iota_i$ with the standard $n$-simplex $\Delta$ which preserves the  ordering (see pages 102 and 107 of \cite{hatcher} for a discussion of different notions of triangulations).  
\subsection{A simplicial category from a triangulation of $Q_i$} \label{sec:simp}
We begin by constructing a differential graded category with two objects called $Q_1$ and $Q_2$ and morphisms the simplicial cochain complexes
\begin{align} \label{eq:morphisms_simp}
\Hom_{*}^{\S}(Q_i,Q_i) & = C^{*}(\sQ_i) \\
\Hom_{*}^{\S}(Q_1,Q_2) & = C^{*}(\Delta) \\
\Hom_{*}^{\S}(Q_2, Q_1) & = C^{*}(\Delta, \partial \Delta).
\end{align}
In particular, the rank of $\Hom(Q_1,Q_2)$ is $2^n-1$, while the rank of $\Hom(Q_2, Q_1)$ is $1$.  The cochain complexes $C^{*}(\sQ_i)$ are differential graded algebras with respect to the cup product; we shall extend this product to a differential graded composition law by noting that the inclusions of $\Delta_i$ in $\sQ_i$ induce ring homomorphisms
\begin{equation}
 \iota_{i}^{*} \co C^{*}(\sQ_i) \to C^*(\Delta).
\end{equation}
We define compositions using cup-product on the left and on the right with the restricted classes
\begin{align} \label{eq:multiplication_cochain_delta}
C^*(\Delta)  \otimes C^{*}(\sQ_1)   \stackrel{\id \otimes \iota_{1}^*}{\longrightarrow}  & C^*(\Delta) \otimes  C^*(\Delta) \stackrel{\cup}{\longrightarrow}    C^*(\Delta) \\ 
C^{*}(\sQ_2) \otimes  C^*(\Delta) \stackrel{\iota_{2}^*\otimes \id}{\longrightarrow}  &  C^*(\Delta) \otimes  C^*(\Delta) \stackrel{\cup}{\longrightarrow}   C^*(\Delta) .
\end{align}
Since the maps $\iota_{i}^{*}$ are maps of differential graded algebras, and the Leibniz and associativity formulae hold for the cup product and differential on $C^*(\Delta)$, we readily conclude that the multiplication \eqref{eq:multiplication_cochain_delta} satisfies the Leibniz rule, and moreover defines an associative composition.  For example, we have a commutative diagram
\begin{equation}
 \xymatrix @C=-15pt{  C^{*}(\sQ_2) \otimes  C^*(\Delta)   \otimes C^{*}(\sQ_1)   \ar[rr] \ar[dd] \ar[rd] & & C^{*}(\Delta) \otimes  C^*(\Delta)   \otimes C^{*}(\sQ_1)  \ar[dd] \ar[dl] \\
&  C^{*}(\Delta) \otimes  C^*(\Delta)   \otimes C^{*}(\Delta) \ar[dr] & \\ 
C^{*}(\sQ_2) \otimes  C^*(\Delta)   \otimes C^{*}(\Delta)   \ar[rr] \ar[ru] & & C^{*}(\Delta)
}
\end{equation}

In order to define this composition in the other direction, we use the fact that the cup product also makes $C^*(\Delta, \partial \Delta)$ into a bimodule over $C^*(\Delta)$.  In particular, we define composition maps
\begin{align} \label{eq:multiplication_rel_cochain_delta}
 C^{*}(\sQ_1) \otimes  C^*(\Delta, \partial \Delta) \stackrel{\iota_{1}^*\otimes \id}{\longrightarrow}  &  C^*(\Delta) \otimes  C^*(\Delta, \partial \Delta) \stackrel{\cup}{\longrightarrow}   C^*(\Delta, \partial \Delta) \\
 C^*(\Delta, \partial \Delta)  \otimes C^{*}(\sQ_2)   \stackrel{\id \otimes \iota_{2}^*}{\longrightarrow}  & C^*(\Delta, \partial \Delta) \otimes  C^*(\Delta) \stackrel{\cup}{\longrightarrow}    C^*(\Delta, \partial \Delta).
\end{align}
Again, the proof that these maps satisfy the properties required of a differential graded category follows immediately from the fact that $C^*(\Delta, \partial \Delta)$ is in fact a differential graded bi-module over $C^*(\Delta)$ (this is trivial since the differential on this relative cochain group vanishes).  Finally, the inclusions of $\Delta_i$ in $\sQ_i$ induce maps
\begin{equation}
{\iota_{i}}_{!} \co C^*(\Delta,\partial \Delta) \to C^{*}(\sQ_i).
\end{equation}
Since $C^*(\Delta,\partial \Delta)$ has rank $1$, there is an obvious formula for these maps; the image is simply the $n$-dimensional cochain with value $1$ on $\Delta_i$ and which vanishes on every other chain.  More conceptually, this map is induced by a composition
\begin{equation}
C^*(\Delta,\partial \Delta) \stackrel{\sim}{\to} C^{*}(\sQ_i, \sQ_i - \inte(\Delta_i)) \to C^{*}(\sQ_i),
\end{equation}
where the first map is an excision isomorphism.  Using this, we define the final two composition maps
\begin{align} \label{eq:multiplication_shriek}
C^*(\Delta, \partial \Delta)  \otimes C^*(\Delta)  \stackrel{\cup}{\longrightarrow}  &  C^*(\Delta,\partial \Delta)  \stackrel{{\iota_{1}}_{!}}{\longrightarrow} C^{*}(\sQ_1) \\
C^*(\Delta) \otimes C^*(\Delta, \partial \Delta)    \stackrel{\cup}{\longrightarrow}  & C^*(\Delta, \partial \Delta) \stackrel{{\iota_{2}}_{!} }{\longrightarrow}    C^*(\sQ_2) .
\end{align}
There are $8$ associativity diagrams whose commutativity remains to be checked.  Even though the construction is not completely symmetric, we only discuss the following four maps:
\begin{align} 
\notag C^{*}(\sQ_2) \otimes C^*(\Delta) \otimes C^*(\Delta, \partial \Delta)   \longrightarrow &    C^*(\sQ_2) \\
\notag C^*(\Delta) \otimes C^{*}(\sQ_1) \otimes C^*(\Delta, \partial \Delta)   \longrightarrow &    C^*(\sQ_2) \\
\notag C^*(\Delta) \otimes C^*(\Delta, \partial \Delta)  \otimes C^{*}(\sQ_1) \longrightarrow &    C^*(\sQ_2) \\
C^*(\Delta, \partial \Delta)  \otimes C^*(\Delta) \otimes C^*(\Delta, \partial \Delta)  \longrightarrow  &   C^{*}(\Delta). \label{eq:simp_chains_2_1_2}
\end{align}
Associativity of the first three of the above operations follows immediately from the property that $C^*(\Delta, \partial \Delta)$ is a differential graded bi-module over $C^*(\Delta)$.  The last operation requires in addition the fact that the composition
\begin{equation}
 C^*(\Delta,\partial \Delta) \stackrel{ {\iota_{i}}_{!}}{\to} C^{*}(\sQ_i)  \stackrel{\iota_{i}^{*}}{\to} C^*(\Delta).
\end{equation}
agrees with the natural inclusion of relative cochains into ordinary cochains (alternatively, one can notice that the two maps we are comparing in Equation \eqref{eq:simp_chains_2_1_2} vanish altogether).  We have provided all the necessarily ingredients to construct a differential graded category:

\begin{defin}
The morphisms in Equation \eqref{eq:morphisms_simp}, and the composition laws given by Equations \eqref{eq:multiplication_cochain_delta}, \eqref{eq:multiplication_rel_cochain_delta}, and \eqref{eq:multiplication_shriek} satisfy the axioms of a differential graded category.  We denote the corresponding $A_{\infty}$ category where the differential and the composition are twisted by a sign
\begin{align} \label{eq:twist_diff_dg_to_A_infty}
\mu_{1}^{\S}(\check{\sigma}) & = (-1)^{\deg(\check{\sigma})} d \check{\sigma} \\
 \mu_{2}^{\S}(\check{\sigma},\check{\tau}) & = (-1)^{\deg(\check{\tau})} \check{\sigma} \cup \check{\tau}
\end{align}
by $\Simp(\sQ_1,\sQ_2)$.
\end{defin}

We shall also need a geometric interpretation of the cup product as an intersection product on the chains of subdivisions dual to the chosen triangulations of $Q_{i}$.  Let us therefore fix such dual  subdivisions $\check{\sQ}_{i}$, whose top dimensional cells are identified, by smooth charts, with polyhedra in $\bR^{n}$.   In particular, the $n-1$-dimensional cells intersect cleanly.  

\begin{rem} \label{rem:sign_difference}
The sign change in \eqref{eq:twist_diff_dg_to_A_infty} is particularly convenient from the point of view of the dual cell subdivision: if we think geometrically, every generator of $C_{*}(\sQ)$ corresponds to a  simplex carrying an orientation, and the differential is the restriction of this orientation to the boundary.  If we want to think of the differential $d$  on $C^{*}(\sQ)$ as taking a dual cell to its boundary, we find that the resulting operation differs from the natural restriction of orientations exactly by a sign $(-1)^{1+\deg(\check{\sigma})}$.  The additional sign of $1$ comes from the fact that an outward pointing vector at the boundary of a cell corresponds to an inward pointing vector from the dual subdivision's point of view.
\end{rem}

We shall also require that the identification $\Delta_{1} \cong \Delta_{2}$ map the cells of $\check{\sQ}_{1} \cap \Delta_{1}$ to the corresponding cells of $\check{\sQ}_{2} \cap \Delta_{2}$.  In particular, we obtain a subdivision dual to the trivial triangulation of the standard simplex $\Delta$.  Given a cell $\sigma$ of $\sQ_i$, we write $\check{\sigma}$ for the dual cell of $\check{\sQ}_i$, as well as for the dual generator of $C^{*}(\sQ_i)$.  A slight generalization of the main observation in Appendix E of \cite{abouzaid} is the existence of vector fields $X_i$ on $Q_i$ which generate flows $\phi^{i}_{t}$ such that the following Lemma holds

\begin{lem} \label{lem:cup_is_intersection}
Whenever $\sigma$, $\tau$, and $\rho$ are cells of $\sQ_i$ such that 
\begin{equation}
 \check{\sigma} \cup \check{\tau} = \check{\rho},
\end{equation}
there is a smooth map
\begin{equation}
\check{\rho} \times [0,1] \to Q_i
\end{equation}
such that the image of $\check{\rho} \times \{t\}$ agrees with $\check{\sigma} \cap \phi^{i}_{t}(\check{\tau})$ for any sufficiently small $t$. \noproof
\end{lem}
Since  $\check{\sigma} \cup \check{\tau}$ either vanishes or is of the form $\check{\rho}$ for some cell $\rho$, we conclude that the cup product $\check{\sigma} \cup \check{\tau}$ vanishes whenever $\check{\sigma} \cap \phi^{i}_{t}(\check{\tau})$ is empty for sufficiently small $t$, and we interpret the above Lemma to state that the cup product is given by the intersection product of the dual cells, provided that they are appropriately perturbed.

The construction of the vector fields $X_i$ is  elementary; since $\sQ_i$ is a simplicial triangulation, its vertices are ordered, hence so are the top dimensional cells of dual subdivision.  We may choose $X_i$ to be any vector field satisfying the following property
\begin{defin}
A  vector field $X_i$ is said to be {\bf compatible with the simplicial triangulation} if along each cell $\check{\sigma}$ the restriction of $X_i$ lies in the tangent space of the top dimensional cell which is minimal among those adjacent to $\sigma$. 
\end{defin}
By abuse of language, we shall often call such vector fields {\bf simplicial}. The existence of such a vector field follows from a local argument described in Appendix E of \cite{abouzaid}. Since the above condition is respected by taking positive linear combinations, we may and shall require the restrictions of $X_1$ and $X_2$ to $\Delta_1$ and $\Delta_2$ to be intertwined by the maps identifying these two simplices.  Figure \ref{fig:simplicial_vector_field} shows the restriction of such a vector field to the $1$-skeleton of the dual subdivision of a $2$ dimensional simplex.

\begin{figure}[h] 

 \input{simplicial_vector_field.pstex_t}
   \caption{}
   \label{fig:simplicial_vector_field}
\end{figure}

In terms of the dual subdivisions $\check{\sQ}_i$, the restriction of cochains corresponds to taking the inverse image of a cell, while the map $\iota_{i !}$ takes the vertex of the subdivision dual to $\Delta$ to its image under the inclusions into $Q_i$.  Since composition maps in the category $\Simp(\sQ_1,\sQ_2)$ are defined using these operations together with the cup product, Lemma  \ref{lem:cup_is_intersection} provides an intersection theoretic interpretation of this category.

\subsection{A Morse category} \label{sec:stasheff}
Fix a pair of open sets $U_i \subset Q_i$ together with diffeomorphisms
\begin{equation} \label{eq:diffeo_open_sets_Q_i,j}
 \bar{U}_1 \cong Q_{1,2} \cong Q_{2,1} \cong \bar{U}_2,
\end{equation}
where $Q_{1,2}$ and $Q_{2,1}$ are manifolds with boundary (there is no gain in notational simplicity in assuming anything more).    To make the connection with the rest of the paper, we shall later restrict to the situation where $U_1$ and $U_2$ are both balls, corresponding to neighbourhoods in $Q_1$ and $Q_2$ of the point $b$ along which we shall perform the plumbing.

The identification between these manifolds shall be fixed throughout the following discussion, and we use it to define the space
\begin{equation*}  Q = Q_{1} \cup_{U_1 \cong U_{2}} Q_{2} \end{equation*}
which is a non-Hausdorff smooth manifold whose locus of points which cannot be separated from each other is the union $\partial \bar{U}_1 \cup \partial \bar{U}_2$.  By construction, this space admits canonical maps from the manifolds $Q_{1,2}$ and $Q_{2,1}$ as well.  For the remainder of the paper, we sometimes also write  $Q_{i,i}$ for $Q_i$ as a convenient notational device.

We assume that we have chosen Riemannian metrics on $Q_i$ whose restrictions to $U_i$ are intertwined by these diffeomorphisms, and we fix a choice of Morse-Smale functions 
\begin{align}
f_{i,j} & \co Q_{i,j}  \to \bR,
\end{align}
such that $f_{1,2}$ and $f_{2,1}$ are respectively inward and outward pointing on the boundary.  In particular, we also choose functions $f_{1,1}$ and $f_{2,2}$ on $Q_1$ and $Q_2$.   Writing $CM^{*}(f)$ for Morse cochains, we shall define an $A_{\infty}$ category with morphism space
\begin{equation} \label{eq:morphisms_morse}
 Hom_{*}^{\M}(Q_i,Q_i) = CM^{*}(f_{i,j}) = \bigoplus_{x \in \Crit(f_{i,j}) } | \ro_{x} |
\end{equation}
where  $| \ro_{x} |$ is the orientation line associated to $\ro_{x} \equiv  \lambda(W^{s}(x))$, i.e. the quotient of the free abelian group generated by a choice of orientation of the stable manifold of $x$ by the relation that the sum of generators associated to opposite orientations vanishes.  This rank $1$ group has degree equal to the dimension of the stable manifold of $x$.  The differential counts negative gradient lines of $f_{i,j}$; the detailed description is given in Section \ref{sec:orientations}.

To define higher compositions in this category, we consider moduli spaces of gradient trees in $Q$ which are defined later in this section.  Given any ribbon tree $T$ with $d+1$ leaves, we shall assume that one of the leaves has been distinguished as outgoing, which determines the unique ordering on the remaining $d$ leaves (called incoming) illustrated in Figure \ref{fig:ordering_convention_trees}.  We shall write $\cE(T)$ for the set of edges of $T$ (including the external edges).  Our conventions are that all edges include their endpoints and, except for the outgoing edge which is always infinite,  are allowed to have length $0$.  

\begin{figure}[h] 

 \input{tree_labeling.pstex_t}
   \caption{}
   \label{fig:ordering_convention_trees}
\end{figure}

The edges adjacent to any vertex of a ribbon tree $T$ are by definition cyclically ordered, and we define a labelling $(i(e),j(e))$ of the edges $e \in \cE(T)$ by elements of the set $\{ 1, 2 \}$  to be compatible with the cyclic ordering if $i(e') = j(e)$ whenever $e'$ follows $e$ with respect to the cyclic ordering at some vertex.  The following result may be proved by induction, or by embedding $T$ in a disc labelling the complementary regions appropriately.
\begin{lem}
A sequence $\vI = (i_0, \ldots, i_{d})$ determines a unique labelling of all edges of $T$ which is compatible with the cyclic ordering, such that the incoming leaves are labelled by $(i_k,i_{k+1})$.  Moreover, the outgoing leaf is labelled by $(i_0,i_{d})$. \noproof
\end{lem}

In order to resolve potential problems with transversality in defining the Morse category, we introduce perturbation data, following Seidel's construction for Lagrangian Floer theory in \cite{seidel-book}.    Let us illustrate the case of the cup product on $CM^{*}(Q_1)$.  Given three critical points $x_0$, $x_1$ and $x_2$, the $x_0$ component of  $\mu_2^{\M}([x_2],[x_1]) $ should count triple intersections between the ascending manifold of $x_0$, and the descending manifolds of $x_1$ and $x_2$.  This space is necessarily empty unless $x_1=x_2$, in which case the intersection is not transverse.   

This triple intersection admits a description as the moduli space of maps from the unique trivalent tree with three infinite external edges, which converge at the respective ends to the critical points $x_0$, $x_1$ and $x_2$, and which on each edge solve the gradient flow equation for the function $f_{1,1}$.  We shall perturb this moduli space by changing the differential equation obeyed by each edge.  Concretely, as each edge is isometric to $[0,+\infty)$, our perturbation will vanish on $(1,+\infty)$ and will be essentially arbitrary on $[0,1]$.  Integrating the perturbed gradient flow along the interval $[0,1]$, we obtain a triple $\phi_0$, $\phi_1$, and $\phi_2$ of diffeomorphisms of $Q_1$.   Changing our perspective once again, we interpret the moduli space of  perturbed gradient trees  as the intersection of the images under $\phi_i$ of the appropriate ascending and descending manifolds.  As the diffeomorphisms $\phi_i$ are essentially arbitrary (and independent of each other) we achieve transversality (and hence an honest count) upon perturbing the gradient flow.  One must still prove that this product is the quadratic term of an $A_{\infty}$ structure, which requires defining compatible perturbations of the gradient flow equation on higher dimensional moduli spaces of trees.

Let us return to the general discussion, and fix a sequence of labels $\vI$ as above:
\begin{defin} \label{defin:pert_datum}
A {\bf gradient flow perturbation datum}  on $(T,g_{T})$ is a choice, for each edge $e \in \cE(T)$ of a family of vector fields 
\begin{equation}
 X_{e} \co e \to \C^{\infty}(TQ_{i(e),j(e)})
\end{equation}
which vanish away from a bounded subset of $e$ and such that the restriction of the vector field
\begin{equation} \label{eq:perturbed_gradient_flow}
 -\nabla f_{i(e),j(e)} + X_{e}
\end{equation}
to the boundary of $Q_{i(e),j(e)}$ is (i) is outward pointing if $(i(e),j(e)) = (1,2)$ and (ii) inward pointing $(i(e),j(e)) = (2,1)$.
\end{defin}
\begin{rem}
 Since the gradient vector fields of $f_{1,2}$ and $f_{2,1}$ are respectively inward and outward pointing, the vanishing perturbation datum satisfies the above properties.
\end{rem}

We shall now specialise and consider the trees which control the $A_{\infty}$ structure on the Morse category.  Let us write $\Stasheff_d$ for the moduli space of {\bf Stasheff trees} $(T,g_T)$, i.e. metric ribbon trees with $d+1$ infinite external edges, and $\Stasheff_{\vI}$ for the moduli space of trees whose inputs are labelled by the successive elements of a sequence $\vI$ (this is a copy of $\Stasheff_{|\vI|}$).  By allowing singular trees, we obtain compactifications $\Stasheffbar_d$ and $\Stasheffbar_{\vI}$ which are in fact polyhedra.   The collection of spaces $\Stasheffbar_d$ for varying $d$ were shown by Stasheff to form an $A_{\infty}$ operad; i.e. to control $A_{\infty}$ algebras. In particular whenever $d=d_1+d_2-1$, we may construct a singular Stasheff tree with $d$ inputs by attaching the output of a singular tree with $d_2$ inputs to any of the inputs of an element of $\Stasheffbar_{d_1}$.  This construction yields a collection of $d_1$ maps
\begin{equation} \label{eq:compose_stasheff}
\Stasheffbar_{d_1} \times \Stasheffbar_{d_2} \to \Stasheffbar_{d}.
\end{equation}
whose images for all pairs $(d_1,d_2)$ cover the boundary of $\Stasheffbar_{d}$.  This familiar construction can also be done with labels: if $\vI[1] = (i^{1}_0, \ldots, i^{1}_{d_1})$ and $\vI[2]=(i^{2}_0, \ldots, i^{2}_{d_2})$ are sequences  such that $(i^{1}_k,i^{1}_{k+1})$ is the label for the output of a tree in $\Stasheff_{\vI[d]}$, we obtain a natural map
\begin{equation}
 \label{eq:compose_stasheff_labels}
\Stasheffbar_{\vI[1]} \times \Stasheffbar_{\vI[2]} \to \Stasheffbar_{\vI},
\end{equation}
where $ \vI = (i^{1}_0, \ldots, i^{1}_{k} = i^{2}_{0}, i^{2}_{1}, \ldots, i^{2}_{d_2} = i^{1}_{k+1}, \ldots,  i^{1}_{d_1})$ and the boundary of $\Stasheffbar_{\vI}$ is again covered by the images of such maps.

\begin{defin}
A {\bf universal consistent perturbation datum for trees}  is a choice $\bfX^{\Stasheff}$ of a smooth family of perturbation data for elements of $\Stasheffbar_{\vI}$ for  every sequence $\vI$,  which is compatible with the gluing maps \eqref{eq:compose_stasheff_labels} and is invariant under the automorphisms of each tree.
\end{defin}

\begin{rem}
The meaning of a smoothness is clarified in Definition \ref{defin:smooth_pert_data_stasheff}.  A priori, each edge in a singular ribbon tree is equipped with at least two perturbation data: one comes from the singular tree, the other from the (smooth) tree wherein the edge lies.  Compatibility with the maps \eqref{eq:compose_stasheff_labels} is the requirement that these two perturbation data agree.  The condition of invariance implies that the perturbation datum vanishes whenever there is only one input (because the $\bR$ translation symmetry forces a non-vanishing perturbation datum to have non-compact support, contradicting the boundedness requirement in Definition \ref{defin:pert_datum}).  
\end{rem}

Every edge $e$ of a Stasheff tree is isometric to a segment in $\bR$.  Writing $t_{e}$ for the induced coordinate, we define a (perturbed) gradient flow line on $e$ to be a map $\psi_e$ from $e$ to $ Q_{i(e),j(e)}$ which is a solution to the differential equation
\begin{equation}
 d\phi_{e} (\partial_{t_e}) = -\nabla f_{i(e),j(e)} + X_{e}.
\end{equation}

\begin{defin}
A {\bf perturbed gradient tree} $\psi$ on $(T,g_T)$ with labels $\vI$ is a continuous map
\begin{equation}
\psi \co T \to Q
\end{equation}
whose restriction to  every edge $e \in \cE(T)$  lifts to a map
\begin{equation}
 \psi_{e} \co e \to  Q_{i(e),j(e)}
\end{equation}
which is a gradient trajectory of $f_{i(e),j(e)}$.
\end{defin}
\begin{rem}
One usually imposes a balancing condition at vertices of $T$; i.e. the sum of all vector fields associated to edges adjacent to $v$ is required to vanish.  This constraint is not part of our setup for gradient trees, and in fact our choices of perturbation data are such that we can choose the vector fields defined by the various edges which share a common vertex to have arbitrary and independent values, as long as the conditions on the boundary of $Q_{1,2}$ or $Q_{2,1}$, described in Definition \ref{defin:pert_datum}, are satisfied.
\end{rem}

If $\vx \equiv (x_{1}, \ldots, x_{d})$ is a sequence of critical points of the functions $f_{i_{k},i_{k+1}}$, and $x_0$ is a critical point of $f_{i_0, i_d}$,  and define
\begin{equation} \Stasheff_d( x_0, \vx ) \end{equation}
to be the set of isomorphism classes of perturbed gradient trees $\psi$ such that the image of the $k$\th incoming leaf is $x_k$, and the image of the unique outgoing leaf is $x_0$.  This moduli space admits a bordification $\Stasheffbar_d(x_0, \vx)$ by adding all singular gradient trees. 

We begin by showing that these moduli spaces are compact:
\begin{lem} \label{lem:compactness_stasheff}
Given $x_0$ a critical point of $f_{i_0,i_d}$, and  $(x_1, \ldots, x_d)$ critical points of the functions $f_{i_k,i_{k+1}}$, the moduli space $\Stasheffbar(x_0, \vx)$ is compact.
\end{lem}
\begin{proof}[Sketch of proof]
When studying Morse theory on a manifold with boundary, one must first show that gradient trees cannot escape to the boundary:
\begin{claim}
There exists a neighbourhood of the boundary of $Q_{i,j}$ whose image in $Q$ is disjoint from the image of every edge $e$ labelled by $Q_{i,j}$ under a gradient tree
\begin{equation}
 \psi \co (T,g_T) \to Q.
\end{equation}
\end{claim}
\begin{proof}[Proof of claim]
The assumption that the perturbation data are consistent implies that there is a uniform (i.e. independent of $(T,g_T)$ in a fixed space $\Stasheff_{d}$) neighbourhood of the boundary  $Q_{1,2}$ (respectively $Q_{2,1}$) in which the perturbed gradient flow \eqref{eq:perturbed_gradient_flow} associated to every edge labelled by $(1,2)$ (or $(2,1)$) of a tree in $\Stasheff_{d}$ is inward (respectively outward) pointing in the sense that it increases the distance to the boundary.  In particular, the gradient flows of $f_{1,2}$ and $f_{2,1}$ are themselves inward or outward pointing in this neighbourhood.  By shrinking these neighbourhoods, we may assume that they are identified by the fixed diffeomorphism $Q_{1,2} \cong Q_{2,1}$.  We shall prove that the images of edges labelled by $Q_{1,2}$ or $Q_{2,1}$ cannot intersect this neighbourhood.

The case of edges labelled by $Q_{2,1}$ is simplest.  Note that if $e$ is such an edge, then it must lie on a descending arc $(e_{0}, \ldots, e_{r} = e)$ such that $e_{0}$ is an incoming leaf, and each edge of the arc is labelled by $Q_{2,1}$.  Assuming by contradiction that the image of $e$ intersects the neighbourhood of the boundary where the  perturbed negative gradient flow points inward, we conclude that $e_0$ is contained in this fixed neighbourhood.  This implies that the image of this incoming leaf is a critical point which is also contained in this neighbourhood, which contradicts the fact that $f_{2,1}$ has no critical points near the boundary.

If an edge is labelled by $Q_{1,2}$, there are two possibilities: either (i) there is a descending arc $(e, e_1, \ldots, e_{r})$ with $e_{r}$ the outgoing leaf such that all edges succeeding $e$ are labelled by $Q_{1,2}$ or (ii) there  exists a descending arc $(e=e_0,e_1, \ldots, e_{r})$ all of whose edge are labelled by $Q_{1,2}$ and an ascending arc $(e'=e'_0, \ldots, e'_{s})$ whose edges are labelled by $Q_{2,1}$ such that  $e_{r}$ and $e'$ are adjacent.  Using the same argument as in the previous case, we conclude that the critical point of $f_{1,2}$ (or $f_{2,1}$) whither $e_{r}$ (or $e'_{s}$) limits must be contained in the previously fixed neighbourhood of the boundary of $Q_{2,1}$ (or $Q_{1,2}$) contradicting our assumptions that the gradient flow may not vanish in these regions.
\end{proof}
This implies that whenever the image of an edge under a family of gradient trees converges toward the singularities of $Q$, the edge must be labelled by $Q_{1,1}$ or $Q_{2,2}$, so all analysis can be done locally in either manifold.   The standard proof of compactness for gradient trees for compact smooth Riemannian manifolds can then be used to prove the desired result, see \cite{FO}.
\end{proof}

The following is the main result proved in Section \ref{app:stasheff}:
\begin{lem} \label{lem:existence_morse_category}
For a generic choice of universal perturbation data, all spaces $ \Stasheffbar_d(x_0, \vx)$ are naturally compact topological manifolds with boundary of dimension
\begin{equation}
 d -2 + \deg(x_0) - \sum_{1 \leq k \leq d} \deg(x_k).
\end{equation}
\end{lem}
\begin{rem}
Technically, $Q$ is a smooth non-Hausdorff manifold, with the set of points which cannot be separated from each other diffeomorphic to $\partial Q_{1,2}$.  Any local construction for smooth manifolds is valid on $Q$ as long as it performed away from this subset.  In addition, if we are working near this subset, and we have additional data that distinguishes one of the sheets, then we may again use standard constructions from the study of smooth manifolds.  Note that this is precisely what happens if we are considering an edge of a gradient tree labelled by $(i,i)$ whose image in $Q$ happens to intersect the non-Hausdorff locus; by definition, we have a lift to the relevant (usual) manifold $Q_i$. The key component in the proof of Lemma \ref{lem:existence_morse_category} is the Claim appearing in Lemma \ref{lem:compactness_stasheff} which asserts that there is a neighbourhood of the non-Hausdorff locus of $Q$ which (perturbed) gradient edges labelled by $Q_{i,j}$ with $i \neq j$ cannot intersect if they form part of a gradient Stasheff tree.  
\end{rem}

Given a sequence $\vx$ such that
\begin{equation}
  d -2 + \deg(x_0) - \sum_{1 \leq k \leq d} \deg(x_k) = 0 
\end{equation}
we conclude that the moduli space $\Stasheff_d(x_0, \vx )$ consists of finitely many points.  The description of this moduli space as a fibre product determines an isomorphism
\begin{equation}   \label{eq:iso_det_bunldes_stasheff} \lambda(\Stasheff(x_0, \vx)) \otimes \lambda(Q^{d+1}) \cong \lambda(Q) \otimes \lambda(\Stasheff_{d}) \otimes \lambda(W^{s}(x_0)) \otimes \lambda (W^{u}(\vx)) \end{equation}
where $W^{u}(\vx)$ is the product of the descending manifolds of the critical points $x_k$.  Our conventions are explained in Section \ref{sec:orientations}. Whenever $\psi$  is a rigid tree in $\Stasheff(x_0, \vx)$, the above isomorphism and Equation \eqref{eq:decom_tangent_space_crit_points} give a natural map
\begin{equation} \ro_{x_d} \otimes \cdots \otimes \ro_{x_1} \to \ro_{x_0} . \end{equation}
We define the $d$\th higher product
\begin{equation}  
\mu^{\M}_{d}  \co CM^{*}(f_{i_{d-1},i_d}) \otimes \cdots \otimes CM^{*}(f_{i_1,i_2})  \to CM^{*}(f_{i_1,i_d}) \end{equation}
 to be a sum over the induced maps $\mu_{\psi}$ on orientation bundles
\begin{equation}
[x_d] \otimes \cdots \otimes [x_{1}]  \mapsto \sum_{\stackrel{ x_0}{\psi \in \Stasheff(x_0, \vx)}}  (-1)^{(n+1) (\deg(x_0) + \dagger( \vx))} \mu_{\psi}  ( [x_d] \otimes \cdots \otimes [x_{1}])  \end{equation}
where the sign is given by:
\begin{equation}  \label{eq:sign_twisting_morse} \dagger( \vx) = \sum_{k=1}^{d} k \deg(x_{k}).  \end{equation}

The following result follows directly from Proposition \ref{lem:existence_morse_category} and the proof of the analogous result in Appendix C of \cite{abouzaid}. 
\begin{prop}
The operations $\mu^{\M}_d$ satisfy the axioms of an $A_{\infty}$ category.  \noproof
\end{prop}

In particular, we define the $A_{\infty}$ category $\Morse(Q_1,Q_2)$ to have objects and morphisms as in Equation \eqref{eq:morphisms_morse}, and (higher) compositions defined by the operations $\mu^{\M}_d$.

\begin{rem}
 The orientation convention used in \cite{abouzaid} differs in a very minor way from the one introduced here.  In particular, Equation \eqref{eq:sign_twisting_morse} is a simplification of Equation (C.2) in the previous paper.
\end{rem}

\subsection{The Fukaya category of the skeleton of a plumbing} \label{sec:introduce_fukaya}
On $\bC^{n}$ equipped with the standard symplectic form, complex structure, and with the coordinates $(x_1, \ldots, x_n, y_1, \ldots, y_n) = (\vx, \vy)$, the Lagrangians 
\begin{align*}
 L_1=\bR^{n} & = \{ \vy = 0 \} \\
 L_2 = i \bR^{n} & = \{ \vx = 0 \}
\end{align*}
intersect transversely at the origin.  We define the plumbing
\begin{equation} D^{*} L_1 \# D^* L_2 \end{equation}
to be the open neighbourhood of radius $1$ of the union of $L_1$ and $L_2$ (in the standard euclidean metric).   

We shall now use this local model to construct a Liouville manifold by gluing cotangent bundles:  pick a Riemannian metric on $Q_i$ whose restriction to a neighbourhood of $b_i$ is isometric to the ball of radius $4$ in $\bR^{n}$ with $b_i$ mapping to the origin and $U_i$ to the ball of radius $2$.   We obtain a fixed symplectomorphism between  an open subset of $D^* Q_i$ and those points in $D^*L_i$ consisting of cotangent vectors lying over a point in $L_i$ of euclidean norm bounded by $4$.   By identifying the points of $D^* Q_1$ and $D^* Q_2$ whose images agree in $D^{*} L_1 \# D^* L_2$ we obtain a symplectic manifold
\begin{equation}
M = D^* Q_1 \#_{(b_1,b_2)} D^* Q_2,
\end{equation}
which is diffeomorphic to the classical plumbing construction in differential topology \cite{milnor}.  We shall write $b$ for the image of $b_1$ and $b_2$ in $M$, and $M_{b}$ for the fixed neighbourhood of $b$ in $M$ along which the gluing is performed and $M_{1}$ and $M_{2}$ for the two components of the complement of $M_{b}$ in $M$ (these are symplectomorphic to the unit disc (cotangent) bundles of the manifolds obtained by removing balls from $Q_1$ and $Q_2$).

Since the boundary of $M$ has corners, it is convenient to define a smooth submanifold $M^{\ins} \subset M$ as follows:  Choose a convex smooth function $\chi \co [0,+\infty)^{2} \to [0,+\infty)$ which is a small $C^0$ perturbation of the maximum of the coordinates, and which agrees with this maximum away from a neighbourhood of the diagonal.  Writing the squared euclidean norms as $\rho(\vx,\vy) =( |\vx|^2,  |\vy|^{2} ) $ we consider
\begin{equation} \label{eq:inside_manifold_near_0}
  M_{0}^{\ins} = (\chi \circ \rho)^{-1}(\epsilon) \subset  D^{*} L_1 \# D^* L_2;
\end{equation}
\begin{figure}
  \centering
\epsfxsize=2in
\epsffile{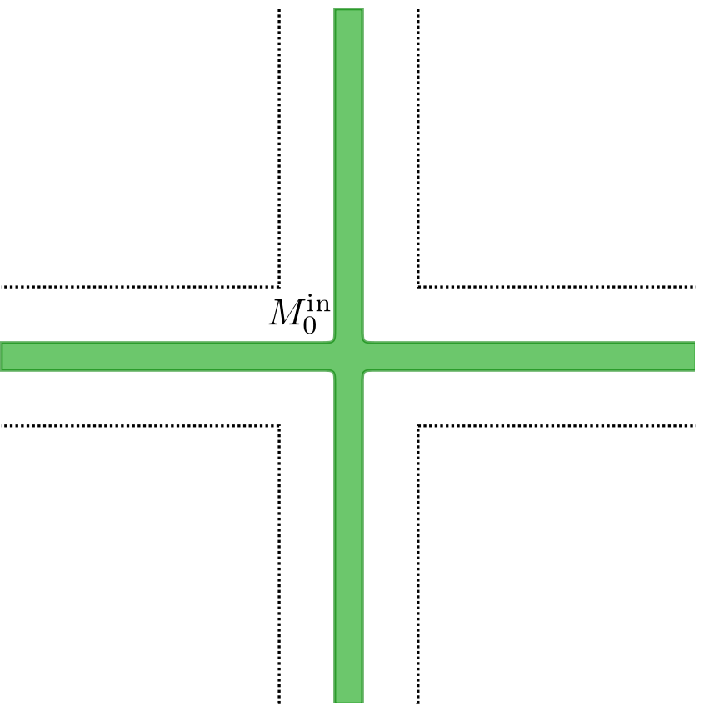} 
  \caption{ }
  \label{fig:inner_manifold}
\end{figure}
an example for $n=1$ is shown in Figure \ref{fig:inner_manifold}.  The smoothing parameter for $\rho$ should be sufficiently small that the following properties hold:
\vspace{-10pt}
\begin{enumerate}
\item Every point in $M_{0}^{\ins}$ lies within $2 \epsilon$ of a point in $L_1 \cup L_2$.
\item If $ |\vx| \geq 2\epsilon$, then $(\vx, \vy) \in M_{0}^{\ins} $  if and only if $|\vy|  \leq \epsilon$ (and vice versa with $x$ replaced by $y$).
\end{enumerate}
In words the second condition says that, away from the disc of radius $2 \epsilon$ in $L_j$, $ M_{0}^{\ins} $ agrees with the subset of the cotangent bundle consisting of vectors whose length is bounded by $\epsilon$.  Writing $M_{i}^{\ins}$ for the set of vectors in $T^*Q_i$ of length bounded by $\epsilon$ which project to points in $Q_i$ of distance greater than $2 \epsilon$ from $b_i$, we define a smooth manifold with boundary
\begin{equation*}
  M^{\ins} = M_{1}^{\ins} \cup M_{b}^{\ins} \cup M_{2}^{\ins} \subset M .
\end{equation*}


The symplectic form on $M$ can be written as the differential of a primitive $\theta$ for which the embedding of each $Q_i$ is exact (i.e., the restriction of $\theta$ to $Q_i$ is the differential of a function).   Moreover, the tangent space of $M$ is equipped with a homotopy class of complex volume forms (with respect to any almost complex structure compatible with the symplectic form) whose restriction to the cotangent bundle of $Q_i$ is homotopic to a form obtained by complexifying a volume form on $Q_i$.  In particular, $Q_i$ admits a \emph{grading} so that the intersection point between $Q_1$ and $Q_2$ is assigned a degree; we omit the choice of grading from our notation, but note that it may be normalised so that
\begin{equation} \label{eq:normalisation_floer_complex}
 \parbox{36em}{$CF^{*}(Q_1,Q_2)$ is supported in degree $0$.}
\end{equation}

\begin{defin}
The category $\Fuk(Q_1,Q_2)$ is the full subcategory of the Fukaya category of $M$ whose objects are $Q_1$ and $Q_2$.
\end{defin}
The existence of an $A_{\infty}$ category with exact Lagrangians as objects follows from \cite{seidel-book}.  The construction depends on certain choices of almost complex structures and Hamiltonian perturbations, but the category is independent, up to $A_{\infty}$ equivalence, of these choices.  In order to fix these choices, let us consider a second subdomain $M^{\mid}$ of $M$ containing $M^{\ins}$, whose construction we shall specify in Section \ref{sec:comp-curv-leaves} (we shall only use it to prove compactness for a certain moduli space of holomorphic discs entering in the construction of the functor from Floer to Morse theory). We shall now pick Floer data from the following spaces:
\begin{equation} \label{eq:allowed_J-H}
\parbox{36em}{Let $\sJ$ denote the space of almost complex structures which are of contact type near $\partial M^{\ins}$ and $\partial M^{\mid}$  and  let $\sH = \C_{c}^{\infty}(M^{\ins}, \bR)$ denote the space of smooth functions on $M$ whose support is contained in $M^{\ins}$.}
\end{equation}
In the next section we briefly recall how to use perturbations of the $\dbar$ equation coming from families valued in $\sJ$ and $\sH$ in order to construct $\Fuk(Q_1,Q_2)$.

\subsubsection{Review of the construction of the Fukaya category} \label{app:review_seidel}
We follow the discussion in \cite{seidel-book}, keeping our notation as close as possible.  First, we must define the morphism spaces (i.e., Floer complexes) 
\begin{defin} \label{lem:floer_data}
A {\bf Floer datum} for each pair $(i,j)$, is the choice of (i) a family $J^{i,j} = \{ J^{i,j}_{t} \in \sJ \}_{t \in [0,1]}$ of almost complex structures and (ii) a function $H^{i,j} \in \sH$.
\end{defin}

This choice determines a perturbed $\dbar$ equation on the strip
\begin{equation} \label{eq:perturbed_dbar_strip}
\dbar^{i,j}  u = ( d u - Y^{i,j}\otimes dt)^{0,1}
\end{equation}
where $Y^{i,j}$ is the Hamiltonian flow of $H^{i,j}$, and the $(0,1)$ part is taken with respect to the $t$-dependent almost complex structure $J^{i,j}_{t}$.  Given a pair $(p_0,p_1)$ of time-$1$ chords  the Hamiltonian flow of $H^{i,j}$ with endpoints on  $Q_j$ and  $Q_i$,  we write  
\begin{multline}
\cR(p_0, p_1) = \{ u \co [0,1] \times \bR \to M |\dbar^{i,j} u = 0 , \, u(0,s) \in Q_i, \, u(1,s) \in Q_j , \\
\lim_{s \to + \infty} u(s,t) = p_{0} , \, \lim_{s \to - \infty} u(s,t)= p_{1} \} / \bR
\end{multline}
where the $\bR$ action comes from translation on the source.  Using the fact that the complex structures which are allowed are of contact type near $\partial M^{\ins}$, one can use Lemma 7.2 of \cite{abouzaid-seidel} (see also Lemma \ref{lem:convexity_weak_assumptions} bellow) to prove the analogue of Gromov compactness and conclude:
\begin{lem} \label{lem:strips_moduli_spaces_cpact}
If the image of  $Q_i$ under the time-$1$ Hamiltonian flow generated by $H^{i,j}$ is transverse to $Q_j$, then a generic choice of family $J^{i,j}_t$ ensures that all moduli spaces of gradient trajectories $ \cR(p_0, p_1)$ for $p_0 \neq p_1$ are regular, and hence of expected dimension
\begin{equation}
 \deg(p_0) - \deg(p_1) -1.
\end{equation} 
Moreover, the Gromov bordification  $\Rbar(p_0, p_1)$ is compact, and whenever $\deg(p_0) - \deg(p_1) = 2$, it is a manifold with boundary
\begin{equation}
 \partial \Rbar(p_0, p_1) = \coprod_{\deg(p) = \deg(p_1) + 1}  \cR(p_0, p) \times \cR(p,p_1).
\end{equation} \noproof
\end{lem}
The Floer chain complex
\begin{equation}
 CF^{*}(Q_i, Q_j)
\end{equation}
is a direct sum of free abelian groups $| \ro_{p} |$  associated to each time-$1$ chord $p$ from $Q_i$ to $Q_j$ for the Hamiltonian $H^{i,j}$.  As the notation suggests, $| \ro_{p} |$ is the orientation line associated to a rank $1$ space $\ro_{p}$, whose construction we shall not recall.  The key fact we use is the existence of a canonical isomorphism
\begin{equation}  \label{eq:index_iso_strip} \lambda(\cR(p_0, p_1)) \cong \ro_{p_0} \otimes \ro_{p_1}^{\vee}\end{equation}

In general, the degree $\deg(p)$ of a generator of $CF^{*}(Q_i, Q_j)$ is determined by an additional choice of grading on $Q_i$ and $Q_j$.  This choice has been fixed (up to a global shift) by requiring that the rank $1$ group $CF^{*}(Q_1, Q_2)$ is supported in degree $0$.  Whenever $\deg(p_0) - \deg(p_1)= 1$,   $\lambda(\cR(p_0, p_1)) $ is canonically trivialised by the translative action of $\bR$, so  Equation \eqref{eq:index_iso_strip}  naturally assigns to each element $u$ of  $\cR(p_0,p_1)/ \bR$ a homomorphism
\begin{equation*} \mu_{u}^{F} \co | \ro_{p_1}|  \to | \ro_{p_0} |. \end{equation*}
The sum of all such homomorphisms is the matrix coefficient of $ | \ro_{p_1}| $ and $  | \ro_{p_0} |$ in the differential in the Floer complex.  Writing $[p_1]$ for a generator of $| \ro_{p_1}| $, we obtain:
\begin{equation}
 \mu_{1}^{\F}([ p_1 ]) = \sum_{u \in \cR(p_0,p_1)/ \bR} (-1)^{\deg(p_1)}  \mu_{u}^{F} ([ p_1 ]) 
\end{equation}

To define the $A_{\infty}$ structure, we must choose perturbation data for moduli spaces of discs with an arbitrary number of marked point.  We write $\cR_{d}$ for the moduli space of holomorphic discs with $d+1$ marked points, $d$ of which are marked as incoming, and, given a sequence $\vI = (i_0, \ldots, i_d)$, $\cR_{\vI}$ for a copy of the moduli space $\cR_{d}$  in which each incoming arc is labelled by an elements of the set $\{ 1, 2 \}$; in particular, the incoming marked points are labelled by pairs $(i_k, i_{k+1})$, and the outgoing marked point by $(i_0, i_d)$.  This space admits a compactification $\Rbar_{\vI}$ which is a manifold with corners whose points are in bijective correspondence to $\Stasheffbar_{\vI}$, and in the setup of Equation \eqref{eq:compose_stasheff_labels}, is equipped with maps
\begin{equation} \label{eq:boundary_strata_moduli_discs}
 \Rbar_{\vI[1]} \times  \Rbar_{\vI[2]} \to  \Rbar_{\vI}
\end{equation}
which cover its boundary.  In addition, we choose strip-like ends $\{ \epsilon_{i_k} \}_{k=0}^{d}$ near the marked points of all surfaces $\Sigma$ in $\Rbar_{\vI}$; this choice may be done globally in a smooth family over the moduli space. 
\begin{defin}
A {\bf perturbation datum} on $\Sigma$ is a choice of an $\sH$ valued $1$-form $K^{\vI}$ on $\Sigma$ and a function $J^{\vI} \co \Sigma \to \sJ$.  The pull-backs of these data under a strip-like end corresponding to a marked point labelled $(i,j)$ must agree with the Floer data for the pair $(i,j)$.

A {\bf universal perturbation datum for discs} $(\bfJ^{\cR}, \bfK^{\cR})$ is a choice of perturbation data for all elements of $\Rbar_{\vI}$, varying smoothly over the moduli space, and which is consistent; i.e. compatible with the maps \eqref{eq:boundary_strata_moduli_discs}.
\end{defin}

The choice of a perturbation datum on $\Sigma$ defines a perturbed  $\dbar$ equation analogous to \eqref{eq:perturbed_dbar_strip}: the Hamiltonian vector field associated to an element of $\sH$ and the choice of $K^{\vI}$ define a vector-field valued $1$-form $Y^{\vI}$ on every surface, so we consider the operator
\begin{equation} \dbar_{\Sigma}^{\vI} (u) = \left( du - Y^{\vI} \right)^{(0,1)}  \end{equation}
where the $(0,1)$ part is taken with respect to the almost complex structure $J^{\vI}$ on the target (and of course the underlying complex structure on $\Sigma$ on the source).

 In particular, given a universal perturbation datum, and a sequence $(p_0,\vp) = (p_0, p_1, \ldots, p_d)$ with $p_0$ a chord from $Q_{i_0}$ to $Q_{i_d}$, and $p_{k}$ for $1 \leq k \leq d$ chords from $Q_{i_{k-1}}$ to $Q_{i_{k}}$ (for the appropriate Hamiltonians), we define 
\begin{equation}
\cR(p_0, \vp) = \coprod_{\Sigma \in \cR_{\vI}} \{ u \co \Sigma \to M | u (\partial \Sigma) \subset Q_{1} \cup Q_2,  \lim_{s \to \pm \infty} u(\epsilon_{i_k}(s,t)) = p_{i_k} , \dbar_{\Sigma}^{\vI} u =0\}
\end{equation}
where the condition on the boundary is more precisely stated as follows:  the image of a segment in $\partial \Sigma$ labelled by $i_k$ is contained in $Q_{i_k}$.

The following result  generalises Lemma  \ref{lem:strips_moduli_spaces_cpact} to discs with multiple inputs, and follows from the results proved in \cite{seidel-book}:
\begin{prop}
For a generic choice of perturbation data, all moduli spaces $\cR(p_0, \vp)$ are regular, and hence have the expected dimension
\begin{equation}
 d-2 + \deg(p_0) - \sum_{1 \leq k \leq d} \deg(p_k).
\end{equation}
Their Gromov bordifications $\Rbar(p_0, \vp)$ are compact, and those which have expected dimension $1$  are manifolds with boundary. \noproof
\end{prop}

The gluing theorem for elliptic operators on the disc gives an isomorphism
\begin{equation} \lambda(\cR(p_0, \vp)) \cong \lambda(\cR_{d}) \otimes \ro_{p_0} \otimes \ro_{p_1}^{\vee} \otimes \ldots \otimes \ro_{p_d}^{\vee} ,\end{equation}
where $\lambda$ is the top exterior power, and hence every rigid holomorphic disc $u$ determines an isomorphism
\begin{equation} \ro_{p_1}\otimes \ldots \otimes \ro_{p_d} \to  \ro_{p_0} . \end{equation}
 Writing $\mu_{u}$ for the induced map on the orientation lines, we use Seidel's conventions, and define the $d$\th higher product
\begin{equation*}  
\mu^{\F}_{d}  \co CF^{*}(Q_{i_{d-1}},Q_{i_d}) \otimes \cdots \otimes CF^{*}(Q_{i_1} ,Q_{i_2}) \to CF^{*}(Q_{i_{1}}, Q_{i_d}) 
\end{equation*}
as a sum of the contributions of all holomorphic discs
\begin{equation*}
[p_d] \otimes \cdots \otimes [p_{1}] \mapsto (-1)^{\dagger( \vp)} \sum_{\stackrel{ p_0}{u \in \cR(p_0, \vp)}} \mu_{u}  ( [p_d] \otimes \cdots \otimes [p_{1}])  \end{equation*}
where the sign is given by
\begin{equation}  \label{eq:sign_twisting_fukaya} \dagger( \vp) = \sum_{k=1}^{d} k \deg(p_{k}).  \end{equation}

\section{From simplicial to Morse cochains} \label{sec:simp_to_morse}
In this section, we construct an $A_{\infty}$ equivalence
\begin{equation} \label{eq:functor_simp_morse}
\cF \co \Simp(\sQ_1,\sQ_2) \to \Morse(Q_1,Q_2) 
\end{equation}
The main ingredients in the construction of such an equivalence is the interpretation of the cup product as an appropriate (perturbed) intersection product discussed in Section \ref{sec:simp}, and a moduli space of shrubs $\Shrub_{d}$  introduced in \cite{abouzaid}; it seems that our spaces are combinatorially equivalent to those later introduced by Forcey \cite{forcey} under the name composihedra.  These moduli spaces are quotients of multiplihedra, and control $A_{\infty}$ functors from differential graded algebras to $A_{\infty}$ algebras.  From now on, we shall assume
\begin{equation} \label{cond:morse_simp_compatible}
\parbox{36em}{The cells of $\check{\sQ_{i}}$ intersect the boundary of $\bar{U}_i$ transversely, and there are nested weak homotopy equivalences $\check{\Delta}_{i} \subset U_i \subset \Delta_{i}$.  Moreover, the identification $\Delta_{1} \cong \Delta_{2}$ restricts on $U_i$ to the diffeomorphism of Equation \eqref{eq:diffeo_open_sets_Q_i,j}.}
\end{equation}
In particular for $i \neq j$, we obtain a cellular subdivision $\check{\sQ}_{i,j}$ of $Q_{i,j}$, which is transverse to the boundary, by pulling back the dual subdivision to $\sQ_{i}$, and the diffeomorphism $Q_{1,2} \cong Q_{2,1}$ respects this cellular subdivision.

We shall briefly review the construction of $\Shrub_{d}$, and focus on the choices of perturbation data which are necessary to bypass the transversality problems that would arise if we use only gradient flow lines as in \cite{abouzaid}.

\begin{defin}
The moduli space of {\bf shrubs} $\Shrub_{d}$ is the space of isomorphism classes of metric ribbon trees $(S,g_S)$ with one infinite outgoing edge, and $d$  finite incoming edges whose endpoints are equidistant to the outgoing edge.
\end{defin}

The moduli space of shrubs admits a natural compactification $\Shrubbar_{d}$ by allowing finite edges to become infinite, or incoming edges to shrink to $0$ length.  The property of equidistance required of the outgoing leaves implies that whenever a sequence in $\Shrub_{d}$ converges to a point on the boundary where the length of some finite edge becomes infinite, every path from the outgoing leaf to an incoming one must contain such an edge.  We can reconstruct such limit points as follows: whenever $\sum_{k=1}^{r} d_k = d$ we obtain a singular shrub by attaching the outputs of $r$ shrubs (the $k$\th shrub having $d_k$ inputs) to a Stasheff tree with $r$ inputs.  This construction defines a map
\begin{equation} \label{eq:boundary_shrub_tree}
\Stasheffbar_{r} \times \Shrubbar_{d_1} \times \ldots \times \Shrubbar_{d_r} \to \Shrubbar_{d}.
\end{equation}
Similarly, whenever the length of some incoming edge shrinks to $0$, there must be at least one other (adjacent) edge which is also collapsed, so that we have for each integer $k$ between $1$ and $d-1$ a map
\begin{equation} \label{eq:boundary_shrub_lenght_0}
\vee_{k} \co \Shrubbar_{d-1} \to \Shrubbar_{d}
\end{equation}
obtained by grafting a pair of edges of length $0$ at the $k$\th external vertex.  The following result appears in Appendix B of \cite{abouzaid}:
\begin{lem}
The moduli space $\Shrubbar_{d}$ is a compact manifold with boundary whose boundary is covered by the images of the maps \eqref{eq:boundary_shrub_tree} and \eqref{eq:boundary_shrub_lenght_0}.
\end{lem}

As in the previous section,  a sequence $\vI$ induces a unique labelling, compatible with the cyclic order, of all edges by pairs $(i(e), j(e))$, such that the incoming edges are labelled by $(i_k, i_{k+1})$.  Moreover, the maps \eqref{eq:boundary_shrub_lenght_0} and \eqref{eq:boundary_shrub_tree} have their analogues for labelled shrubs
\begin{align} \label{eq:boundary_shrub_labelled-tree}
\Stasheffbar_{\vR}  \times \Shrubbar_{\vI[1]} \times \ldots \times \Shrubbar_{\vI[r]}   & \to \Shrubbar_{\vI} \\
\label{eq:boundary_shrub_labelled-length-0}
 \Shrubbar_{\vI - \{i_{k} \}} & \to \Shrubbar_{\vI},
\end{align}
where the labels on the output of a shrub in $ \Shrubbar_{\vI[k]}$ agree with the  labels of the $k$\th incoming leaf of a Stasheff tree in $\Stasheffbar_{\vR}$.  A choice of $\vI$ determines a gradient flow equation on every edge of a shrub $(S,g_S)$, which we perturb as follows:
\begin{defin}
A {\bf universal consistent perturbation datum for shrubs} is a choice  $\bfX^{\Shrub}$ of a smooth family of perturbation data on $\coprod_{\vI}\Shrub_{\vI}$ which are compatible with the maps \eqref{eq:boundary_shrub_labelled-tree} and \eqref{eq:boundary_shrub_labelled-length-0}.
\end{defin}

The discussion of smoothness is relegated to Section \ref{app:stasheff}.  Recall that the definition of a perturbation datum requires that the resulting perturbed negative gradient flow be outward (respectively inward) pointing on the boundary of $Q_{1,2}$ (respectively $Q_{2,1}$).  The choice of dual cell subdivision gives another condition
\begin{defin}
A universal perturbation datum is {\bf compatible with $\check{\sQ}$} if, whenever $e_k$ and $e_{k+1}$ are successive incoming leaves of length $0$ in a shrub $(S,g_S)$, the vector field
\begin{equation} \label{eq:simplicial_condition_perturbation}
- \nabla f_{i_{k},i_{k+1}} + X_{e_{k+1}} - \left( - \nabla f_{i_{k-1},i_{k}} + X_{e_{k}} \right)
\end{equation}
defined on the intersection of the images of $Q_{i_{k-1},i_{k}}$ and $Q_{i_{k},i_{k+1}}$ in $Q$ is simplicial.
\end{defin}

By Condition \eqref{cond:morse_simp_compatible} the behaviour of the (perturbed) gradient flow near the boundary is compatible with a requirement that the vector field be simplicial.  Indeed, the restriction of the tangent space of $Q_{i,j}$ to every cell $\check{\tau}$ of $\check{\sQ}_{i,j}$ contains an open cone $C_{\tau}$ consisting of vectors pointing in the direction of the top dimensional cell adjacent to $\check{\tau}$ which is minimal with respect to the ordering; the simplicial condition is the requirement that the restriction of a vector field to $\check{\tau}$ lie in $C_{\tau}$.  As the boundary of $C_{\tau}$  is covered by the tangent spaces of the codimension $1$ faces meeting at $\check{\tau}$,  $T \check{\tau}$ is included in its closure.  Since $\partial Q_{i,j}$ intersects $\check{\tau}$ transversely, we conclude that the tangent space  of the boundary separates $C_{\tau}$ into two non-empty open components.  Assuming that $ f_{i_{k},i_{k+1}}$, $f_{i_{k-1},i_{k}}$ and $X_{e_{k}}$ have been chosen already, we can always choose $X_{e_{k+1}}$ deep inside the appropriate cone so that \eqref{eq:simplicial_condition_perturbation} is a simplicial vector field.
\begin{lem}
There is a non-empty open subset of the space of universal consistent perturbation data for shrubs which consists of data that are compatible with $\check{\sQ}$.
\end{lem}
\begin{proof}
Since universal perturbation data are constructed inductively, it suffices to see that the compatibility condition on $\Shrubbar_{\vI}$ does not restrict the choice of perturbation data for $\Shrubbar_{\vI[']}$ whenever $\vI[']$ is a subset of $\vI$.  This is essentially obvious since our perturbation data on each edge are chosen independently of each other, and the compatibility condition restricts edges in $\Shrubbar_{\vI}$ which do not exist in $\Shrubbar_{\vI[']}$.
\end{proof}
From now on when we speak of a universal perturbation datum for shrubs, we shall assume that it is compatible with $\check{\sQ}$.  With the choice of such a datum, we can now define the moduli spaces which shall be used to construct the functor of Equation \eqref{eq:functor_simp_morse}.  Let us write $\vsig = (\check{\sigma}_1, \ldots, \check{\sigma}_{d})$ for a sequence of cells $\check{\sigma}_{k} \in \check{\sQ}_{i_k,i_{k+1}}$, and $x_0$ for a critical point of $f_{i_0,i_d}$:
\begin{defin} \label{defin:shrub_map}
The moduli space $\Shrub(x_0,\vsig)$ of (perturbed) gradient shrubs with inputs $\vsig$ and output $x_0$ is the set of maps
\begin{equation}
 \psi \co (S,g_S) \to Q
\end{equation}
such that (i) each edge $e$ of $S$ lifts as a perturbed gradient flow segment to the labelling manifold $Q_{i(e),j(e)}$, (ii) the lift of the outgoing leaf is the critical point $x$ and (iii) the lift of the $k$\th incoming leaf lies on $\check{\sigma}_k$.
\end{defin}
The moduli space $\Shrub(x_0,\vsig)$ admits a natural stratification coming from the stratification of $\Shrub_{d}$ by the topological type of the tree, and a bordification $\Shrubbar(x,\vsig)$ by allowing edges to become infinite.   In particular, we have a subset 
\begin{equation}
 \Shrubbar_{\vee_{k} }(x_0,\vsig) \subset \Shrubbar(x_0,\vsig)
\end{equation}
consisting of maps whose domain lies in the image of the inclusion of Equation \eqref{eq:boundary_shrub_labelled-length-0}; i.e. singular shrubs whose $k$ and $k+1$\st incoming edges both have length $0$.  

\begin{lem} \label{lem:simplicial_perturbation_shrubs_gives_cup}
If the cup product of $\check{\sigma}_{k}$ and $\check{\sigma}_{k+1}$ vanishes, then there is a neighbourhood of the image of $\vee_k$ in $\inte\left(\Shrubbar_{\vI} \right)$ whose inverse image under the forgetful map
\begin{equation}
 \Shrubbar(x_0,\vsig) \to \Shrubbar_{\vI}
\end{equation}
is empty.  In particular, this stratum of $\Shrubbar(x_0,\vsig)$ does not lie  in the closure of the top dimensional part of the moduli space $\Shrub(x_0,\vsig) $.
\end{lem}
\begin{proof}
The proof is completely local.  A slight generalisation of Lemma \ref{lem:cup_is_intersection}, implies that the condition that the cup product of $\check{\sigma}_{k}$ and $\check{\sigma}_{k+1}$ vanish is equivalent to the absence of intersection points between the images of $\check{\sigma}_{k}$ under a flow $\phi_{k}$ and $\check{\sigma}_{k+1}$ under a flow $\phi_{k+1}$, whenever $\phi_{k}$ and $\phi_{k+1}$ are generated by sufficiently small vector fields $X_k$ and $X_{k+1}$ whose difference is compatible with the simplicial triangulation.  If $(S,g_{S})$ is a singular shrub with sufficiently short incoming edges $e_{k}$ and $e_{k+1}$, the flows obtained by integrating the perturbed gradient vector field along these edges satisfy this property due to our assumption that the perturbation datum $X^{\Shrub}$ is smooth and compatible with the subdivision $\check{\sQ}$.  The desired result follows immediately.
\end{proof}

We also need a compactness result analogous to Lemma \ref{lem:compactness_stasheff}
\begin{lem} \label{lem:compactness_shrubs}
Given $x_0$ a critical point of $f_{0,d}$, and  $(\check{\sigma}_1, \ldots, \check{\sigma}_d)$ cells of the subdivisions $\sQ_{i_k, i_{k+1}}$, the moduli space $\Shrubbar(x_0, \vsig)$ is compact.
\end{lem}
\begin{proof} 
  Again, it suffices to prove that no edge labelled by $Q_{1,2}$ or $Q_{2,1}$ may escape to the boundary, which can be done in three case, of which we explain only one.   We fix a neighbourhood of the boundary of $Q_{2,1}$ which does not intersect the cell dual to the interior of $\Delta$.   Assuming that the image of an edge $e$ labelled by $Q_{2,1}$ intersects this neighbourhood, we follow an ascending arc all of whose edges are labelled by $Q_{2,1}$, and derive the contradictory conclusion that an input cells labelled by $(1,2)$ intersects a neighbourhood of the boundary of $Q_{2,1}$.
\end{proof}

In Section \ref{app:stasheff}, we explain how to prove:
\begin{lem} \label{lem:moduli_shrubs_nice_manifold}
$\Shrubbar(x,\vsig)$ is compact, and for generic universal perturbation data, all strata are smooth manifolds of the expected dimension.    Moreover, the closure of the top-dimensional strata is a compact manifold with boundary  of dimension
\begin{equation}
d -1 + \deg(x_0) - \sum_{k} \deg( \check{\sigma}_{k})
\end{equation}
which we denote $\Shrubhat(x_0,\vsig)$.
\end{lem}

We now restrict attention to the case where the moduli space of shrubs is rigid, i.e. to a sequence of cells $\vsig$ and a critical point $x_0$ such that
\begin{equation}
 \deg(x_0) = 1- d + \sum_{k} \deg( \check{\sigma}_{k})
\end{equation}
The description of the moduli space of shrubs as a fibred product (see Section \ref{sec:orienting_maps}) gives an isomorphism
\begin{equation} \label{eq:isomorphism_orientation_bundles_shrubs} \lambda(\Shrub_{d}(x_0, \vsig)) \otimes \lambda(Q^{d+1}) \cong \lambda(Q) \otimes \lambda(\Shrub_{d}) \otimes \lambda(W^{s}(x_0)) \otimes \lambda( \vsig) ,\end{equation}
where the last factor on the right hand side is a tensor product
\begin{equation}  \lambda( \vsig) \cong \lambda(\check{\sigma}_{1}) \otimes \cdots \otimes \lambda(\check{\sigma}_{d}) .\end{equation}
Using the isomorphism
\begin{equation} \label{eq:factor_tangent_space} \lambda(Q) \cong \lambda(\sigma_k) \otimes \lambda(\check{\sigma}_k),  \end{equation}
we conclude that a rigid shrub $\psi \in \Shrub_{d}(x_0, \vsig)$  determines an isomorphism
\begin{equation}  \lambda(\sigma_{1}) \otimes \cdots \otimes \lambda ({\sigma_{d}} )\to \ro_{x_0} . \end{equation}
Writing $\cF^{\psi}$ for the induced map on the orientation line, the $d$\th order map in our functor from the simplicial to the Morse category is defined via the formula
\begin{align} \label{eq:functor_simp_morse_formula}
 \cF^{d} \co C^{*}(\sQ_{i_{d-1}, i_{d}}) \otimes \cdots \otimes C^{*}(\sQ_{i_{0}, i_{1}}) & \to CM^{*}(f_{0,d}) \\
\check{\sigma}_{d} \otimes \ldots \otimes \check{\sigma}_{1} & \mapsto (-1)^{(n+1)\dagger(\vsig)} \sum_{\stackrel{x_0}{\psi \in \Shrub_{d}(x_0, \vsig)}} \cF^{\psi}(\vsig),
\end{align}
where the sign $\dagger(\vsig)$ is determined by a choice of orientation on the moduli of shrubs.  With our sign conventions, we would find
\begin{equation}  \label{eq:twist_signs_shrubs} \dagger(\vsig) = \sum_{k=1}^{d} k  \deg(\check{\sigma}_{k}).   \end{equation}

To illustrate the case of one input, let us fix a simplex $\sigma_1$ of $\sQ_{i,j}$, and a critical point $x_0$ of $f_{i,j}$ such that $\check{\sigma}_1$ and $W^{s}(x_0)$ intersect transversely.  At every such intersection point $\psi$, we have an isomorphism
\begin{equation}  \label{eq:isomorphism_rigid_shrub} \lambda(Q) \cong \lambda(W^{s}(x_0)) \otimes \lambda(\check{\sigma}_1), \end{equation}
which, upon using the decomposition \eqref{eq:factor_tangent_space} yields an isomorphism
\begin{equation} \lambda(\sigma_1) \cong \lambda(W^{s}(x_0)) .\end{equation}
The sign from Equation \eqref{eq:twist_signs_shrubs} is equal to
 \begin{equation} (-1)^{(n+1)\deg(\check{\sigma})} \end{equation}
in this case,  which determines the contribution of $\psi$ to the map 
\begin{equation} \label{eq:sign_F1}  \cF^{1} \co C^{*}( \sQ_{i,j}) \to CM^{*}(f_{i,j}). \end{equation}

To prove that this is a chain map, we must consider, as usual, $1$-dimensional moduli spaces of shrubs with one input.  This occurs whenever we are given a critical point $x_0$ such that
\begin{equation} \deg(x_0) = \deg(\check{\sigma}_1) +1.  \end{equation}
Since the space $\Shrubbar(x_0, \check{\sigma}_1)$ is the closure of the intersection of $\check{\sigma}_1$ with $W^{s}(x_0)$, its boundary is easily seen to consist of two strata: either (i) a family of gradient trajectories from $x_0$ escapes to a cell $\check{\sigma}_0 $ on the boundary of $\check{\sigma}_1$, or (ii) a family of such trajectories converges to a broken one, corresponding to the concatenation of an ascending gradient flow line from $x_0$ to a critical point $x_1$ of degree equal to $\deg(\check{\sigma}_1)$, followed by a gradient flow line from $x_1$ to $\check{\sigma}_1$.  These two configurations correspond to the two terms whose equality implies that $\cF^{1}$ is a chain map:
\begin{equation} \label{eq:chain_map_simplicial_morse} \cF^{1} \circ  \mu_1^{\S}  = \mu_{1}^{\M} \circ \cF^{1}.  \end{equation}
To prove the correctness of the signs, we note that we have a natural isomorphism
\begin{equation} \label{eq:isomorphism_1_dim_moduli_shrub} \lambda( \Shrub(x_0,\check{\sigma}_1)) \otimes \lambda(Q) \cong \lambda(W^{s}(x_0)) \otimes \lambda(\check{\sigma}_1)  \end{equation}
which we must compare with the isomorphisms \eqref{eq:isomorphism_rigid_shrub} and \eqref{eq:isomorphism_rigid_line} when we pass to the compactification $\Shrubbar(x_0, \check{\sigma}_1)$.

First, we obtain an isomorphism
\begin{equation} \lambda(\bR) \otimes \lambda(W^{s}(x_1)) \cong \lambda( W^{s}(x_0)) \end{equation}
from Equation \eqref{eq:isomorphism_rigid_line} and \eqref{eq:decom_tangent_space_crit_points}.  The result of taking the tensor product with the two sides of Equation \eqref{eq:isomorphism_rigid_shrub} gives:
\begin{equation}  \label{eq:analyse_shrub_breaking_trajectory} \lambda(\bR)  \otimes \lambda(W^{s}(x_1)) \otimes \lambda(Q) \cong \lambda(W^{s}(x_0)) \otimes  \lambda(W^{s}(x_1)) \otimes \lambda(\check{\sigma}_1) \end{equation}
The reader may easily check that the gluing theorem for gradient flow lines gives an isomorphism $\lambda(\bR) \cong \lambda(\Shrub(x_0,\check{\sigma}_1) )$ near a broken trajectory, with the positive direction in $\bR$ corresponding to the \emph{outward} pointing normal vector.  In order to pass from this to the isomorphism \eqref{eq:isomorphism_1_dim_moduli_shrub}, we must switch some factors taking care to record the appropriate Koszul signs:
\begin{enumerate}
\item Permute the first copy of $\lambda(W^{s}(x_1)) $ past $\lambda(\bR)$, introducing a parity change of  $\deg(\check{\sigma}_1)$, then
\item permute  the two factors of $\lambda(W^{s}(x_0)) \otimes  \lambda(W^{s}(x_1))$, introducing a sign equal to $\deg(x_0) \deg(x_1) $, which vanishes modulo $2$, then
\item cancel the $ \lambda(W^{s}(x_1))$ factors on either side.
\end{enumerate}
The result of the operations above is exactly the isomorphism of Equation \eqref{eq:isomorphism_1_dim_moduli_shrub}, with a change in signs given by the parity of $\deg(\check{\sigma}_1) $.

We must now analyse the same problem near the other type of boundary strata.  Near such a stratum, we have a natural isomorphism between  $\lambda( \Shrub(x_0,\check{\sigma}_1))$ and the normal vector of the boundary to $\check{\sigma}_1$.  In particular, Equation \eqref{eq:isomorphism_1_dim_moduli_shrub} gives an isomorphism
\begin{equation}  \lambda( \Shrub(x_0,\check{\sigma}_1)) \otimes \lambda(Q) \cong \lambda(W^{s}(x_0)) \otimes  \lambda( \Shrub(x_0,\check{\sigma}_1)) \otimes  \lambda( \partial \check{ \sigma}_1)  \end{equation}
which, after introducing a sign with parity $\deg(x_0)= \deg(\check{\sigma}_1) +1$, gives exactly the isomorphism of \eqref{eq:isomorphism_rigid_shrub} for the boundary of $\sigma$.   Note that this would naively seem to imply that $\cF^{1}$ satisfies the desired equation, since the two types of boundary strata have \emph{opposite relative orientations.}

Unfortunately, there are additional signs coming from (i) the signs in formulas \eqref{eq:formula_differential_morse} and  \eqref{eq:twist_diff_dg_to_A_infty} for the differentials in the Morse and simplicial categories and  (ii) the sign in the definition of the map $\cF^{1}$.  Remark \ref{rem:sign_difference} explains that the sign introduced  in the definition of $\Simp(\sQ_1, \sQ_2)$ contributes a $-1$ to the total formula because it corrects for the fact that duals cells do not have compatible orientations on their boundaries.  On the other hand, both the Morse differential and  $\cF^{1}$ contribute $(-1)^{n}$ and $(-1)^{n+1}$ respectively, so the artificial signs we have introduced cancel in this case.  The generalisation for higher products, gives:

\begin{prop} \label{prop:cup_product_agrees_perturbation}
The maps of Equation \eqref{eq:functor_simp_morse_formula} define an $A_{\infty}$ equivalence.  In particular, $\cF^{1}$ is a chain map inducing an isomorphism on cohomology, and
\begin{multline}
\sum_{1 \leq k \leq d} (-1)^{\maltese_{k}} \cF^{d}   \left(\id^{d-k-1} \otimes \mu_1^{\S} \otimes \id^{k} \right) + \sum_{1 \leq k \leq d-1} (-1)^{\maltese_{k}} \cF^{d-1}  \left(\id^{d-k-2} \otimes \mu_{2}^{\S} \otimes \id^{k} \right) \\
= \sum_{d_1+ \ldots +d_r =d} \mu_{r}^{\M}  \left( \cF^{d_1} \otimes \cdots \otimes \cF^{d_r} \right) \label{eq:simp_to_morse_A_infty_holds}
\end{multline}
where 
\begin{equation*} \maltese_{k}(\vp) = k + \sum_{j=1}^{k} \deg(p_j). \end{equation*}
\end{prop}
\begin{proof}
The claim that $\cF^{1}$ is a chain isomorphism is the familiar statement that the cohomology of smooth manifolds can be computed using Morse theory. To prove the $A_{\infty}$ equation \eqref{eq:simp_to_morse_A_infty_holds}, we study, as usual, moduli spaces $\Shrub(x,\vsig)$ of expected dimension $1$.  By Lemma \ref{lem:moduli_shrubs_nice_manifold} the closure of the interior of this moduli space is a $1$-dimensional manifold with boundary projecting to the interior or top dimensional strata of the boundary of $\Shrubbar_{d}$ (this follows from the fact that all strata of $\Shrub(x,\vsig)$ are of the expected dimension, hence that those lying over higher codimension strata of $\Shrubbar_{d}$ are empty).

It suffices to check that the boundary points of $\Shrubhat(x,\vsig)$ account for all the terms of \eqref{eq:simp_to_morse_A_infty_holds}.  The first sum in the left hand side of \eqref{eq:simp_to_morse_A_infty_holds} corresponds to boundary points of  $\Shrubhat(x,\vsig)$ which project to the interior of $\Shrubbar_{d}$.  The top dimensional boundary strata of $\Shrubbar_{d}$ are of two types: those which are the image of \eqref{eq:boundary_shrub_tree} correspond to the right hand side of \eqref{eq:simp_to_morse_A_infty_holds}, and we claim that the second term of the left hand side of that equation correspond to those points which project to the image of \eqref{eq:boundary_shrub_lenght_0} and are in the closure of  $\Shrub(x,\vsig)$.  Indeed, points in $\Shrubbar(x,\vsig)$ lying over the image of \eqref{eq:boundary_shrub_lenght_0} correspond to $\Shrub(x,\vsig[\cap k])$, where $\vsig[\cap k]$ is obtained from $\vsig$ by replacing the cells $\check{\sigma}_{k}$ and $\check{\sigma}_{k+1}$ by their \emph{intersection}. Lemma \ref{lem:simplicial_perturbation_shrubs_gives_cup} implies that whenever this intersection does not contribute to the cup product, the corresponding cell does not lie in the boundary of $\Shrubhat(x,\vsig)$, so, modulo signs, the Proposition follows from the next Lemma.  We briefly discuss signs and orientations in Section \ref{sec:orienting_maps}.
\end{proof}

\begin{lem}
If $\mu_{2}^{\C}(\check{\sigma}_{k},\check{\sigma}_{k+1}) \neq 0$, and the moduli space $\Shrub(x,\vsig[\cap k])$ is non-empty, the corresponding points of $\Shrubbar(x,\vsig)$ lie in the boundary of $\Shrubhat(x,\vsig)$.
\end{lem}
\begin{proof}
For the purpose of this argument we introduce a moduli space
\begin{equation}
 \Shrub_{d-1}^{k,\epsilon}
\end{equation}
consisting of ribbon trees satisfying all the conditions for being a shrub except that the distance from the outgoing edge to the incoming leaves numbered $k$ and $k+1$ are smaller by $\epsilon$ than the distances to the other incoming vertices.  Note that the regularity of $\Shrub(x,\vsig[\cap k])$ and the implicit function theorem imply that, whenever $\epsilon$ is sufficiently small, we have a bijective correspondence between its elements and those of $\Shrub^{k,\epsilon}(x,\vsig[\cap k]_{\phi})$, where $\vsig[\cap k]_{\phi}$ is obtained from $\vsig[\cap k]$ by replacing $\check{\sigma}_{k} \cap \check{\sigma}_{k+1}$ by its image under a diffeomorphism sufficiently close (in the $C^1$ norm) to the identity, and keeping all other cells unchanged. Using the fact that our chosen perturbations are compatible with the subdivision $\check{\sQ}$, we find that whenever $\mu_{2}^{\C}(\check{\sigma}_{k},\check{\sigma}_{k+1}) \neq 0$, there is a diffeomorphism $\phi_{k,k+1}^{S,g_S}$ such that
\begin{equation}
 \phi_{k,k+1}^{S,g_S}( \check{\sigma}_{k} \cap \check{\sigma}_{k+1}) = \phi_{k}^{S,g_S}( \check{\sigma}_{k}) \cap \phi_{k+1}^{S,g_S}( \check{\sigma}_{k+1})
\end{equation}
where $\phi^{S,g_S}_{k}$ and $\phi^{S,g_S}_{k+1}$ are the diffeomorphisms obtained by integrating the perturbed gradient flow along the respective incoming leaves.  Moreover, as $(S,g_S)$ converges to the relevant boundary stratum of $\Shrubbar_{d}$, the map $\phi_{k,k+1}^{S,g_S}$ may be chosen to converge to the identity.  Writing $\epsilon_{k}(S,g_S)$ for the length of the $k$\th incoming leaf, the claim follows by noticing that the elements of 
\begin{equation*} 
 \Shrub^{k,\epsilon_{k}(S,g_{S})}\left(x,\vsig[\cap k]_{\phi_{k,k+1}^{S,g_S}} \right)
\end{equation*}  are in bijective correspondence with the elements of the top dimensional stratum of $\Shrub_{(S,g_S)}(x,\vsig)$ which are sufficiently close to $\Shrub(x,\vsig[\cap k])$.
\end{proof}


\section{A convenient model for Multiplihedra} \label{sec:mushrooms}
Boardman and Vogt \cite{BV} introduced spaces of ``colored trees'' as a model for  multiplihedra; this collections of polyhedra controls $A_{\infty}$ functors.  Since our goal in later sections is to construct an $A_{\infty}$ functor whose source is the Fukaya category, it seems reasonable to replace one of the trees by holomorphic discs.  Keeping relatively close to our previous botanical terminology, we shall call these spaces moduli spaces of mushrooms.

\begin{defin}
A {\bf mushroom}  $C$ with $d$ inputs is defined by the following data 
\vspace{-10pt}
\begin{enumerate}
\item A partition $d = d_{1} + d_{2} + \cdots + d_{r}$ (in particular $d_{k} \geq 1$) and
\item for each integer $1 \leq k \leq r$, a {\bf cap}  $P_{k}(C)$ which is a disc with $d_{k} + 2$ marked boundary points of which $d_k$ successive ones are distinguished as incoming, and
\item a {\bf stem} $S(C)$, which is a a shrub with $k$ inputs.
\end{enumerate}
\end{defin}
\begin{rem} \label{caps_are_discs}
Sometimes, it shall be convenient to discuss caps separately.  By a $d$-cap, we will simply mean a disc with $d+2$ marked points of which $d$ successive ones are distinguished as incoming.  Alternatively, we can distinguish the pair of outgoing vertices, or the arc connecting them.  The {\bf moduli space of caps} $\Pil_{d}$  is simply the (open) Stasheff moduli space $\Stasheff_{d+1}$ under another guise.  Later, we shall find it convenient to fix the identification with the moduli space of discs $\cR_{d+1}$ for which the boundary marked point $v_{1}^{out}$  in Figure \ref{basic_mushrooms} is mapped to the (unique) outgoing marked point for discs in $\cR_{d+1}$.
\end{rem}
\begin{figure}[h] 

 \input{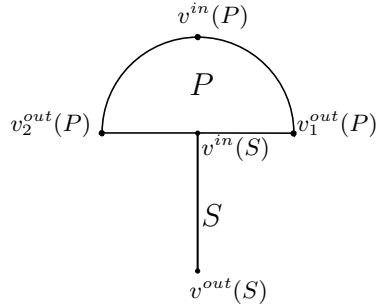}
   \caption{The unique mushroom with one input, with cap $P$ and stem $S$.}
   \label{basic_mushrooms}
\end{figure}
\begin{rem}
Even though we will draw mushrooms as in Figure \ref{basic_mushrooms}, we remind the reader that the incoming vertices of the stem are not associated to a marked point on the outgoing arc.
\end{rem}

An isomorphism  $C \to C'$ of mushrooms is an isomorphism of the underlying caps and stem.  Explicitly, such a map is a collection of $r+1$ isomorphisms:
\begin{align*}
P_{k}(C) \to P_{k}(C')  & \quad 1 \leq k \leq r\\
S(C) \to S(C') & 
\end{align*}

\begin{lem}
The only non-trivial automorphism of a mushroom is the identity. 
\end{lem}
\begin{proof}
Such an automorphism would correspond to an automorphism of the stem or of one of the caps.  However, we have excluded discs with less than $3$ marked points as caps.
\end{proof}

By definition, the set $\Champ$ of mushrooms with $d$ inputs admits a bijective map to
\begin{equation} \label{eq:moduli_mushrooms_disjoint_union}
   \coprod \Shrub_{r} \times \Pil_{d_1} \times \cdots \times \Pil_{d_r}
\end{equation}
which we later use to define a topology.  Recall that whereas the moduli space $\Stasheff_{d}$ is an open manifold, the moduli space of shrubs is a (non-compact) manifold with boundary (the boundary consists of trees with outgoing edges of vanishing length).  Note that the dimension of the components of \eqref{eq:moduli_mushrooms_disjoint_union} depends only on the number of inputs:
\[ (r-1) + \sum_{i=1}^{r} (d_i - 1) = d - 1 .\]

\begin{defin}
A {\bf singular mushroom} with $d$ inputs is defined by the following data
\vspace{-10pt}
\begin{enumerate}
\item A partition $d = d_{1} + d_{2} + \cdots + d_{r}$ with $d_{k} \geq 1$, and
\item for each integer $1 \leq k \leq r$, a {\bf cap}, which is a singular disc with $d_{k} + 2$ marked boundary points with a distinguished arc connecting two successive (outgoing) marked points, and 
\item a {\bf stem}, which is a singular shrub with $r$ inputs.
\end{enumerate}
\end{defin}
\begin{rem}
Again, we consider singular caps whose moduli spaces $\overline{\Pil}_{d}$ can be identified with the Stasheff polytopes in their holomorphic disc model $\Rbar_{d+1}$.  Decomposing a singular cap into its irreducible components, we observe that those components whose boundary intersects the distinguished arc can themselves be interpreted as lower dimensional caps.  Not all irreducible components are of this type (See Figure \ref{singular_cap}).

\end{rem}

\begin{figure}[h]
\epsfxsize=2in
\epsffile{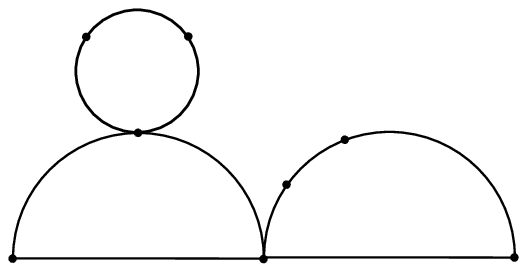} 
\caption{A singular cap consisting of two caps and an ``ordinary'' disc.}
\label{singular_cap}
\end{figure}

As usual, we will focus attention on {\bf stable} singular mushrooms.  These are characterized by stable caps and stem.  As before, we can write this space as 
\[ \coprod    \Shrubbar_{r}  \times \overline{\Pil}_{d_1} \times \cdots \times \overline{\Pil}_{d_r}.\] 
However, it is natural to think of the moduli space of stable singular mushrooms as a quotient of this space.

For each $d'$ and $d''$, we consider the embedding $\circ$
\begin{equation} \label{eq:gluing_caps_unlabelled} \overline{\Pil}_{d'} \times \overline{\Pil}_{d''} \to \partial \overline{\Pil}_{d' + d''} \hookrightarrow  \overline{\Pil}_{d' + d''}  \end{equation}
whose image is the stratum of $\partial \overline{\Pil}_{d' + d''} $ consisting of discs which can be decomposed into two components (not necessarily irreducible), each of which contains an outgoing marked point, and which have respectively $d'$ and $d''$ incoming marked points. 

On the other hand, for each integer $1 \leq k \leq r -1$, we have a map
\[ \vee_{k} \co \overline{\Shrub}_{r-1} \to \overline{\Shrub}_{r} \]
which creates two incoming edges of vanishing length at the $k$\th vertex.  We therefore have maps
\[ \xymatrix{  \Shrubbar_{r} \times   \overline{\Pil}_{d_1} \times \cdots  \times \overline{\Pil}_{d_{k-1}} \times \overline{\Pil}_{d_{k}} \times \cdots \times \overline{\Pil}_{d_r} \\
 \Shrubbar_{r-1} \times \overline{\Pil}_{d_1} \times \cdots \times \overline{\Pil}_{d_{k-1}} \times \overline{\Pil}_{d_{k}} \times \cdots \times \overline{\Pil}_{d_r} \ar[u]^{\vee_{k} \times \id^{r}} \ar[d]^{\id^{k-1} \times \circ \times \id^{r-k}} \\
 \Shrubbar_{r-1} \times  \overline{\Pil}_{d_1} \times \cdots \times \overline{\Pil}_{d_{k-1} + d_k} \times \cdots \times \overline{\Pil}_{d_r} }
 \]

We say that two stable singular mushrooms are {\bf equivalent} if they are identified upon gluing along the maps $\vee_{k} \times \id$ and $\id \times \circ \times \id$.  We write $\sim$ for the equivalence relation generated by all such diagrams:

\begin{defin} \label{defin:mod_space_stable_mushrooms}
The {\bf moduli space of stable mushrooms} with $d$ inputs is the quotient space
\[ \overline{\Champ}_{d} = \coprod   \Shrubbar_{r} \times \overline{\Pil}_{d_1} \times \cdots \overline{\Pil}_{d_r} / \sim .\]
\end{defin}

We will equip $\Champ_{d}$ with its inherited topology as a subset of $\overline{\Champ}_{d}$.  The spaces $\Champbar_{d}$ admit two important families of maps of operadic type relating them to the other moduli spaces we have studied. On the one hand, by grafting $r$ mushrooms at the endpoints of a Stasheff tree, we obtain a map
\begin{equation} 
\label{eq:apply_m_k_target_mushrooms}
 \Stasheffbar_{r}  \times \Champbar_{d_1} \times \cdots \times \Champbar_{d_r}  \to \Champbar_{d}
\end{equation}
where $\sum d_j = d$.  Alternatively, by attaching a singular disc at any of the incoming marked point, we obtain $d_1$ maps
\begin{equation}
\label{eq:apply_m_k_source_mushrooms}
 \Champbar_{d_1} \times \Rbar_{d_2}  \to  \Champbar_{d_1+d_2-1}.
\end{equation}

\begin{figure}[h]
\begin{center}
$\begin{array}{c@{\hspace{.5in}}c}
\epsfxsize=2in
\epsffile{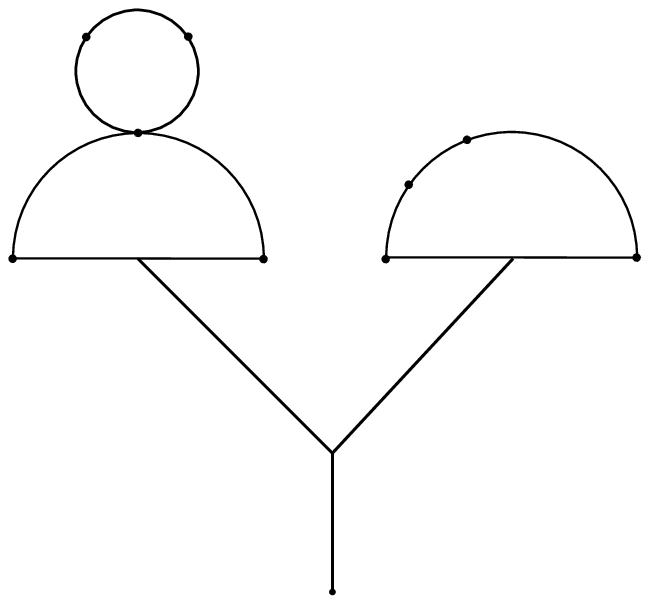} &
\epsfxsize=2in
\epsffile{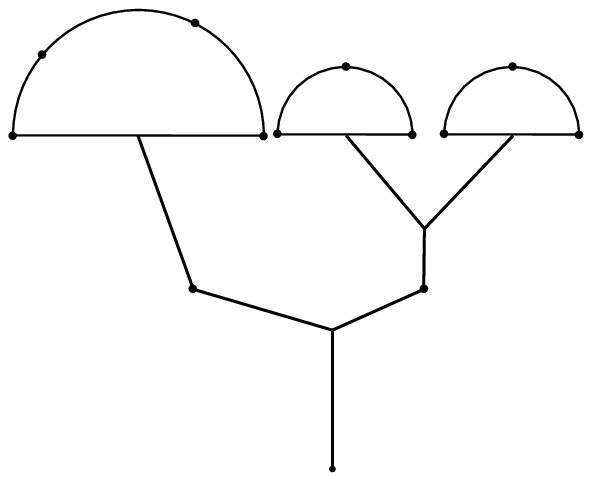}
\end{array}$
\end{center}
\caption{Singular mushrooms lying on two different boundary strata of $\overline{\Champ}_{4}$.} 
 \label{top_mushroom_boundary_stratum}
\end{figure}

The moduli space $\Rbar_{d+1}$ is stratified by cells labelled by topological types of ribbon trees with $d+2$ leaves, which induces a stratification of $\Pilbar_{d}$.  When considering this latter space, we shall assume that two such successive leaves are distinguished as outgoing.  Note that the trees in which the two outgoing leaves are not adjacent (i.e. share a vertex) are exactly those that label singular caps whose outgoing edge meets more than one irreducible component.   The reader may compare this constraint with the condition imposed by Boardman and Vogt on their coloured trees \cite{BV} (see also, \cite{forcey}), which shows that our moduli space $\Champbar_{d}$ is homeomorphic to the multiplihedron of the same dimension.  Moreover, the homeomorphism may be chosen so that the maps \eqref{eq:apply_m_k_target_mushrooms} and \eqref{eq:apply_m_k_source_mushrooms} are the operadic structure maps which make multiplihedra into a bimodule over the Stasheff operad.  We conclude:

\begin{lem} \label{lem-top_mushroom_boundary_stratum}
$\overline{\Champ}_{d}$ is a compact manifold with boundary stratified into finitely many (smooth) manifolds with corners. The top boundary strata are the images of the map \eqref{eq:apply_m_k_target_mushrooms} and \eqref{eq:apply_m_k_source_mushrooms} (See Figure \ref{top_mushroom_boundary_stratum}). \noproof
\end{lem}

Note that the top boundary strata of $\overline{\Champ}_{d}$ are the boundary strata of 
\[\coprod  \Shrubbar_{r} \times  \overline{\Pil}_{d_1} \times \cdots \times \overline{\Pil}_{d_r}  \]
which are not identified by the equivalence relation.  In fact, the top boundary strata which are identified by the equivalence relation are exactly those for which a pair of successive external edges have vanishing length.  It is easy to check that the remaining strata are exactly as described above.

As with trees (and discs), this construction can be done with labels.   Given a sequence $\vI = (i_0, \ldots, i_d)$, we shall consider the moduli space of mushrooms labelled by $\vI$, denoted $\Champbar_{\vI}$,  which is a copy of $\Champbar_{d}$.   We start by setting $\Pilbar_{\vI[k]}$ to be a copy of $\Pilbar_{|\vI|-1}$ where the incoming arcs are labelled by the sequence $\vI[k]$, which allows us to define $\Champbar_{\vI}$  to be the quotient (under the identification of used in Definition \ref{defin:mod_space_stable_mushrooms}) of the disjoint union
\begin{equation} \label{eq:product_description_mushroom}
\coprod  \Shrubbar_{\vR} \times  \Pilbar_{\vI[1]} \times \cdots \times \Pilbar_{\vI[r]} 
\end{equation}
where the initial elements of $\vI[k]$ agrees with the last element of $\vI[k-1]$, the sequence  $\vI$ is obtained by concatenating $\vI[1]$ through $\vI[r]$ and dropping the redundant terms, and $\vR$ is the sequence of length $r+1$ obtained by taking $i_0$ together with the last elements of  $\vI[1]$ through $\vI[r]$.   Whenever $\vR$ and $(\vI[1], \ldots, \vI[r])$ are any sequences satisfying this condition, we obtain a map
\begin{equation} \label{eq:labelled_mushrooms_left_module_stasheff}
 \Stasheffbar_{\vR}  \times \Champbar_{\vI[1]} \times \cdots \times \Champbar_{\vI} \to \Champbar_{\vI} 
\end{equation}
which is the analogue of \eqref{eq:apply_m_k_target_mushrooms} in the presence of labels.  The analogue of Equation
\eqref{eq:apply_m_k_source_mushrooms} is a map
\begin{equation}\label{eq:labelled_mushrooms_right_module_stasheff}
 \Champbar_{\vI[1]}  \times \Rbar_{\vI[2]}  \to  \Champbar_{\vI}
\end{equation}
where $\vI$ is obtained from $\vI[1]$ and $\vI[2]$ as in Equation \eqref{eq:compose_stasheff_labels}.  This later map is induced by a map
\begin{equation} \label{eq:gluing_disc_cap}
\Pilbar_{\vI[1]} \times \Rbar_{\vI[2]}   \to  \Pilbar_{\vI}.
\end{equation}

Finally, the gluing operation \eqref{eq:gluing_caps_unlabelled} also has a labelled analogue
\begin{equation} \label{eq:gluing_cap_cap} \Pilbar_{\vI[1]} \times \Pilbar_{\vI[2]}     \to  \Pilbar_{\vI} \end{equation}
where the last element of $\vI[1]$ is assumed to agree with the first element of  $\vI[2]$, and $\vI$ is obtained by concatenating the two sequences, and dropping the repeated term.

\section{From Floer to Morse cochains} \label{sec:floer_to_morse}
In this section, we construct a functor
\begin{equation} \label{eq:Fuk_to_Morse}
 \cG \co \Fuk(Q_1,Q_2) \to \Morse(Q_1,Q_2)
\end{equation}
using ``moduli spaces of mushrooms maps'' which we shall define in Section \ref{sec:mushroom_maps}, and a Lagrangian foliation on the plumbing, parametrized by $Q$, which we shall define in Section \ref{sec:lagrangian_foliation} (recall that $Q$ is the space obtained by gluing $Q_1$ and $Q_2$ along open balls).     Informally, a mushroom map consists of a gradient tree whose source is the stem, and a holomorphic map with Lagrangian boundary conditions on each cap.  The Lagrangian condition on the boundary components connecting incoming vertices (or an incoming vertex and an outgoing one) will be one of the Lagrangians $Q_i$ that we have been studying.  The outgoing boundary component of a cap (the one connecting the two outgoing vertices) will be required to map to a leaf of the Lagrangian foliation on $M$ whose label (a point in $Q$) agrees with the image of the corresponding endpoint of the stem.  From now on, we shall assume: 
\begin{equation} \label{eq:morse-floer-compatible-conditions_0}
 \parbox{36em}{The open sets $U_i$ used to define the Morse category are the images of the balls of radius $2$ in $\bR^{n}$ under the charts used in Section \ref{sec:introduce_fukaya}.}
\end{equation}

\subsection{A Lagrangian foliation} \label{sec:lagrangian_foliation}
Returning to the local model of Section \ref{sec:introduce_fukaya}, we find that the submanifold
\begin{equation}
 L_{\Delta} = \{ x_i = y_i \}_{i=1}^{n}
\end{equation}
is also Lagrangian, and given $\vx \in \bR^{n}$, we let $(\vx,0)+L_{\Delta}$ denote the translate by $(\vx,0)$ of $L_\Delta$; the union of all such leaves is a foliation of $\bC^{n}$.  Note that a fixed leaf $(\vx,0)+L_{\Delta}$ intersects $L_2$ transversely at $(0, -\vx)$.  Using Stokes's theorem to integrate the action of a path from $L_1$ to $L_2$ along $(\vx,0)+L_{\Delta}$, we conclude:
\begin{lem} \label{lem:stokes_application}
The integral of $\omega$ over any disc with counter-clockwise boundary conditions $(L_2, L_1,(\vx,0)+L_{\Delta})$ is proportional to $-|x|^2$; in particular no such disc is holomorphic if $x \neq 0$, in which case only constant ones arise. \qed
\end{lem}

A minor deficiency with this model is that it differs from the standard fibration of $\bC^{n}$ as the cotangent bundle of $L_1$ or $L_2$, so we must modify it in order to extend it to $M$.    We write  $L$ for the result of gluing $L_1$ to $L_2$ along the identification of their balls of radius $1$ about the origin under the diffeomorphism
\begin{equation} \label{eq:leaf_local_model}
 (\vx,0) \mapsto (0,-\vx).
\end{equation}
This space is a local model for $Q$ in the sense that there is an open subset of $Q$, containing all non-Hausdorff points, which is diffeomorphic to $L$, and which maps $L_1$ to the image of $Q_1$ and $L_2$ to the image of $Q_2$.  

\begin{figure}
  \centering
$$  \begin{matrix}
\epsfxsize=2in \epsffile{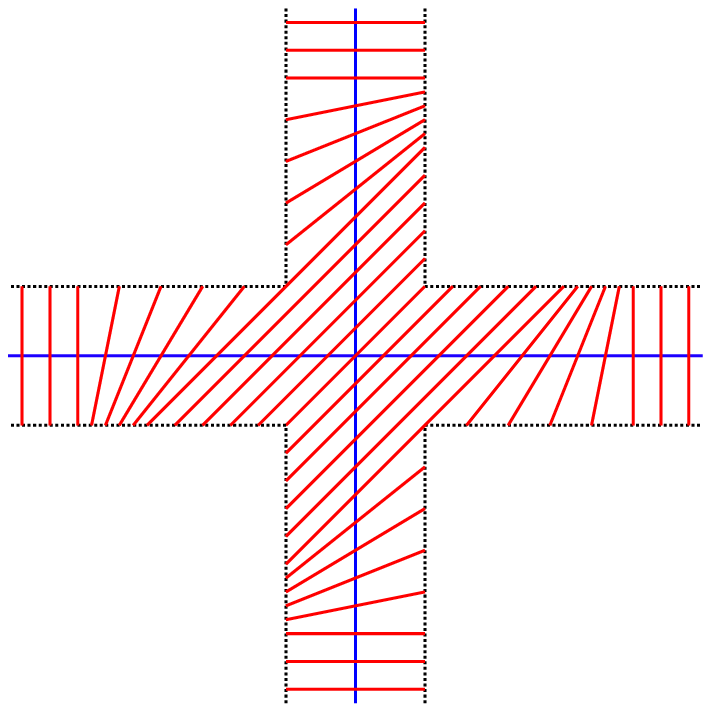}  & \epsfxsize=2in \epsffile{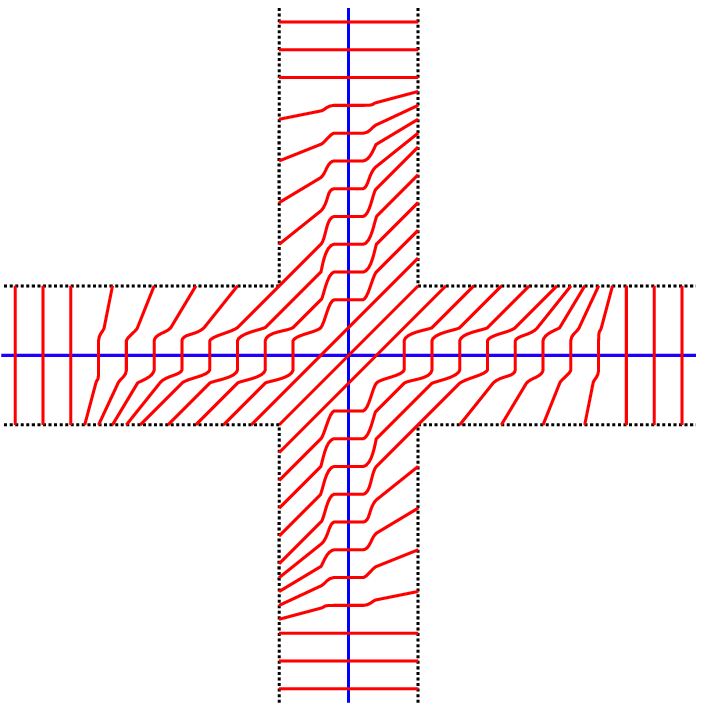} 
\end{matrix} $$
  \caption{ }
  \label{fig:foliation_model_dim_1}
\end{figure}

We would like to define a foliation of   $D^{*} L_1 \# D^* L_2 $, whose leaves $\sL_{\ell}$ are Lagrangian discs parametrised by $\ell \in L$.   and such that the following essential properties hold
\begin{align*}
&  \parbox{36em}{  $  \sL_{\ell} $ intersects $L_1$ at $(\vx,0)$ if and only if $(\vx,0)$ represents $ \ell $  in $L$ (with the same property for $L_2$)} \\
&  \parbox{36em}{  If $4 < |\vx|  $, then $ \sL_{(\vx,0)}$ and  $ \sL_{(0,\vx)}$ are respectively given by the affine Lagrangians orthogonal to $ L_1 $ and $L_2$ at  $(\vx,0)$ and $ (0,\vx) $.}
\end{align*}

The first condition gives the correspondence between leaves and points of $L$, while the second implies that this foliation agrees with the canonical foliation of the cotangent  bundles of $L_1$ and $L_2$ away from a neighbourhood of the origin, and hence can be extended to $M$.  In addition, we require three technical conditions which will be used to prove compactness for moduli spaces of holomorphic curves (see Section \ref{sec:comp-curv-leaves}), and which depend on parameters $\epsilon \ll \delta$ which can be made arbitrarily small; $\epsilon$ already appeared in the construction of $M^{\ins}$ Section \ref{app:review_seidel}:
  \begin{align}
\label{item:orthogonal_interior}   &  \parbox{36em}{   If $\delta  \leq |\vx|$ then the intersections $ \sL_{(\vx,0)} \cap M^{\ins}_1$ and  $ \sL_{(0,\vx)} \cap M^{\ins}_2$ agree with the Lagrangians orthogonal to $ L_1 $ and $L_2$ at  $(\vx,0)$ and $ (0,\vx) $.} \\
\label{item:diagonal_small_epsilong}   &  \parbox{36em}{  If $|\vx| \leq \delta $  then  $ \sL_{(\vx,0)}  $  and $ (\vx,0)+L_{\Delta} $ agree in the region where $  |\vx + \vy| \geq \delta$.} \\
\label{item:C-1-close-foliation}   &  \parbox{36em}{  If $ \delta \leq |\vx| \leq 2 $  then  $ \sL_{(\vx,0)}  $  and $ (\vx,0)+L_{\Delta} $ agree away from a $4 \epsilon$  neighbourhood of $L_1 \cup L_2$, and are $C^1$-close (with $C^1$-distance independent of $\delta$) in this region.} 
\end{align}
The reader should consult the picture on the right of Figure \ref{fig:foliation_model_dim_1} which illustrates the case $n=1$.  To see the fact that the leaves are parametrised by $L$, note  that, in the second and fourth quadrant, there are two leaves passing through the corner point; these form the non-Hausdorff locus of the space of leaves $\sL$.

In order to construct $\sL$, we choose any monotone function $\kappa \co [2,4] \to [0,1]$ which is identically $1$ near $2$ and identically $0$ near $4$ and whose derivative is uniformly bounded by $1$, and consider the affine Lagrangian planes
\begin{equation} \label{eq:linear_foliation}
(\vx[0],0)+\{ (\vx,\vy) |  \kappa(|\vx[0]|) x_i = y_i \}.
\end{equation}
The proof that these planes do not intersect each other in  $D^{*} L_1 \# D^* L_2 $, and hence form a foliation, is an elementary computation left to the reader (see the left side of Figure  \ref{fig:foliation_model_dim_1}).  Note that this foliation satisfies all but Condition \ref{item:orthogonal_interior} above.  In order to ensure that the leaf corresponding to a vector $\vx[0]$ whose norm lies between $4 \epsilon$ and $4$ intersects $M^{\ins}$ in Lagrangian planes orthogonal to $L_1$ and $L_2$, we deform this foliation by affine Lagrangians using an appropriate Hamiltonian flow.  The existence of such a flow can be proved, for example, by noting that the region
\begin{align} \label{eq:x-region_rectify}
  3 \epsilon < |\vx| <4 & \textrm{ and } |\vy| <  \epsilon
\end{align}
carries a unique function $H_0$ which vanishes along $L_1$, and whose restriction to any leaf \eqref{eq:linear_foliation} has differential that agrees with the restriction of $- \sum (x_i - x_i^{0}) dy_i$; we can write its restriction to a leaf explicitly:
\begin{equation*}
H_0|\sL_{(\vx[0], 0)} =    \frac{\kappa(|\vx[0]|)}{2}| \vy|^2 .  
\end{equation*}
We define $H_{T}$ implicitely via the property that it also vanishes along $L_1$ and that the restriction of $H_{T}$ to the image of a leaf under the flow $\phi^{T}$ generated by the Hamiltonians $\{ H_{t} \}_{t \leq T}$ also has differential agreeing with the restriction of  $- \sum (x_i - x_i^{0}) dy_i$.  The reader may easily check that the image of each leaf under $\phi^{T}$ is a linear Lagrangian, passing through the same point of $L_1$, and that the norm of the restriction of $\sum (x_i - x_i^{0}) dy_i $ to each leaf decreases with $T$.   Whenever $T=1$, we find that the image of each leaf is in fact orthogonal to $L_1$.  We extend $H_{T}$ to a function on $M$ so that it satisfies the analogous conditions in the region
\begin{align} \label{eq:y-region_rectify}
  3 \epsilon < |\vy| < 4 & \textrm{ and } |\vx| <  \epsilon
\end{align}
and define the new foliation $\sL$ to be the image of \eqref{eq:linear_foliation} under $\phi^{1}$.  More precisely, we require $H_{T}$ to have support contained in a $4 \epsilon$ neighbourhood of the union of the  regions \eqref{eq:x-region_rectify} and \eqref{eq:y-region_rectify}, and have $C^{0}$ and $C^1$ norms bounded by (say) $2 \epsilon$.  These support conditions imply that $\sL$ inherits the remaining conditions from the foliation \eqref{eq:linear_foliation}.  Choosing $\epsilon \ll \delta$ implies that the estimate of Lemma \ref{lem:stokes_application} applies to $\sL$ as well:
\begin{lem} \label{lem:stokes_application_perturbed}
The integral of $\omega$ over any disc with counter-clockwise boundary conditions $(L_2, L_1,\sL_{(\vx,0)})$ is strictly negative if $ \delta \leq |\vx|$. \qed
\end{lem}

\subsection{Moduli spaces of caps}

In this section, we will discuss moduli spaces of maps from caps to the plumbing of two cotangent bundles.  Strictly speaking, the setup we will be using is not exactly the one we will later need to establish the desired results about maps of mushrooms.  However, since all the main ideas are already present and the notation is much simpler, we will start here.

Condition \eqref{eq:morse-floer-compatible-conditions_0} identifies $Q_{1,2}$ and $Q_{2,1}$ with balls of radius $2$ in flat Euclidean space.  We assume that the metrics on $Q_1$ and $Q_2$ restrict to this flat metric in $U_1$ and $U_2$.  The first step is to impose a final condition on the perturbation $H^{i,j}$ and the Morse functions $f_{i,j}$:
\begin{equation} \label{eq:morse-floer-compatible-conditions}
\parbox{36em}{The image of $Q_{i}$ under the time-$1$ Hamiltonian flow of $H^{i,j}$ agrees with the graph of the differential of $f_{i,j}$, in a neighbourhood of $Q_i \cap Q_{j}$  which is identified with a (small) disc cotangent bundle of $Q_{i,j}$ using the foliation $\sL_{q}$.  Moreover, the images of $Q_{1}$ and $Q_{2}$ under the time-$t$ Hamiltonian flow of $H^{i,j}$ are transverse to all leaves $\sL_{q}$ if $|t| \leq 1$.  Finally, all critical points of $f_{i,j}$ have distinct values.}
\end{equation}

Note that when $i =j$, these conditions may easily be achieved by taking any (sufficiently small) generic function $H^{i,i}$.  The image of $Q^i $ under such a map may be written uniquely up to a global constant as the graph of the differential of a function $f_{i,i}$.  The same strategy works for $i \neq j$ once we observe that $L_2$ (in $\bC^{n}$) is the graph of the differential of $|\vx|^{2}$ as a function on $L_{1}$ with respect to the linear foliation $(0, \vx) + L_{\Delta}$, of which $\sL_{q}$ is a small perturbation.  In particular, if we choose $H^{2,1} = H^{1,2} = 0$, we conclude:
\begin{lem}
 \label{lem:morse_function_grows}
The distance to  $b$ grows along gradient flow lines of $f_{1,2}$ (and decreases along gradient lines of $f_{2,1}$) in the subset of $U_{1}$ (respectively $U_2$) where this distance is greater than $ \delta$. \qed
\end{lem}

Assumption \eqref{eq:morse-floer-compatible-conditions} establishes a correspondence between critical points of $f_{i,j}$ and intersection points of the corresponding Lagrangians, which implies:
\begin{cor} \label{cor:morse_floer_generators_equal}
There is a bijective correspondence between the generators of $CM^{*}(f_{i,j})$ and those of $CF^{*}(Q_i,Q_j)$. \qed
\end{cor}
Floer proved in \cite{floer} that, for an appropriate choice of almost complex structure, this correspondence becomes a chain isomorphism.  Instead of using this result directly, we shall exhibit a chain equivalence, which will be the first term of an $A_{\infty}$ functor:

Fix a cap $P \in \Pil_{\vI}$ for a sequence $\vI$ of length $d+1$.  To each point $(z,q)$ of $\partial P \times Q$, we can associate a Lagrangian $L \subset M$ as follows:
\vspace{-10pt}
\begin{itemize}
\item If $s$ lies on the  outgoing segment, we let $L = \sL_{q}$,
\item otherwise, orienting the disc counterclockwise starting at the outgoing segment, $L = Q_{j}$ if $s$ lies on the $j$\th segment.
\end{itemize}

Let us assume, as in \cite{seidel-book}*{Section 9g}, that we have chosen a universal consistent choice of strip-like ends on the moduli space of discs, i.e. a choice of strip-like end on every stable holomorphic disc which is compatible with gluings for small parameters.  Note that the isomorphism $\Pil_{d} \cong \cR_{d+2}$ therefore determines a choice of strip-like ends on the cap $P$.  In addition, for points sufficiently close to the boundary of $\Pilbar_{d}$, we obtain a decomposition  into a thin part (obtained as the image of a strip-like end) and a thick part (the complement).  The thin part decomposes as a union of strip-like ends with the images of maps
\begin{equation} [0,1] \times [-R,R] \to P. \end{equation}
Whenever both boundary components of a strip in the thin part are labelled by a Lagrangian $Q_{i}$, the Floer datum determines a perturbation of the $\dbar$ equation.  We extend this choice (which was fixed in the construction of $\Fuk(Q_1,Q_2)$ to the rest of the thin part as follows:
\begin{equation} \label{eq:Floer_datum_caps_outgoing}
\parbox{36em}{the Floer datum for $Q_{i}$ and $\sL_{q}$ is given by a fixed almost complex structure $J_{q}$ and vanishing Hamiltonian.}
\end{equation}
This is a reasonable assumption because $\sL_{q}$ intersects $Q_i$ transversely at a single point.  No breaking of holomorphic strips can therefore occur, so one need perturb neither the complex structure nor the Lagrangians in order to avoid it.

\begin{defin}
A {\bf $Q$-parametrized perturbation datum} on $P$ is a pair $(K^{\vI}, J^{\vI})$ where
\begin{align*}
K^{\vI} & \in \C^{\infty}(P \times Q, T^{*}P \otimes \sH ) \\
J^{\vI} & \in \C^{\infty}(P \times Q,  \sJ ).
\end{align*}
We require that (i) the restriction of the perturbation data $(K^{\vI}, J^{\vI})$ to the thin part agrees with the labelling Floer data and (ii) if $(z,q) \in \partial P  \times Q$ is labelled by the Lagrangian $L$, then for each tangent vector $\xi \in T_{z} \partial P$, we have $K^{\vI}(z,q)(\xi)|_{L} \equiv 0$.
\end{defin}

In other words, we choose a family of almost complex structures $J^{\vI}$ on $M$ parametrized by $Q$ and by the points of the disc $P$, which are required to agree with the choices made in defining the Fukaya category near the incoming marked points, and with $J_q$ near the outgoing points.  We also choose a family of $1$-forms $K^{\vI}$ on $P$, parametrized by $Q$, and with values in the space of functions on $M$.  The $1$-forms are also compatible with the Hamiltonian perturbations used to define the Fukaya category.

Note that $K^{\vI}(z,q)(\xi)$ induces a Hamiltonian vector field $Y^{\vI}(z,q)(\xi)$. Given such a datum, for each $q \in Q$ we can therefore consider the perturbed Cauchy-Riemann equation on $P$:
\begin{equation} \label{perturbed-CR} (du - Y^{\vI})^{0,1} = 0 \end{equation}
where the $(0,1)$ part is taken with respect to the almost complex structure $J^{\vI}(z,q)$, and the Lagrangian boundary conditions are prescribed by the labels.

\begin{defin}
A {\bf $Q$-parametrized perturbation datum} on a singular $d$-cap $P$ is the choice of a $Q$-parametrized perturbation datum on each irreducible component.  
\end{defin}

As the Fukaya category requires counting curves parametrised over the moduli space $\cR_{d}$, we consider the analogous generalisation for caps:
\begin{defin}
A {\bf universal $Q$-parametrized perturbation datum for caps} $(\bfK^{\Pil}, \bfJ^{\Pil})$ is a choice of a $Q$-parametrized perturbation datum for each singular cap $P$, smoothly varying with respect to the modulus, and which  is compatible with the perturbation data $(\bfK^{\cR}, \bfJ^{\cR})$ and the maps \eqref{eq:gluing_disc_cap} and \eqref{eq:gluing_cap_cap}.
\end{defin}
\begin{rem}
The compatibility condition can be stated explicitly as follows:  Every component of a singular cap in the image of the maps \eqref{eq:gluing_disc_cap} or \eqref{eq:gluing_cap_cap} is a priori equipped with two perturbation data, one coming from being the component of a disc or cap in a factor of the source, the other from being the component of a cap in $\Pilbar_{\vI}$. We require these data to agree.  
\end{rem}

The proof of the following result follows immediately from the inductive method of proof for the analogous Lemma 9.5 of \cite{seidel-book}.
\begin{lem} \label{lem:extend_pert_data_caps}
Every perturbation datum on a fixed (smooth) cap may be extended to a universal perturbation datum. \noproof
\end{lem}
Let $\vp = \{ p_{i_k} \}_{i=1}^{d}$ denote a sequence of time-$1$ Hamiltonian chords of  $H^{i_{k-1}, i_{k}}$ points starting at $Q_{i_{k}}$ and ending on $Q_{i_{k-1}}$  In addition, we will denote  the point of intersection $\sL_q \cap Q_i$ by $q_i$.

\begin{defin}
The {\bf moduli space of caps with inputs} $\vp$ is the space $\Pil(\vp) = \coprod_{q \in Q} \Pil_{q}(\vp)$ of solutions to the perturbed Cauchy Riemann equation \eqref{perturbed-CR} with incoming marked points mapping to $\vp$ and outgoing marked points mapping to $q_{i_0}$ and $q_{i_d}$ which project to the same point $q \in  Q$.
\end{defin}

In Section \ref{app:transversality_cap} we prove transversality for the moduli space of caps:
\begin{lem} \label{cap_transversality}
For a generic choice of  universal consistent $Q$-parametrized cap perturbation data,  $\Pil(\vp) $ is a smooth manifold.  More generally, any stratum of $\Pilbar(\vp)$ consisting of a single cap (and an arbitrary number of discs) has the expected dimension, and,  if $N \subset Q$ is a submanifold, the projection map $\Pil(\vp) \to Q$ (at every smooth point of the source) is generically transverse to $N$ on every such stratum. 
 \end{lem}
We shall also prove compactness in Section  \ref{sec:comp-curv-leaves}, by appropriately generalising Lemma \ref{lem:stokes_application}:
\begin{lem} \label{lem:cap_cpactness}
Whenever $i_0=i_d$ or $(i_0,i_d) = (2,1)$, the Gromov bordification $\Pilbar(\vp)$ is compact.  In the second case, $  \Pil_{q}(\vp) $ is in addition empty whenever $q$ lies more than $4 \epsilon$ away from $b$.

If $(i_0,i_d )=(1,2)$, then the Gromov bordification of the moduli space of caps with boundaries on leaves corresponding to points within $4 \epsilon$ of $q$ is compact.
\end{lem}

\subsection{Mushroom maps} \label{sec:mushroom_maps}

Given the seeming complexity of the definitions coming up, a word of explanation is in order.  Naively, a mushroom map is simply a holomorphic map defined on each cap together with a gradient tree satisfying  certain compatibility conditions. Let us fix the homeomorphism type of a tree $S$ with $r$ inputs, and a sequence of integers $d_j$ such that $\sum d_j = d$.  Assume we have fixed in addition, a sequence of labels $\vI$ of length $d$, a sequence $\vp = \vp[1] \cup \cdots \cup \vp[r]$ of Hamiltonian chords with endpoints on the Lagrangians $Q_{i_{k-1}}$ and $Q_{i_{k}}$, and a critical point $x_0$ of $f_{i_0, i_d}$.  Each sequence $\vp[k]$ determines a subsequence  $\vI[k]$ of $\vI$, whose elements we denote $\{i_{j}^{k} \}_{j=1}^{d_k}$.

We would have liked to define the space of mushroom maps with inputs $\vp$, output $x_0$, and fixed topological type as a fibre product
\begin{equation} \label{eq:naive_mushroom}
\Shrub_{S}(x_0, Q_{i_0=i^{1}_0, i^2_0}, \ldots Q_{i_0^{r}, i_{d_r}^{r} = i_{d}}) \times_{Q^{r}} \left( \Pil(\vp[1]) \times \cdots \times \Pil(\vp[r]) \right)
\end{equation}
where $\vI[k]=\{i^{k}_{j} \}$ is the sequences of labels for the incoming segments of an element of $\Pil(\vp[j])$. Whenever $S$ has no incoming leaves of vanishing length, one may easily choose perturbation data on the moduli space of shrubs and caps once and for all, so that all such fibre products are transverse (we shall see later how the lack of Hausdorffness is circumvented).  However there is no way to achieve transversality in this way for the stratum shown in Figure \ref{deepest_stratum} whenever some of the caps have the same inputs.  

\begin{figure}[h]
\epsfxsize=2.5in
\epsffile{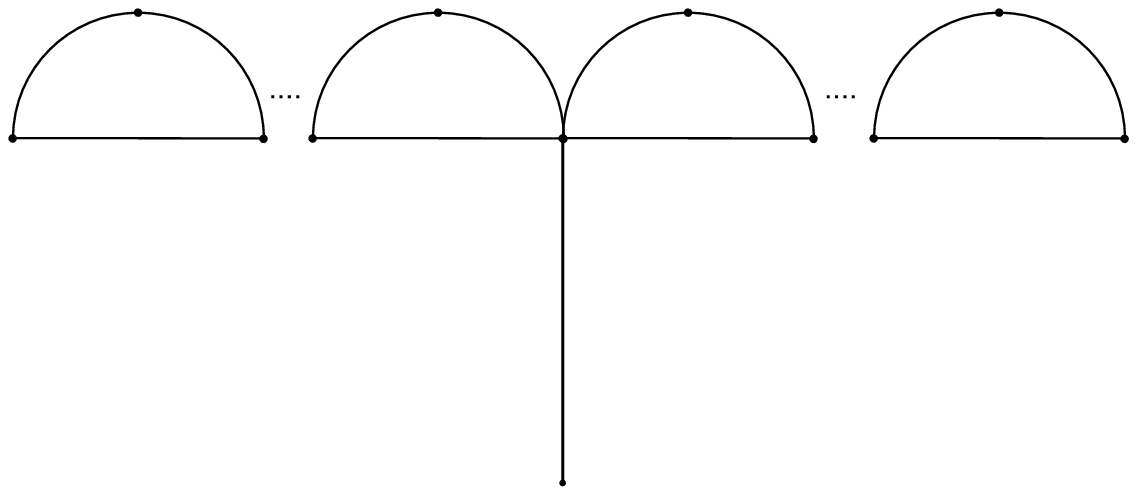} 
\caption{The deepest stratum of $\Champ_{d}$.}
\label{deepest_stratum}
\end{figure}

The main observation (we used a similar trick in achieving regularity for the moduli space of cascades in \cite{abouzaid-seidel}) is that there is no reason for the perturbation data on the various caps of a fixed mushroom to be defined in exactly the same way.   In particular, even if all the caps of the mushroom shown in Figure \ref{deepest_stratum} have marked points labelled by the same Hamiltonian chords, our setup will allow up to choose different perturbation data for each of them, thereby making transversality easy to achieve.  The rest of this section implements this idea, which we shall return to repeatedly (see in particular Remark \ref{rem:free_choice} and the discussion following Lemma \ref{lem:surjective_restriction_pert_data})

We first consider a simplest case of a mushroom with one input whose incoming boundary segments are labelled $(i_0,i_1)$. 
\begin{defin} \label{lem:base_case}
A {\bf $Q$-parametrized perturbation datum} on the unique mushroom with one input is the choice of (i)  a family of almost complex structures in $\sJ$ parametrised by a disc with three punctures, agreeing with the respective complex structure at the end and (ii) a closed $1$-form $\gamma$ on the cap with three punctures whose pullback under the positive end with boundary labels $(Q_{i_0}, Q_{i_1})$ and the negative end with boundary $(Q_{i_0},\sL_{q})$ agrees with $dt$, and which vanishes near the remaining end. 
\end{defin}

In this case, we fix the Hamiltonian perturbation datum as
\begin{equation}  K^{i_0,i_1} \equiv H^{i_0, i_1} \otimes \gamma. \end{equation}

If we have more inputs, we shall also require a perturbation of the stem.  Recall that a labelling $\vI$ of the inputs of a mushroom $C$ is equivalent to the data of compatible labellings $\vI[k]$ of all caps, and $\vR$ of the stem.  In particular, each edge of the stem is equipped with a gradient flow equation which we shall have to perturb.  Unlike the usual setup, we shall allow the Floer data on the ends meeting the outgoing segment of a cap to vary depending on the modulus.  This does not cause any difficulties as no breaking of holomorphic strips can happen at these ends.

\begin{defin}
A {\bf $Q$-parametrized perturbation datum} on a singular mushroom $C$ is the choice of (i) a gradient perturbation datum on the stem of $C$ and (ii) a $Q$-parametrized perturbation datum on each (possibly singular) cap such that
\begin{equation} \label{eq:varying_floer_datum} \parbox{36em}{The Floer datum for $(\sL_{q}, Q_{i,d})$ and $(Q_{i_0}, \sL_{q})$ is given by the constant choice of Condition \eqref{eq:Floer_datum_caps_outgoing} with the following exception:  If the $j$\th cap has exactly $2$ incoming segments with Lagrangian labels $(Q_{i_0},Q_{i_1})$, then the Hamiltonian perturbation associated to the end $(Q_{i_0}, \sL_{q})$ is given by  multiplying $H^{i_0,i_d}$ by a non-negative constant which is bounded by $1$.  This constant vanishes whenever the $j$\th incoming leaf of the stem is short (say, has length less than $10$) and equal $1$ whenever it is sufficiently long.}  \end{equation}
\end{defin}
\begin{rem}
The perturbation of the gradient flow need not be compatible with the simplicial subdivision $\sQ$ as it was in Section \ref{sec:simp_to_morse}.  However, we still require that the gradient flow be transverse to the boundary of $Q_{i,j}$ (and pointing in the appropriate direction).
\end{rem}
\begin{rem}
The constraint \eqref{eq:varying_floer_datum} is designed to achieve two things:  (i) whenever a mushroom with a single input breaks off, we need the perturbation datum to agree with the one chosen in Definition \ref{lem:base_case} and (ii) whenever two successive incoming edges of the stem have vanishing length, we require vanishing Floer data in order for the gluing theorem to hold.
\end{rem}
Of course, the perturbation datum is not strictly speaking parametrized by $Q$ anymore, but rather by $Q^{r}$.  For the next definition, we fix, as before a sequence $\vp =(p_1, \ldots, p_{d})$ of Hamiltonian chords with endpoints on $Q_1$ and $Q_2$ and a critical point $x_0$ of $f_{i_0,i_d}$.

\begin{defin}
A {\bf singular mushroom map} $\Psi$ with inputs $\vp$ and output $x_0$  is a collection of maps
\begin{align*}
\Psi_{S} & \co S(C) \to Q \\
\Psi_{j} & \co P_{j}(C) \to M.
\end{align*}
where $\Psi_{S}$ is a (singular) perturbed gradient tree (as in Definition \ref{defin:shrub_map}), and $\Psi_j$ are finite energy solutions to the perturbed Cauchy-Riemann equation \eqref{perturbed-CR} on the relevant (possibly singular) caps satisfying the following conditions
\begin{enumerate}
\item If the $k$\th incoming vertex lies on $P_{j}(C)$, then $\Psi_{j}$ converges on the corresponding end to $p_k$, and the incoming boundary segments of $P_j(C)$ are mapped to the labelling Lagrangians.
\item The outgoing arc on $P_{j}(C)$ is mapped to $\sL_{\Psi_{S}(v_{j})}$, where $v_j$ is the $j$\th incoming leaf of the stem.
\end{enumerate}
\end{defin}
Note that the last condition and the removal of singularities theorem forces the outgoing marked points of the $j$\th cap to be mapped to the lifts of $\Psi_{S}(v_{j})$ to the appropriate manifold $Q_1$ or $Q_2$.

The next result is the main tool in establishing compactness for maps of mushrooms:
\begin{lem}
The image of every cap of a singular mushroom map is contained in a compact subset of $M$, and every edge of a stem labelled by $(i,j)$  with $i \neq j$ is mapped to the ball of radius $\delta$ in $Q_{i,j}$.  In particular, if the distance between $q$ and $b$ is greater than $\delta$, there is no mushroom map admitting a cap whose first and last labels differ, and whose boundary condition on the outgoing arc is $\sL_{q}$.
\end{lem}
\begin{proof}
The proof is a combination of Lemma \ref{lem:cap_cpactness} and the proof of Lemma \ref{lem:strips_moduli_spaces_cpact}.  We focus on the case of a cap $P_{j}(C)$ whose first segment (in the counterclockwise order as usual) is mapped to $Q_1$, and penultimate one to $Q_2$ (Lemma \ref{lem:cap_cpactness} excludes the other possibility).

As in the proof of Lemma  \ref{lem:strips_moduli_spaces_cpact}, we find that either (i) there is a descending arc on the stem $S(C)$, starting at the $j$\th incoming leaf of $S(C)$ and ending on the outgoing edge, all of whose edges are labelled $(1,2)$ or (ii) there is an arc from the $j$\th incoming leaf of $S(C)$ to another leaf meeting an incoming cap $P_{j'}(C)$ which can be broken into a descending arc with edges labelled $(1,2)$, followed by an ascending arc with edges labelled by $(2,1)$.  Since the perturbed negative gradient flow along an edge labelled by $(1,2)$ points outward (and inward for edges labelled by $(2,1)$), we conclude that the critical point of $f_{1,2}$ which is the output of our mushroom map (or the image of the cap $P_{j'}(C)$ under the projection to $Q_{2,1}$) lies outside the ball of radius $\delta$ about the base point in $Q_{1,2}$ (or in $Q_{2,1}$ in the second case).  This contradicts our assumption \eqref{eq:morse-floer-compatible-conditions} on the gradient flow $f_{1,2}$, or Lemma \ref{lem:cap_cpactness} applied to the cap $P_{j'}(C)$.
\end{proof}

Following the construction for discs and trees, we shall consider perturbation data varying with the modulus of the mushroom:

\begin{defin} \label{lem:universal_datum_mushrooms}
A {\bf universal consistent $Q$-parametrized perturbation datum for mushrooms} is a choice $(\bfX^{\Champ}, \bfK^{\Champ}, \bfJ^{\Champ})$ of $Q$-parametrized perturbation data for each singular mushroom $C$,  varying smoothly in each stratum of $\Champbar$, and which is compatible with (i) the maps  \eqref{eq:labelled_mushrooms_right_module_stasheff} and the perturbation data  $(\bfK^{\cR}, \bfJ^{\cR})$ used to define $\Fuk(Q_1,Q_2)$ and (ii) the maps  \eqref{eq:labelled_mushrooms_left_module_stasheff} and the perturbation datum $\bfX^{\Stasheff}$ used  to define $\Morse(Q_1.Q_2)$.  Moreover, whenever the incoming leaf of the stem meeting a given cap $P_{j}$ has lenght smaller than $1$, we require that the Floer datum on the ends adjacent to its outgoing segment to be the constant ones fixed in \eqref{eq:Floer_datum_caps_outgoing}.
\end{defin}
\begin{rem}\label{rem:free_choice}
Had we naively pulled back the perturbation datum for a mushroom using the projections of components of $ \Champ_{d} $ to a product as in Equation \eqref{eq:product_description_mushroom}, we would not have been able to achieve transversality in the situation pictured in Figure  \ref{deepest_stratum}.  However, while the stratum labelled by the topological type of such mushrooms lies on the boundary of the Gromov compactification of a component of $\Champ_{d}$, it in fact lies in the interior of $\Champbar_{d}$.  In particular, the choice of perturbation data for a fixed $d$ on this part of the moduli space is essentially independent from the choices made for smaller value of $d$.  This implies that the choice of datum of each cap is unconstrained by the other choices, thereby allowing us to achieve transversality as explained in the discussion following Lemma \ref{lem:surjective_restriction_pert_data}.
\end{rem}
Given such a perturbation datum, we can define the moduli spaces which allows us to construct a functor from $\Morse(Q_1,Q_2)$ to $\Fuk(Q_1,Q_2)$.

\begin{defin} \label{mushroom-maps}
The {\bf moduli space of mushroom maps} 
\[ \Champ(x_{0};\vp)  \]
is the space of maps $\Psi$ from a non-singular mushroom $C$ with incoming marked points $\vp$ and outgoing marked point $x_0$.  The space of stable maps from a possibly singular mushroom is denoted
\[ \Champbar(x_0;\vp). \]
\end{defin}

The topology on $\Champbar(x_0;\vp)$ is defined as usual using the Gromov  $\C^{\infty}$ topology for the underlying caps and stem.  Moreover, this space carries a natural stratification coming from the stratifications on $\Pilbar$ and $\Shrubbar$ (and possible breakings of gradient flow lines between critical points, or holomorphic strips at the incoming ends).  We state the main result of Section \ref{sec:transverse_mushrooms}:
\begin{lem} \label{lem:mushrooms_expected_dimension}
For a generic choice of perturbation data, all strata of  $\Champbar(x_0;\vp)$ are of the expected dimension.
\end{lem}

The proof of this next result is given in Section \ref{sec:gluing_mushrooms}, and requires a gluing result in codimension $1$ for moduli spaces of caps:
\begin{prop} \label{prop:1_dim_moduli_space_manifold}
Given generic perturbation data,  $\Champbar(x_0;\vp)$ is a compact $1$-dimensional manifold with boundary whenever
\begin{equation} \deg(x_0) = 1-d+ \sum \deg(\vp) . \end{equation}
\end{prop}

\subsection{Construction of the functor}
We begin by defining the linear term of the functor $\cG$ promised in the beginning of this section.  To do this, we first observe that the index theorem (say, Equation (12.2) in \cite{seidel-book}) gives a natural isomorphism 
\begin{equation} \lambda (\Pil_{q}(p_1))  \cong \ro_{q_{i_0}}^{\vee} \otimes \ro^{\vee}_{p_1}  \otimes  \ro_{q_{i_1}} \cong  \ro^{\vee}_{p_1} \end{equation}
where $\Pil_{q}(p_1)$ is the moduli space of caps projecting to a fixed point $q$ in $Q$, and we are writing $q_{i}$ for the intersection point between $Q_i$ and $\sL_{q}$.    It follows immediately that the determinant bundle of the moduli space of caps admits a natural isomorphism
\begin{equation} \label{eq:orientation_cap_one_input} \lambda(\Pil(p_1)) \cong \ro^{\vee}_{p_1} \otimes  \lambda(Q). \end{equation}

Given a mushroom $\Psi$ in $\Champ(x_0,p_1)$, the description of the moduli space of such mushrooms as the fibre product of $\Pil(p_1)$ with $W^{s}(x_0)$ over $Q$ gives an isomorphism
\begin{equation} \label{eq:orientation_moduli_mushrooms} \lambda(\Champ(x_0,p_1)) \otimes  \lambda(Q) \cong\lambda(W^{s}(x_0)) \otimes   \lambda(\Pil(p_1)) . \end{equation}
Specialising to the case $\deg(x_0) = \deg(p_1)$, and using the previous result for  $\lambda(\Pil(p_1))$ and the definition of $\ro_{x_0}$, we conclude that every mushroom $\Psi$ determines an isomorphism
\begin{equation}  \ro_{p_1}  \cong \ro_{x_0},  \end{equation}
and hence a map
\begin{equation} \cG^{\Psi}\co   |\ro_{p_1}|  \to |\ro_{x_0}| \end{equation}
which is the contribution of $\Psi$ to the linear part of the functor $\cG$:
\begin{align*}
\cG^{1} \co  CF^{*}(Q_i,Q_j) & \to CM^{*}(f_{i,j}) \\
[p_1]  & \mapsto  (-1)^{n \deg(p_1)} \sum_{\stackrel{\deg(x_0) = \deg(p_1)}{\Psi \in  \Champ(x_0,p_1) }  }   \cG^{\Psi}( [p_1]).
\end{align*}
Ignoring signs, the fact that $\cG^{1}$ is a chain map is a consequence of Proposition \ref{prop:1_dim_moduli_space_manifold}.  Indeed, considering $p_1$ and $x_0$ such that
\begin{equation} \deg(x_0) = \deg(p_1) + 1, \end{equation}
the boundary strata of $\Champbar(x_0,p_1)$ consist of either (i) a critical point $x_1$ of $f_{i,j}$ and an element of $\Stasheff(x_0,x_1) \times \Champ(x_1,p_1)$ (i.e., a rigid mushroom with input $p_1$ and output $x_1$ together with a rigid gradient trajectory from $x_1$ to $x_0$) or (ii) an element of $ \Champ(x_0,p_0) \times \cR(p_0,p_1)$ for a time-$1$ chord $p_0$ from $Q_i$ to $Q_j$ for the Hamiltonian $H^{i,j}$.  These two cases correspond to the two sides of the equation
\begin{equation} \label{eq:chain_map_fuk_morse} \cG^{1} \circ  \mu_1^{\F}  = \mu_{1}^{\S} \circ \cG^{1}.  \end{equation}

To prove the correctness of the signs, we must compare the natural isomorphism \eqref{eq:orientation_moduli_mushrooms}, with a product isomorphism coming from its boundary strata.  The case (i) above is virtually indistinguishable from the analogous situation for the moduli spaces of shrubs interpolating between simplicial and Morse cochains,  analysed in the discussion surrounding Equation  \eqref{eq:analyse_shrub_breaking_trajectory}.  The result of that discussion is that there is a sign difference given by the parity of $\deg(p_1)$ between the product orientation, and the one induced from the interior.

In the other case, we first note that a rigid strip gives a canonical isomorphism 
\begin{equation} \lambda(\bR) \cong \ro_{p_0} \otimes \ro_{p_1}^{\vee} \end{equation}
with the left hand side corresponding to translation.  Taking the tensor product with Equation \eqref{eq:orientation_moduli_mushrooms} with $p_1$ replaced by $p_0$, and  $T\Champ(x_0,p_0)$ trivialised canonically by virtue of being rigid, we obtain an isomorphism
\begin{equation} \lambda(Q) \otimes  \lambda(\bR)  \cong \lambda(W^{s}(x_0)) \otimes  \ro^{\vee}_{p_0} \otimes  \lambda(Q) \otimes\ro_{p_0} \otimes \ro_{p_1}^{\vee}  . \end{equation}
Since  gluing the strip to the cap at $p_0$ gives an identification $\bR \cong T\Champbar(x_0,p_1)$ with the positive direction in $\bR$ pointing inward we recover the isomorphism of \eqref{eq:orientation_moduli_mushrooms}, multiplied by $-1$.

To conclude that \eqref{eq:chain_map_fuk_morse} holds, we must keep in mind that there are signs of $(-1)^{\deg(p_1)}$ and $(-1)^{n}$ in the definitions of the differential in the Floer and Morse theories, as well as a sign in our formula for $\cG^{1}$, which contribute $n \deg(p_1)$ in one case, and $n \deg(p_1) +n$ in the other.  

\begin{prop}
The map $\cG^{1}$ is a chain isomorphism.
\end{prop}
\begin{proof}
 We shall use the energy filtration to see that, if we order the critical points by energy, the matrix for $\cG^{1}$ is upper triangular with ones along the diagonal, hence is invertible.

We first show that $\Champ(x_0; p_1)$ is empty whenever 
\begin{equation} \label{eq:contradiction_energy} f_{i_0,i_1}(x_1) <  f_{i_0,i_1}(x_0) \end{equation}  
where $x_1$ is the critical point corresponding to $p_1$ in Corollary \ref{cor:morse_floer_generators_equal}.  Assume by contradiction that a mushroom map $\Psi$ with these asymptotic conditions exists.   Let $q$ denote the point of $Q$ such that the cap of $\Psi$ has a boundary segment mapping to $\sL_{q}$, and write $q_{i_0}$ for the lift of this point to $Q_{i_0}$.  Since the stem is a descending flow line of $f_{i_0,i_1}$, we conclude that
\begin{equation} f_{i_0,i_1}(x_0) \leq f_{i_0,i_1}(q),\end{equation}
and hence, as we are assuming \eqref{eq:contradiction_energy}, we find that
\begin{equation} f_{i_0,i_1}(x_1) < f_{i_0,i_1}(q) . \end{equation}
We claim that this inequality contradicts positivity of energy for the cap.  Indeed, an application of Stokes's theorem analogous to Lemma \ref{lem:stokes_application},  the closedness of the $1$-form in Definition \ref{lem:base_case} implies that 
\begin{equation} \label{eq:action_integral} \int |d\Psi_{P}|^{2} = \cA_{0}(q_{i_1}) + \cA_{H^{i_0,i_1}}(q_{i_0})-  \cA_{H^{i_0,i_1}}(p)  \end{equation}
where $q_{i_0}$ and $q_{i_1}$ now stand for the intersection points of $\sL_{q}$ with the respective Lagrangian, and $\cA$ for the appropriate action functional for the Hamiltonian perturbations fixed in Definition \ref{lem:base_case}. Using the primitive $\theta_{i_0}$ which vanishes on $Q_{i_0}$ and $\sL_{q}$, we compute that $\cA_{0}(q_{i_1})$ vanishes, while $\cA_{H^{i_0,i_1}}(q_{i_0})$ and $\cA_{H^{i_0,i_1}}(p)$ are respectively equal to $ -f_{i_0,i_1}(q)$ and   $-f_{i_0,i_1}(x_1)$.   As the left hand side of Equation \eqref{eq:action_integral} is non-negative, we conclude, as desired, that  $\Champ(x_0; p_1)$ is empty whenever the inequality \eqref{eq:contradiction_energy} holds.

Applying this energy analysis to the case  $p$ is the Hamiltonian chord  from $Q_{i_0}$ to  $Q_{i_1}$ corresponding to the critical point $x$ of $f_{i_0,i_1}$, we find that the only solutions in this case must have constant  cap and stem (any non-constant solution must carry some energy).   The proof that there are ones along the diagonal and hence that $\cG^{1}$ is a chain isomorphism now follows from the fact that such constant solutions are regular.
\end{proof}

The higher order analogues of this map are defined in essentially the same manner as the maps $\cF^{d}$ which define a functor from Floer to Morse theory.  Whenever $\deg(x_{0})= 1-d + \sum_{i} \deg p_{i}$,  we may assign  to each tree $\Psi \in \Champ(x_0, \vp)$ an isomorphism
\begin{equation*}
\ro_{p_1} \otimes \cdots \otimes \ro_{p_{d}} \to \ro_{x_0},
\end{equation*}
 and a sign $\ddagger(\Psi)$.  We discuss a set of possible choices for the isomorphism and the sign in Section \ref{sec:orienting_mush_maps}.  We define the $d$\th higher term of $\cG$
\begin{equation*} \cG^{d} \co CF^{*}(Q_{i_{d-1}}, Q_{i_d}) \otimes \cdots \otimes CF^{*}(Q_{i_0}, Q_{i_1})  \to CM^{*}(f_{i_0,i_d}) \end{equation*}
to be a sum of the maps $\cG^{\Psi}$ induced by $\Psi$ on orientation lines, multiplied by the appropriate sign: 
\begin{equation*}
[p_d] \otimes \ldots \otimes [p_1] \mapsto \sum_{\stackrel{\Psi \in \Champ(x_0, \vp)}{\deg(x_{0})= \sum_{k} \deg p_{k} + d -1}} (-1)^{\ddagger(\Psi)} \cG^{\Psi}([p_f] \otimes \ldots \otimes [p_1])
\end{equation*}

Modulo signs, the analysis of the boundary strata of $1$-dimensional moduli spaces of mushrooms given in Section \ref{sec:gluing_mushrooms} (i.e. the proof of Proposition \ref{prop:1_dim_moduli_space_manifold})  immediately implies:
\begin{prop}
The maps $\cF^{d}$ satisfy the $A_{\infty}$ equation for a functor:
\begin{equation} \label{a_oo-floer-morse}
 \sum_{\substack{r \\ d_1 + \ldots + d_r = d} } \mu^{\M}_{r} ( \cF^{d_1} \otimes \cdots \otimes \cF^{d_r}) =  \sum_{ \stackrel{d_1 + d_2 = d +1}{i < d_1}}  (-1)^{\maltese_{i}}\cF^{d_1}  (\id^{d_1-i-1} \otimes \mu^{\F}_{d_2} \otimes \id^{i}) \end{equation}
\noproof
\end{prop}

\section{Compactness for caps} \label{sec:comp-curv-leaves}
The goal of this section is to prove Lemma \ref{lem:cap_cpactness}.  We start in the first section by discussing a general result which constrains holomorphic curves to remain within compact domains, which we shall then apply in the second section to prove the desired statement:
\subsection{Generalities}
We begin with a generalisation of Lemma 7.2 of \cite{abouzaid-seidel}, which  precludes holomorphic curves from escaping from a domain.  Let us assume therefore that we are given an almost complex structure $I_{W}$ on a manifold $W$, and a $1$-form $\alpha$ whose differential is non-negative on every complex plane.  Let $W^{\ins}$ be a (codimension 0) submanifold with boundary such that $\ker(\alpha|  \partial W^{\ins} ) $ is a complex distribution (in particular, $\alpha$ does not identically vanish near $\partial W^{\ins}$), and 
\begin{equation} \label{eq:weak_Liouville}
 \parbox{36em}{if $\alpha(X) = 1$ for  $X \in T  \partial W^{\ins}$ then  (i) $I_{W} X$ is inward pointing and (ii) $d \alpha$ is positive on the complex plane spanned by $X$.}
\end{equation}
Assuming moreover that we are given a submanifold  $R \subset W - W^{\ins}$   transverse to $\partial W^{\ins}$ such that 
\begin{equation} \label{eq:exact_totally_real}
  \parbox{36em}{ $ \alpha $   restricts to an exact form on $R$ whose primitive is constant on $R \cap \partial W^{\ins} $,} 
\end{equation}
we conclude
\begin{lem} \label{lem:convexity_weak_assumptions}
Every $I_{W}$ holomorphic curve in $W - \inte(W^{\ins})$ with boundary on $R \cup \partial  W^{\ins} $ which intersects $ \partial W^{\ins} $ is contained therein.
\end{lem}
\begin{proof}
Since the proof is essentially indistinguishable from that of Lemma 7.2 in \cite{abouzaid-seidel} we shall only sketch it:  The assumption that $d \alpha$ is non-negative on holomorphic planes implies that its integral over a holomorphic curve is non-negative.  Moreover, given that $d \alpha$ does not vanish on a plane transverse to $ \partial W^{\ins}  $  this integral must be strictly positive for every holomorphic curve which intersects $ \partial W^{\ins} $ but is not contained in it.  In particular, Stokes's theorem implies that the integral of $\alpha$ over the boundary of such a holomorphic curve must be positive.  However, the contribution of the subset mapping to $R$ vanishes by Condition \eqref{eq:exact_totally_real}, while the contribution of the subset mapping to $\partial W^{\ins} $ is non-positive by Condition \eqref{eq:weak_Liouville}.
\end{proof}

Still requiring that $\alpha |R $ be exact, we now assume that the primitive is only locally constant on $ R \cap \partial W^{\ins}$.  In this case, let us consider maps
\begin{equation*}
 u \co \Sigma \to W - \inte(W^{\ins})
\end{equation*}
such that whenever $\gamma \co [0,1] \to \partial \Sigma$ is a parametrisation of a boundary segment mapping to $R$ which is compatible with the complex orientation,
\begin{equation} \label{eq:negative_action_components}
    \parbox{36em}{ the difference between the values of a  primitive of  $\alpha |   R$ at $u(1)$ and $u(0)$ is strictly positive.}
\end{equation}
The previously sketched argument in fact proves:
\begin{lem} \label{lem:compactness_2_boundaries}
There is no holomorphic curve in $W - \inte(W^{\ins})$ with boundary on $R \cup \partial  W^{\ins} $ for which Condition \eqref{eq:negative_action_components} holds. \qed
\end{lem}

\subsection{Proof of compactness} \label{sec:proof-compactness}
If in the discussion of the previous section, $\alpha$ is a Liouville form whose restriction to $\partial W^{\ins}$ is a contact form, then Condition \eqref{eq:weak_Liouville} essentially says that $I_{W}$ is of contact type near $\partial W^{\ins}$, and Condition \eqref{eq:exact_totally_real} that $R$ is exact with Legendrian boundary (and in addition admits a primitive for the Liouville form which strictly vanishes on this boundary).  In this case, we exactly recover Lemma 7.2 of \cite{abouzaid-seidel}, and we may conclude half of Lemma \ref{lem:cap_cpactness}:
\begin{lem} \label{lem:half_of_compactness_caps}
If the distance between $q$ and $b$ is greater than $\delta$ then $\Pilbar_{q}(\vx)$ is compact whenever $i_0=i_d$ and empty whenever $(i_0,i_d) = (2,1)$.
\end{lem}
\begin{proof}
As usual, it suffices to prove the existence of a compact subset of $M$ (in this case $M^{\ins}$) containing  the image of every element of $ \Pil_{q}(\vx) $, so that the usual proof of Gromov compactness applies.  Let $u \co S \to M$ be such a map, and denote by $\Sigma$ the inverse image of $M - \inte(M^{\ins})$.  If $i_0 = i_d$, then the hypothesis of Lemma \ref{lem:convexity_weak_assumptions} holds if $\alpha$ is a primitive for the symplectic form which agrees with the natural primitive near $M^{\ins}$, and the image of $u|\Sigma$  is contained in $\partial M^{\ins}$, hence the image of $u$ is contained in $M^{\ins}$.  If $(i_0, i_d)=(2,1)$, then the computation of Lemma \ref{lem:stokes_application_perturbed} shows that Condition \eqref{eq:negative_action_components} holds, so that Lemma \ref{lem:compactness_2_boundaries} implies the desired result.
\end{proof}

\begin{figure}
  \centering
  \epsfxsize=2in
\epsffile{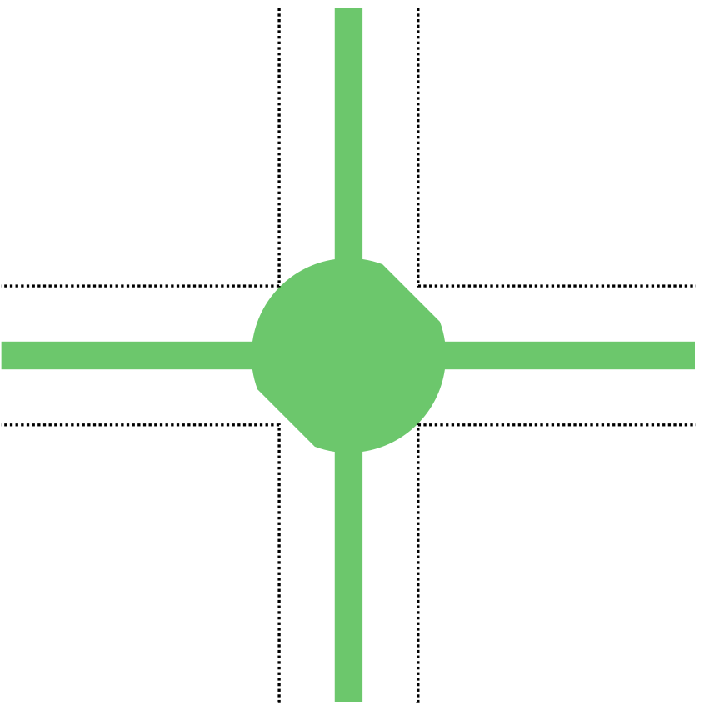} 
  \caption{ }
  \label{fig:middle_manifold}
\end{figure}
Note that while the almost complex structure is of contact type near any point in $\partial M^{\ins}$, the Lagrangians $\sL_{q}$ may not intersect it in Legendrian submanifolds if $q$ is within $\delta$ of $b$.  We shall construct a different submanifold $M^{\mid}$ which will play the same role as $M^{\ins}$ did in the previous Lemma (this manifold was mentioned in Condition \eqref{eq:allowed_J-H} and we are finally about to define it).   Its local model in $\bC^{n}$ is  a smoothing of the manifold with corners
\begin{equation} \label{eq:M_mid_corners}
  \{(\vx,\vy) | |\vx|^{2} \leq 1/4  \textrm{ or } |\vy|^{2} \leq 1/4 \} \cup   \{(\vx,\vy) | |\vy|^{2} +  |\vx|^{2}  \leq 2  \textrm{ and } |\vy + \vx|^{2}  \leq 4 - \delta \},
\end{equation}
see Figure \ref{fig:middle_manifold}.  To understand the construction, it is helpful to keep in mind that the sphere of radius $2$ about the origin meets the hypersurface $ |\vy + \vx|^{2}  = 4 $ precisely on its intersection with the diagonal.  Hence this intersection takes place far away from the hypersurfaces $  |\vx|^{2} = 1/4  $ and $  |\vy|^{2} = 1/4 $.  In particular, this manifold has only codimension $2$ corners, and the smoothing construction which produces  $M^{\mid}$  is entirely local.  We leave the details of the smoothing to the reader, requiring only that
\begin{equation*}
  \parbox{36em}{if the distance between $b$ and $q$ is smaller than $\delta$, the intersection of  the leaf $  \sL_{q} $ with $\partial M^{\mid}$ agrees with its intersection with the hypersurface $ |\vy + \vx|^{2}  = 4 - \delta $.}
\end{equation*}
Note that  $\{(\vx,\vy) | |\vy + \vx|^{2}  \leq 4 - \delta \}$ can be naturally identified with a disc cotangent bundle of the Lagrangian 
\begin{equation*}
  L_{- \Delta} = (\vx, - \vx) \subset \bC^{n}
\end{equation*}
such that affine planes parallel to $L_{\Delta}$ correspond to cotangent fibres.  In particular, the leaves $\sL_{q}$  (whenever $q$ is close to $b$) intersect $M^{\mid}$ in Legendrians.  With this at hand, we complete this Section with the
\begin{proof}[Proof of Lemma \ref{lem:strips_moduli_spaces_cpact}]
Given that we have already established Lemma \ref{lem:half_of_compactness_caps} it remains to show that caps with boundary conditions along $\sL_{q}$ cannot escape $ \partial M^{\mid} $ whenever $q$ is within $\delta$ of $b$.  Condition \eqref{eq:allowed_J-H} implies that all almost complex structures are of contact type near $\partial M^{\mid} $, and that there are no Hamiltonian perturbations in $M - M^{\mid}$; as in Lemma \ref{lem:half_of_compactness_caps}, we may directly apply Lemma \ref{lem:convexity_weak_assumptions} to conclude the inverse image of the complement of $M^{\mid}$ is empty.
\end{proof}

\section{Transversality}

\subsection{Transversality and gluing for gradient trees}  \label{app:stasheff}
In this section, we explain the proofs of results from Sections \ref{sec:stasheff} and \ref{sec:simp_to_morse}. 

\subsubsection{Transversality for Stasheff trees}

In order to prove transversality, we must study the moduli space $\Stasheffbar_{d}$ more carefully.  First, we recall that its boundary strata  are labelled by trees with $d$ inputs
\begin{equation} T^{\infty} \end{equation}
An element of $\Stasheffbar_{d}$ lies in the cell labelled by $T^{\infty}$ if the result of collapsing all its finite edges yields a (ribbon) tree combinatorially isomorphic to $T^{\infty}$, see \cite{seidel-book} which gives a detailed discussion of the parallel case in which the Stasheff polytopes are modeled by moduli spaces of holomorphic discs.  The space $\Stasheffbar_{d}$ also has a cell decomposition with cells $\Stasheff^{T \to T^{\infty}}$ labelled by surjective maps of ribbon trees 
\begin{equation}
 T \to T^{\infty}
\end{equation}
where both  $T$ and $T^{\infty}$ are isomorphism classes of trees with $d$ inputs, and such that the inverse image of every vertex is a connected tree. The stratum determined by such a map consists of trees which are homeomorphic to $T$ and whose finite edges are those collapsed in $T^{\infty}$.  In particular, the set of infinite edges in such a stratum is fixed, so the structure of a metric tree is determined by the lengths of the remaining edges, and we have a canonical isomorphism
\begin{equation} \Stasheff^{T\to T^{\infty}} \cong (0,+\infty)^{|\cE(T)|-|\cE(T^{\infty})|}, \end{equation}
with closure (allowing edges to reach length $0$ or $\infty$):
\begin{equation} \Stasheffbar^{T\to T^{\infty}} \cong [0,+\infty]^{|\cE(T)|-|\cE(T^{\infty})|}. \end{equation}
The cell structure on $\Stasheffbar^{T\to T^{\infty}}$ induced from its inclusion in $\Stasheffbar_{d}$ is compatible with the natural cell structure on $[0,+\infty]^{|\cE(T)|-|\cE(T^{\infty})|}$.  Given a sequence $\vI$ of elements of the set $\{1,2\}$, there is an obvious meaning to the moduli space 
\begin{equation}
 \Stasheffbar^{T \to T^{\infty}}_{\vI}
\end{equation}
of labelled Stasheff trees which lie in the appropriate stratum.

We shall now formally describe what we mean by a smooth choice of perturbation data.  Let us once and for all fix a smooth structure on $[0,+\infty]$ using the identification
\begin{equation}
e^{-x} \co [0,+\infty] \cong [0,1].
\end{equation}
 In particular, we obtain a smooth structure on $\Stasheffbar^{T \to T^{\infty}}$. Note that every flag in $T$ (a pair $(e,v)$ with $v$ a vertex of the edge $e$) and the metric $g_T$ determine a unique orientation and length preserving map
\begin{equation} \label{eq:map_tree_R}
 e \to \bR
\end{equation}
taking $v$ to the origin and $e$ to the appropriate segment on either side of the origin.  In particular, this flag determines a polyhedral subset in
\begin{equation} 
[0,+\infty]^{|\cE(T)|-|\cE(T^{\infty})|} \times \bR
\end{equation}
whose fibre is the image of the segment $e$ under \eqref{eq:map_tree_R}.
\begin{defin} \label{defin:smooth_pert_data_stasheff}
A choice of perturbation data for each element of $\Stasheffbar_{\vI}$ is said to be {\bf smooth} if, given a flag $(e,v)$ of a tree $T$, the perturbation data for the edge $e$ of metric trees in $\Stasheffbar^{T\to T^{\infty}}_{\vI}$  extends to a smooth map
\begin{equation} \label{eq:extended_perturbation_data_stasheff}
\tilde{\bfX}^{T \to T^{\infty}}_{(e,v) } \co [0,+\infty]^{|\cE(T)|-|\cE(T^{\infty})|} \times \bR \to \C^{\infty}(TM).
\end{equation}
\end{defin}

Recall that a universal perturbation is a smooth choice of perturbation data for all labelled Stasheff trees, which is moreover consistent in the sense that the perturbation data for a singular tree is given by the data of its components. 
\begin{lem} \label{lem:inductive_construction_pert_data_stasheff}
Given any sequence $\vI$, every smooth and consistent choice of perturbation data on the union of
\begin{equation}
\coprod_{|\vI[']| < |\vI|} \Stasheffbar_{\vI[']} 
\end{equation}
with the $k$-skeleton of $\Stasheffbar_{\vI}$ extends smoothly and consistently to the $k+1$ skeleton of $\Stasheffbar_{\vI}$. Moreover, the restriction map from the space of all perturbation data to the space of perturbation data for a given singular labelled Stasheff tree $(T,g_{T})$ is surjective.
\end{lem}
\begin{proof}
Since the perturbation data have been fixed on the lower dimensional moduli spaces it suffices to extend the perturbation datum to interior cells.  Fix a tree $T$ labelling such a cell, (in this case $T^{\infty}$ is a tree with a unique $d+1$-valent vertex, and all interior edges are collapsed, so we may drop this data from the notation).  By assumption, perturbation data have been chosen for the boundary of the cell labelled by $T$.  If we fix an interior edge $e$, writing $(e,v_{\pm})$ for the two flags it lies on, then the existence of a smooth extension of a function from the boundary of 
\[ [0,+\infty]^{|\cE(T)|-|\cE(T^{\infty})|} \times \bR \]
is obvious since this is a manifold with corners (we proceed by induction, extending over small neighbourhoods of the skeleta).  It suffices therefore to make sure that the extension can be chosen so the maps associated to the two flags containing $e$ are both smooth.  Writing $t_{e}$ for the coordinate of the factor of  $[0,+\infty]^{|\cE(T)|-|\cE(T^{\infty})|}$ corresponding to the edge $e$, we find that the only potentially interesting extension problem occurs at the boundary facets $t_{e}= 0, +\infty$.  Since the map $(t_{e},y) \mapsto (t_e,y+t_{e})$ is smooth on $\bR^{2}$, and takes the polyhedron associated to $(e,v_-)$ to the one associated to $(e,v_+)$, smoothness of the extensions $\tilde{\bfX}^{T \to T^{\infty}}_{(e,v_-)}$ and $ \tilde{\bfX}^{T \to T^{\infty}}_{(e,v_+)}$ are equivalent near $t_{e} =0$.

In remains to check the existence of an extension from $t_{e}=+\infty$.  The key fact is that the perturbation datum is required to be compactly supported on every edge. This defines a naive way of gluing the perturbation data associated to the flags  $(e,v_{\pm})$ whenever $t_{e} \gg 0$ by ``pregluing'' the two half-infinite intervals as follows: if we identify the edge with the interval $[0,t_e]$, then the perturbation datum on $[t_{e}/3, 2t_{e}/3]$ is set to vanish, while those on $[0,t_{e}/3]$ and $[2t_{e}/3,t_{e}]$ are obtained by pulling back the perturbation datum from the natural inclusions in $[0,+\infty)$ and $[-\infty,0]$ (the second is obtained after translation by $-t_e$).  A simple computation shows that the resulting perturbation datum extends smoothly with respect to both coordinate systems, and the proof proceeds by induction using cutoff functions, proving the first part of the result. The second part follows easily by observing that once one has constructed a perturbation datum over the boundary of a cell, it can be arbitrarily modified at any point in its interior.  
\end{proof}

The next technical result is the infinitesimal version of the fact that  perturbing of the gradient flow equation on a bounded  subset of an edge integrates to an essentially arbitrary diffeomorphism (see the discussion preceding Definition \ref{defin:pert_datum}) .
\begin{lem}
For each gradient tree $\psi \co (T,g_T) \to Q$, $x$ a point in the appropriate ascending or descending manifold, and tangent vector $X$ at $ev_{k}(x)$, there exists a $1$-parameter family of universal perturbation data $\bfX^{\Stasheff}_{\tau}$ whose evaluation maps $\ev^{\tau}_{k'}$ are independent of $\tau$ for $k' \neq k$, and such that
\begin{equation}
\frac{d \ev_{k}^{\tau}(x)}{d \tau} {\Big|}_{\tau = 0} = X.
\end{equation}
\noproof
\end{lem}

With this in mind, we can prove the desired transversality result.  
\begin{lem} \label{lem:stasheff_smooth}
 All moduli spaces $\Stasheffbar(x_0,\vx)$ are manifolds of the expected dimension for a generic choice of universal perturbation data.
\end{lem}
\begin{proof}[Sketch of proof]
If $\vx$ consists of a single element  $x_1$, this follows from the assumption that all functions $f_{i,j}$ are Morse-Smale.  To prove the general case, note that the choice of a perturbation datum associates a diffeomorphism of one of the manifolds $Q_{i,j}$ to each finite edge of a metric tree $(T,g_T)$; the diffeomorphism is simply the flow of the associated perturbed gradient vector field.  By picking finite segments in each incoming edge containing the support of the perturbation datum, we may also associate a diffeomorphism (the perturbed gradient flow on this segment) to a leaf of $T$.  In particular, we define for each integer $0 \leq k \leq d$ a diffeomorphism $\phi_{k}$ obtained by composing the diffeomorphism attached to the $k$\th leaf with the diffeomorphisms associated to the sequence of finite edges which form an arc from the $k$\th leaf to the outgoing segment.  The moduli space of gradient trees for a fixed metric tree is the inverse image of the diagonal in $Q^{d+1}$ under a map
\begin{equation} \label{eq:fibre_product_moduli_space}
\ev \co W^{s}(x_0) \times W^{u}(x_1) \times \cdots \times W^{u}(x_d) \to Q^{d+1}
\end{equation}
whose $k$\th component $\ev_{k}$ is the composition of the inclusion of the given stable or unstable manifold in the appropriate factor with the diffeomorphism $\phi_{k}$. A standard argument using Sard's theorem then reduces the regularity of the moduli space to the previous Lemma.

\end{proof}
\begin{rem}
Lemma \ref{lem:compactness_stasheff} implies that we may discuss transversality of the evaluation map \eqref{eq:fibre_product_moduli_space} to the diagonal even though the target space is not Hausdorff; the image of a point under consideration may only lie in the singular strata of $Q$ if it is labelled by $Q_{1,1}$ or $Q_{2,2}$, in which case all the arguments take place in the appropriate manifold.
\end{rem}

\subsubsection{Transversality for shrubs}
We briefly explain how to extend the techniques of the previous section to obtain moduli spaces of maps of shrubs which are smooth manifolds.  The main difference are that the strata of $\Shrubbar_{d}$ are not manifolds with corners, so one may not extend an arbitrary smooth function from one skeleton to the next. The quickest way to resolve this problem is to consider the inclusion
\begin{equation} \Shrubbar_{d} \to \Stasheffbar_{2d} \equiv \Shrubbarext_{2d} \end{equation}
obtained by attaching two infinite edges to every incoming vertex of a shrub.  We insist that this is a purely formal construction; the right hand 
side should not be thought of as controlling gradient trees for Morse theory in this construction, and so we have given it a different name.  It is clear that the natural stratification of  $\Shrubbar_{d}$ agrees with the  pullback  of the stratification of $\Stasheffbar_{2d}$ under the inclusion.

Given a sequence $\vI$, we define $2\vI$ to be the sequence obtained by repeating every element but the last, allowing us to define the analogue of this inclusion in the presence of labels
\begin{equation}\label{eq:strange_inclusion} \Shrubbar_{\vI} \to \Stasheffbar_{2\vI}  \equiv \Shrubbarext_{2\vI} .\end{equation}
The reader may easily check the commutativity of the diagram,
\begin{equation}\label{eq:gluing_shrubs_is_gluing_trees}
\xymatrix{  \Stasheffbar_{\vR} \times \Shrubbar_{\vI[1]} \times \ldots \times \Shrubbar_{\vI[r]}  \ar[r] \ar[d]  & \Shrubbar_{\vI} \ar[d] \\ 
\Stasheffbar_{\vR} \times \Stasheffbar_{2\vI[1]} \times \ldots \times \Stasheffbar_{2\vI[r]}  \ar[r] \ar[d]^{\equiv} & \Stasheffbar_{2 \vI} \ar[d]^{\equiv} \\
\Stasheffbar_{\vR} \times \Shrubbarext_{2\vI[1]} \times \ldots \times \Shrubbarext_{2\vI[r]} \ar[r] & \Shrubbarext_{2\vI}}
\end{equation}
with vertical arrows obtained by applying \eqref{eq:strange_inclusion}, and the horizontal ones by applying the map \eqref{eq:boundary_shrub_labelled-tree} in the first row, and $r$ iterations of the map \eqref{eq:compose_stasheff_labels} in the second.  Note that the map  \eqref{eq:compose_stasheff_labels} can be interpreted in our new notation as the existence of a map
\begin{equation} \label{eq:structure_map_extended_shrubs}
\Stasheffbar_{\vI_1}  \times \Shrubbarext_{\vI_2}  \to \Shrubbarext_{\vI}
\end{equation}
where $\vI$ is obtained by removing one of the inputs of $\vI_1$ and replacing it by the inputs of $\vI_{2}$.

\begin{defin} \label{defin:smooth_pert_data_shrub}
A {\bf universal perturbation datum on $\Shrubbarext$} is a perturbation datum for every tree in $\Shrubbar_{\vI}$ varying smoothly, and which is compatible with the perturbation datum $\bfX^{\Stasheff}$ used to define $\Morse(Q_1,Q_2)$ and the maps \eqref{eq:structure_map_extended_shrubs}.

A choice of perturbation data for each element of $\Shrubbar_{\vI}$ is said to be {\bf smooth} if it is the restriction of a universal perturbation datum on $\Shrubbarext$ under the map \eqref{eq:strange_inclusion}.
\end{defin}
\begin{rem}
Of course, a universal perturbation datum on $\Shrubbarext$ involves many choices which are discarded by restricting to $\Shrubbar_{\vI}$.  One might reasonably a universal perturbation datum only on neighbourhoods of the image of $\Shrubbar_{\vI}$ which are compatible with the operadic structure maps, but there seems to be no reason to desire this additional level of complexity.  
\end{rem}

The existence of a universal perturbation datum on $\Shrubbarext$ is proved by induction as we did in the construction of $\bfX^{\Stasheff}$.  The induction step uses the fact that any smooth function on a subcomplex of the natural stratification of a manifold with corner may be extended to an additional cell.

The commutativity of the diagram \eqref{eq:gluing_shrubs_is_gluing_trees} implies that the compatibility of a perturbation datum with respect to the structure maps of $\Shrubbarext$ implies compatibility of the restriction to $\Shrubbar$ with \eqref{eq:boundary_shrub_labelled-tree}.   Lemma \ref{lem:simplicial_perturbation_shrubs_gives_cup} implies that the simplicial condition is realised in an open set of such perturbations.  To conclude the desired transversality result, the key ingredient, as in the proof of Lemma \ref{lem:stasheff_smooth} is an intersection-theoretic interpretation of $\Shrubbar_{S,g_{S}}(x,\vsig)$ whenever $\vsig = (\check{\sigma}_{1}, \ldots, \check{\sigma}_{d})$ is a composable sequence of cells with $\sigma_{k}$ a cell of $\sQ_{i_{k},i_{k+1}}$, and $x$  is a critical point of $f_{0,d}$.  As we did for Stasheff trees in the proof of Lemma \ref{lem:stasheff_smooth}, this moduli space is the inverse image of the diagonal under a map
\begin{equation} \label{eq:fibre_product_picture_shrubs}
 W^{s}(x) \times \check{\sigma}_{1} \times \cdots \times \check{\sigma}_{d} \to Q^{d+1}
\end{equation}
whose $k$\th  factor is obtained by composing the inclusion of $\check{\sigma}_{k}$ with the composition of perturbed gradient flows associated to the edges of $S$ lying on an arc from the $k$\th incoming leaf to the outgoing edge.  We conclude:

\begin{lem}
For generic universal perturbation data, all strata of $\Shrubbar(x,\vsig)$ are smooth manifolds of the expected dimension. \noproof
\end{lem}

\subsubsection{Codimension $1$ gluing for trees}

We briefly describe how to use the approach of \cite{KM} in order to prove the desired gluing theorems for moduli spaces of gradient trees.  The results of this section prove that the compactifications of moduli spaces of gradient trees are manifolds with boundary whenever they have virtual dimension $1$.    However, as we shall need a description of a neighbourhood of the codimension $1$ strata of moduli spaces of shrubs of arbitrary dimension, we shall discuss matters in that generality.

Consider a Morse-Smale function $f$ on a smooth submanifold, and $N_1$ and $N_2$  smooth manifolds which are transverse to the gradient flow of $f$.  The manifold of gradient flow lines of length $R$ starting on $N_1$ and ending on $N_2$ will be denoted $\Grad^{R}(N_1,N_2) $.  The union of these spaces admits a compactification $\Gradbar(N_1,N_2) $ by allowing broken trajectories.  We write
\begin{equation} \Grad^{\infty}(N_1,x,N_2)  \equiv W^{s}(x) \times W^{u}(x) \cap N_1 \times N_2 \subset Q \times Q \end{equation}
for broken gradient trajectories passing through a critical point $x$.  Note that all strata of virtual codimension $1$ in $\Gradbar(N_1,N_2) $ are of this type.  The following result is a minor generalisation of Proposition 18.1.4 in \cite{KM}:
\begin{lem} \label{lem:KM}
Each point $(\gamma_1,\gamma_2) \in  \Grad^{\infty}(N_1,x,N_2) $ admits a neighbourhood $U(N_1,x,N_2)$, and an embedding 
\begin{equation} U(N_1,x,N_2) \times (R_0, +\infty]  \to \Gradbar(N_1,N_2) \end{equation}
for $R_0 \gg 0$, such that the image of $U(a_1,x,a_2) \times \{ R \}$ consists of all gradient segments of length $R$ which are sufficiently close to  $(\gamma_1,\gamma_2)$.  Moreover, the evaluation map
\begin{equation}U(N_1,x,N_2) \times \{ R \} \to  Q \times Q \end{equation}
converges in the $\C^{\infty}$ topology as $R$ converges to infinity. \noproof
\end{lem}

Assuming that the perturbation data we have chosen are generic, this result implies that the moduli spaces $\Stasheffbar(x_0, \vx)$ are manifolds with boundary, with boundary stratified by smooth manifolds of the expected dimension.  For simplicity of notation (and as this is the only case that is needed for our arguments), we focus on a neighbourhood of the strata of codimension $1$.  Such a stratum is a product 
\begin{equation} \Stasheff(x_0, \vx[1]) \times \Stasheff(y, \vx[2]) \end{equation} 
with a fixed choice of an element $x_{j}^{1} \in \vx[1]$ which equals $y$, and such that $\vx$ is the result of replacing $x_{j}^{1}$ by $\vx[2]$ in $\vx[1]$.  The outgoing and incoming edge with endpoint $y$ are naturally isometric to $[0,+\infty)$ and $(-\infty,0]$.  We fix a constant $K$ sufficiently large that the support of the perturbation data on these edges are respectively contained in $[0,K]$ and $[-K,0]$, and write $b_{K}$ for the image of $\pm K$ on this edge.  We define evaluation maps
\begin{align*} \Stasheff_{d_1} \times W^{s}(x_0) \times W^{u}(x_{1}^{1}) \times  \cdots \times  W^{u}(x_{j-1}^{1}) \times  W^{u}(x_{j+1}^{1}) \times \cdots \times  W^{u}(x_{d_1}^{1}) & \to Q^{d_1} \\
 \Stasheff_{d_2} \times W^{u}(x_{2}^{1}) \times  \cdots \times  W^{u}(x_{2}^{d_1}) & \to Q^{d_2}
 \end{align*}
which are obtained in the second case by composing the inclusion of  $W^{u}(x_{2}^{j})$ into the $j$\th factor with the flows associated with a descending arc from the $j$\th leaf to the point $b_k$ (the first case is a straightforward modification).  We write $N^{\infty}(x_0, \vx[1] - y)$ and $N^{\infty}(\vx[2])$ for the inverse images of the diagonal under these maps.  By definition, the stratum we are considering is given by
\begin{equation}  \Grad^{\infty}(N^{\infty}(x_0, \vx[1] - y),y,N^{\infty}(\vx[2]) ). \end{equation}
Our regularity assumption can be restated as the fact that this is a transverse fibre product.

Consider the ``pregluing'' map
\begin{equation} (0, +\infty] \times \Stasheff_{d_1} \times \Stasheff_{d_2} \to  \Stasheffbar_{d} \end{equation}
obtained by replacing the bi-infinite edge in the source with a finite one whose length is given by the projection to the first factor in the source.    Let us define 
\begin{equation} N^{R}(x_0, \vx[1] - y) \textrm{ and } N^{R}(\vx[2]) \end{equation}
for $R \in [0,+\infty)$ in the same way as their analogue for $R= +\infty$, but using the perturbation data on $\Stasheffbar_{\vI}$ pulled back under the inclusion of  $\{ R \}  \times \Stasheff_{\vI[1]} \times \Stasheff_{\vI[2]}$.  A neighbourhood of $\Stasheff(x_0, \vx[1]) \times \Stasheffbar(y, \vx[2])$ in $\Stasheffbar(y, \vx)$ consists of gradient trees whose source lies in the image of this pregluing map, and is therefore given by the union
\begin{equation} \coprod_{0 \leq R  \leq +\infty} \Grad^{R}(N^{R}(x_0, \vx[1] - y),y,N^{R}(\vx[2]) ). \end{equation}
The smoothness condition on our choice of perturbation data implies that the evaluation maps on $N^{R}(x_0, \vx[1] - y)$ and $N^{R}(\vx[2])$ converge on compact subsets, and in the $C^{\infty}$ topology to the evaluation maps at $R=+\infty$.  Together with Lemma \ref{lem:KM}, elementary differential topology implies
\begin{lem}
Each point $(\psi_1,\psi_2)   \in  \Stasheff(x_0, \vx[1]) \times \Stasheff(y, \vx[2])$  admits a neighbourhood $U(\psi_1,\psi_2)$, and an embedding 
\begin{equation} U(\psi_1,\psi_2) \times  (R_0, +\infty]  \to \Stasheffbar(x,\vx) \end{equation}
for $R_0 \gg 0$, whose image is a neighbourhood of $(\psi_1,\psi_2)$ in the Gromov topology.  \noproof
\end{lem}
This result has a straightforward generalisation to corners of higher codimension, but we shall not need that fact in this paper.

The same proof technique works for shrubs:  Let $\psi_0$ be a rigid gradient tree in $\Stasheff(x_0, \vx)$, and $\vpsi = ( \psi_1, \ldots, \psi_{d})$  a sequence of shrub maps with outputs $x_j$.  We write $(\vsig[1], \ldots, \vsig[r])$ for the collections of inputs of all shrub maps $\psi_{j}$, and $\vsig$ for the sequence obtained by concatenating the sequences $\vsig[j]$. By attaching $\psi_{j}$ to the $j$\th input of $\psi_0$, we obtain a singular shrub
\begin{equation} (\psi_0, \vpsi) \in  \Shrubbar(x_0, \vsig). \end{equation} 
Let $U(\psi_0, \vpsi)$ denote a sufficiently small neighbourhood of this point in the boundary stratum of $ \Shrubbar(x_0, \vsig) $ wherein it lies (note that his stratum has virtual codimension $1$).  Applying the discussion for trees to this situation, we conclude:
\begin{lem} \label{lem:gluing_shrubs}
There exists an embedding
\begin{equation}  U(\psi_0, \vpsi) \times (R_0, +\infty] \to   \Shrubbar(x_0, \vsig)   \end{equation}
whose image is a neighbourhood of $(\psi_0,\vpsi)$ in  $ \Shrubbar(x_0, \vsig) $. \noproof
\end{lem}

\subsection{Transversality and gluing for caps and mushrooms} \label{app:transversality_cap}
The proofs of technical results from Section \ref{sec:floer_to_morse} are contained in this section.  However, we omit the analytical parts of the proofs, so that we work with smooth functions throughout, without setting up the relevant Banach spaces.
\subsubsection{Transversality for caps}

In this section, we prove Lemma \ref{cap_transversality}.
One way to define the topology on $\Pil(\vp)$, is to introduce the space
\[ \C^{\infty} \Pil( \vp) \]
of all smooth maps $u \co P \to M$ for some cap $P$ whose boundary arcs are mapped to $(Q_{i_0}, \cdots, Q_{i_d}, \sL_q)$, where  $q$ is allowed to vary arbitrarily in $Q$, and whose incoming marked points are mapped to the Hamiltonian chords $\vp$.  Note that $\Pil(\vp)$ is simply the zero-level set of 
\begin{equation} \label{perturbed_d_bar} (du - Y^{\vI})^{0,1} \end{equation}
as a section of the vector bundle over $ \C^{\infty} \Pil( \vp)$ with fibre
\[  \C^{\infty}(P, \Omega^{0,1}(P) \otimes u^{*} T M) \]
at a map $u$ with source $P$.  
\begin{proof}[Proof of Lemma \ref{cap_transversality}]
The proof is a standard application of Sard-Smale to pseudo-holomorphic curve theory.  We note that the assumption that all moduli spaces of holomorphic discs $\cR(p_0,\vp)$ are regular implies that every stratum consisting of one cap and several discs will have the expected dimension as soon as the component corresponding to the cap is regular.  We introduce the ingredients necessary to prove that this is the case for generic perturbation data:

Fix $q \in Q$, and let $u \co P \to M$ be a smooth map .  We define
\[ T_{u} \C^{\infty} \Pil( \vp) \]
to be the subspace of $  \C^{\infty}(P, u^{*}(T M)) $ consisting of vector fields whose restriction to a boundary segment of $P$ labelled by $Q_{i}$ is valued in the tangent space $Q_i$, and whose restriction to the outgoing boundary segment projects to a constant vector in $Q$ (note that the foliation $\sL_{q}$ defines a smooth map $M \to Q$).  The projection map
\begin{equation} \label{projection_restricted_tangent} T_{u} \C^{\infty} \Pil( \vp) \to T_{q}Q \end{equation}
is obviously surjective with kernel consisting of vector fields $X$ which, on each arc point in the direction of the Lagrangian label.  We write
\begin{equation} \C_{\res}^{\infty}(P, u^{*}(T M)) \end{equation}
for this space.

Let $\sT^{\Pil}$ denote the space of infinitesimal deformations of the universal perturbation data $\bfK^{\Pil}$ and $\bfJ^{\Pil}$, and let $\sT_{P}$ denote the analogous space for fixed $P$.  By taking the direct sum of the linearized $\bar \partial $ operator and the projection above, we obtain a map
\begin{equation} \sT_{P} \oplus T_{P} \Pil \oplus  T_{u} \C^{\infty} \Pil( \vp) \to \C^{\infty}(P, \Omega^{0,1}(P) \otimes u^{*} T M) \oplus T_{q}Q.\end{equation}
Note that the surjectivity of this operator implies the desired result by the Sard-Smale theorem, and the infinitesimal version of Lemma \ref{lem:extend_pert_data_caps} which shows that $\sT^{\Pil}$ surjects onto $\sT_{P}$.

The surjectivity of this operator itself follows from the surjectivity of the linearized $\bar \partial$ operator restricted to the kernel of the projection to $T_{P} \Pil \oplus T_{q} Q$:
\[ \sT_{P}\oplus  \C_{\res}^{\infty}(P, u^{*}(T M)) \to \C^{\infty}(P, \Omega^{0,1}(P) \otimes u^{*} T M).\]
This fact, however, is a standard result in pseudo-holomorphic curve theory.  For closed holomorphic curves, this appears in the textbook \cite{MS}*{Proposition 3.2.1} (the Lagrangian case can also be found, for example, in \cite{seidel-les}*{Lemma 2.4}).
\end{proof}

The next result is the key to our proof of the gluing theorem for mushrooms:  Let $u$ be an element of a moduli space of caps $\Pil(\vp[1])$ of dimension $k$ such that the evaluation map
\begin{equation} \Pil(\vp[1]) \to Q \end{equation}
is an immersion near $u$; let $U$ denote a neighbourhood in $\Pil(\vp[1])$ where this property holds.  Let $v$ be a rigid element of a moduli space of discs $\cR(p, \vp[2])$ such that $p$ agrees with the $j$\th input of $\vp[1]$.  Denoting by $\vp$ the result of replacing this input in $\vp[1]$ with $\vp[2]$, we write $(u,v)$ for the singular cap in $\Pilbar(\vp)$ obtained by attaching $v$ to $u$.  The next result is standard gluing theorem in pseudo-holomorphic curve theory:

\begin{lem} \label{lem:gluing_caps}
There exists an embedding
\begin{equation}[0,1) \times  U \to  \Pilbar(\vp) \end{equation}
whose image is a neighbourhood of $(u,v)$ in  $\Pilbar(\vp)$, and whose restriction to $\{\epsilon \} \times U$ is smooth for any value of $\epsilon$.  Moreover, the composition of the evaluation map
\begin{equation}  \Pilbar(\vp) \to Q  \end{equation}
with the inclusion of $\{ \epsilon \} \times U$ converges in the $C^1$ topology to its value at $\epsilon = 0$.
\end{lem}
\begin{proof}[Sketch of proof:]
The key ingredient in the construction of a gluing map is a choice of right inverse.  Let $q$ denote the image of $u$ under the evaluation map, and fix a neighbourhood of $q$ with a diffeomorphism to a neighbourhood of the origin in $\bR^{n}$ such that the image of $U$ under the evaluation map is the intersection of this subset with $\bR^{k}$. 

We consider the codimension $n-k$ subspace
\begin{equation} \C_{\bR^{n-k}}^{\infty}(P, u^{*}(T M)) \subset \C_{\res}^{\infty}(P, u^{*}(T M)) \end{equation}
consisting of vector fields whose image under the projection map to $TQ$ and the local diffeomorphism with $\bR^{n}$ lie in the copy of $\bR^{n-k}$ complementary to the image of the moduli space  under evaluation.  By construction, the restriction of the parametrized (perturbed) $\dbar$ operator to
\begin{equation}  T_{P} \Pil \oplus  \C_{\bR^{n-k}}^{\infty}(P, u^{*}(T M)) \end{equation}
is an isomorphism.

This choice of right inverse determines a gluing construction as follows: since $v$ is rigid its $\dbar$ operator is an isomorphism.  For every real number $\epsilon$, one constructs a ``preglued'' map $u \#_{\epsilon} v$ whose source is a surface obtained by removing parts of the strip-like ends of the domains of $u$ and $v$ and attaching the complements.  This preglued map comes equipped with a right inverse also obtained by gluing the right inverses on both side.  Applying the implicit function theorem to this right inverse, we conclude the existence of a unique vector field $\sol_{u\#_{\epsilon} v}$ along the pre-glued curve, in the image of the right inverse and of sufficiently small norm,  whose image is an honest solution to the $\dbar$ problem being studied.  The standard argument in holomorphic curve theory shows that the norm of  $\sol_{u\#_{\epsilon} v}$ (in any reasonable space of functions) decays as $\epsilon$ goes to $0$.

The same construction can be done for curves $u'$ near $u$, and it is easy to see that the preglued curve varies smoothly with $u$ and $\epsilon$, from which $C^0$ convergence of the evaluation map follows. In order to prove $C^1$ convergence, we must therefore prove that the vector field  $\sol_{u\#_{\epsilon} v}$ decays with $\epsilon$ in the $C^1$ topology as a function of $u$.  The easiest way to do this is to use a right inverse at $u' \#_{\epsilon} v$ obtained by parallel transport of the one constructed by pregluing at $u \#_{\epsilon} v$.  Decay in the $C^1$ topology follows from a standard version of the implicit function theorem applied to a chart for the relevant space of functions which is ``centered at $u \#_{\epsilon} v$'' (see \cite{floer}*{Proposition 24} for the version of the implicit function theorem one should use, and \cite{abouzaid-exotic}*{Proposition 5.4} for a detailed version of this argument in the setting of bubbling of holomorphic discs).
\end{proof}

\subsubsection{Transversality for mushrooms} \label{sec:transverse_mushrooms}
In this section, we prove the existence of universal perturbation data for which the moduli spaces of mushrooms are smooth manifolds.  We adopt the strategy used  for shrubs and consider the disjoint union over sequences with $\sum_{j} d_j = d$:
\begin{equation}  \label{eq:extended_mushrooms} \coprod \Shrubbarext_{2r} \times \Pilbar_{d_1} \times \cdots \times \Pilbar_{d_r} .\end{equation} 
Recall that we constructed $\Champbar_{d}$ starting with a disjoint union of polyhedra each admitting a natural embedding:
\begin{equation} \label{eq:inclusion_mushrooms_into_extended} \Shrubbar_{2r} \times \Pilbar_{d_1} \times \cdots \times \Pilbar_{d_r} \subset \Shrubbarext_{2r}  \times  \Pilbar_{d_1} \times \cdots \times \Pilbar_{d_r}.  \end{equation} 

The moduli space $\Champbar_{d}$ was defined as a quotient under an equivalence relation which naturally extends to the cells of the polyhedra appearing in the disjoint union \eqref{eq:extended_mushrooms}.  The cellular complex $\Champext_{d}$ obtained by taking the quotient under this relation is not a manifold (it doesn't even have pure dimension).

Let us consider the maps which make the moduli space of mushrooms into a bimodule over (two different models) of the Stasheff associahedron:  A short verification shows that \eqref{eq:apply_m_k_source_mushrooms} is compatible with the inclusion \eqref{eq:inclusion_mushrooms_into_extended}, and a natural map
\begin{equation} \label{eq:replace_m_k_source}  \Champext_{d_1} \times \Rbar_{d_2} \to  \Champext_{d_1 + d_2 -1} \end{equation} 
obtained by attaching the disc the one of the $d_1$ inputs of an element of  $\Champext_{d_1}$.  On the other hand, and as in Equation \eqref{eq:structure_map_extended_shrubs}, the module map \eqref{eq:apply_m_k_target_mushrooms} can be recovered from the inclusion \eqref{eq:inclusion_mushrooms_into_extended} and a repeated application of the map
\begin{equation} \label{eq:replace_m_k_target}
\Stasheffbar_{d_1}  \times \Champext_{d_2}  \to \Champext_{d_1+d_2-1}
\end{equation}

All these constructions have analogues for labels, so we shall consider a sequence $\vI$ of such.  As all the components of \eqref{eq:extended_mushrooms} are smooth manifolds with corners, it makes sense to speak of smooth functions on $\Champext_{\vI}$. We define a universal $Q$-parametrised perturbation datum in $\Champext$ to be a $Q$-parametrised perturbation datum on all elements of  $\Champext_{\vI}$ which is smooth on each component, invariant under the equivalence relation $\sim$, and compatible with the perturbation data  $(\bfX^{\Stasheff},\bfK^{\cR}, \bfJ^{\cR})$ and the analogue of the maps \eqref{eq:replace_m_k_source} and \eqref{eq:replace_m_k_target} in the presence of labels.

\begin{defin}
A family of perturbation data on $\Champbar_{\vI}$ is {\bf smooth} if it is the restriction of a universal perturbation datum on $\Champext_{\vI}$.
\end{defin}

The analogue of Lemma \ref{lem:inductive_construction_pert_data_stasheff} holds.  As the proof is the same, we shall not repeat it.  The only subtle point in ensuring that Condition \eqref{eq:varying_floer_datum} holds.  However, the inductive construction starts by choosing perturbation data on the simplest space, which is the moduli of shrubs with $1$ input.  Definition \ref{lem:base_case} dictates how this initial choice should be made, and  Condition \eqref{eq:varying_floer_datum} is compatible with the requirement that the perturbation data be chosen consistently. 
\begin{lem} \label{lem:surjective_restriction_pert_data}
The restriction map from the space of $Q$-parametrised perturbation data on $\Champext_{\vI}$ to the space of perturbation data for a fixed mushroom is surjective. \noproof
\end{lem}

The transversality result we desire now follows by applying the idea of Lemma \ref{cap_transversality}.  We start by studying interior strata of $\Champ_{d}$:

The strata of $\Champ_{d}$ are labelled by the following data: a partition $d = d_1 + \ldots + d_r$, an increasing sequence of integers $1 = r_0 < r_1 < \ldots < r_{e} = r$, and a tree $S$ with $e$ inputs.  Indeed, write $\sD$ for this data, and first consider the stratum of $\Shrub_{\sD} \subset \Shrub_{r}$ consisting of trees homeomorphic to $S$ for which two vertices $v_i$ and $v_j$ are collapsed to the same point whenever $i$ and $j$ lie between consecutive integers $r_l$ and $r_{l+1}$, but are otherwise distinct.  The stratum of $\Champ_{d}$ consisting of mushrooms with $r$ smooth caps with respectively $(d_1, \ldots, d_r)$ inputs and a stem in $\Shrub_{\sD}$ will be denoted by $\Champ_{\sD}$; it is a product
\begin{equation} \Shrub_{\sD} \times \Pil_{d_1} \times \cdots \times \Pil_{d_r}.   \end{equation}

Define
\[ \C^{\infty} \Pil_{(d_1, \ldots, d_r)} (\vp) \equiv \C^{\infty} \Pil(\vp[d_1]_{0}) \times \cdots \times   \C^\infty \Pil(\vp[d]_{d - d_{r}}) ,\]
to be the space of smooth $r$-tuples of maps from the disc to $M$ satisfying the appropriate boundary conditions for the caps of a mushroom map in $\Champ(x_0, \vp)$ whose underlying mushroom has $r$ smooth caps with respectively $(d_1, \ldots, d_r)$ inputs.  Further, define
\[ \Pil_{\sD} (\vp)  \subset \C^{\infty} \Pil_{(d_1, \ldots, d_r)} (\vp)  \times \Champ_{\sD} \]
to be the vanishing locus of the $\bar{\partial}$ equation \eqref{perturbed_d_bar} perturbed according to the chosen perturbation datum (which now depends on the modulus in $\Champ_{\sD}$ rather than depending only on the corresponding moduli space of caps as explained in Remark \ref{rem:free_choice}).

We have a map 
\begin{equation}   \Pil_{\sD} (\vp)  \to \Champ_{\sD} \times Q^r \end{equation}
which, on the second factor, agrees with the projection of the ``outgoing arc" of each cap.

Let us now consider the stem. Generalising Equation \eqref{eq:fibre_product_picture_shrubs}, and  writing $\vQ = (Q_{i_0}, \ldots, Q_{i_r})$ we define
\begin{equation}  \Shrub_{\sD}(x_0, \vQ) \end{equation} 
to be the inverse image of the diagonal under a map 
\begin{equation} \Champ_{\sD} \times W^{s}(x_0) \times Q_{i_0,i^{2}_0} \times \cdots \times Q_{i^{r}_{0},i_d} \to Q^{r+1}, \end{equation}
whose factors are a composition of the natural inclusion with a diffeomorphism of the target obtained by integrating the flow of a descending arc of $S$ as in Equation \eqref{eq:fibre_product_picture_shrubs} (the flow depends on the perturbation  data, which are parametrised by the first factor).  Here, as in Equation \eqref{eq:naive_mushroom}, the sequence $\{ i^{k}_j \}_{j=0}^{d_k}$ records the boundary conditions on $\Pil(\vp[k])$.   Observe that  the moduli space of shrubs with a fixed source, output $x_0$, and arbitrary inputs is naturally an open submanifold of the ascending manifold of $x_0$.    The reader should also keep in mind that while there are $r$ evaluation maps, we are potentially considering a tree in which an arbitrary number of these maps necessarily agree because the corresponding incoming vertices lie on a subset of the shrub which is collapsed to a point.

The next result is simply a reinterpretation of the definition of a mushroom map (compare also Equation \eqref{eq:naive_mushroom}):

\begin{lem} \label{lem:mushroom_fibre_product}
The subset of $\Champ(x_0,\vp)$ consisting of maps whose source is a mushroom in  $\Champ_{\sD}$ is naturally diffeomorphic to the fibre product
\begin{equation} \label{eq:moduli_space_shrubs_fibre} \Champ_{\sD}(x_0,\vp) = \Shrub_{\sD}(x_0; \vQ)   \times_{\Champ_{\sD} \times Q^r }\Pil_{\sD} (\vp) .\end{equation} \noproof
\end{lem}

We complete this section by proving a Lemma from Section \ref{sec:mushroom_maps}

\begin{proof}[Proof of Lemma \ref{lem:mushrooms_expected_dimension}]

Since every component of a singular mushroom map consists of mushroom maps, discs perturbed with respect to the datum $(\bfK^{\cR}, \bfJ^{\cR})$, and gradient Stasheff trees perturbed with respect to $\bfX^{\Stasheff}$, the fact that we have chosen regular perturbation data to define $\Morse(Q_1,Q_2)$ and $\Fuk(Q_1,Q_2)$ implies that Lemma \ref{lem:mushrooms_expected_dimension} would follow once we show that all strata of $\Champ(x_0, \vp)$ have the expected dimension.  The previous Lemma reduces this to showing that for a generic choice of perturbation data, the fibre product \eqref{eq:moduli_space_shrubs_fibre} is transverse.   Of course, this will follow from an application of Sard-Smale.  To  simplify matters, we observe that the forgetful map 
\[ \Shrub_{\sD}(x_0; \vQ) \to \Shrub_{\sD} \]
is a submersion.  In particular, our result would follow once we show that the evaluation map
\begin{equation}\Pil_{\sD} (\vp) \to   \Pil_{d_1} \times \cdots \times \Pil_{d_r}\times Q^r  \end{equation}
can be made transverse to an arbitrary submanifold in the target by changing the perturbation datum.

Fix a mushroom map $\Psi$ with source $C$, and let $\cT_{\vP(C)}$ denote the space of infinitesimal deformations of perturbation data for the caps of $C$.  We introduce the notation
\begin{equation} \C_{\res}^{\infty}(\vP, \Psi^{*}(T M)) = \bigoplus_{1 \leq j \leq r}   \C_{\res}^{\infty}(P_{j}(C), \Psi_{j}^{*}(T M)).  \end{equation}
Using the surjectivity of the restriction map from the space of universal perturbation data to the space of perturbation data for the caps of $C$, we have reduced  the problem to showing that the evaluation map at the level of tangent spaces
\begin{equation}  \cT_{\vP(C)} \times  \C_{\res}^{\infty}(\vP, \Psi^{*}(T M)) \to T Q^r \end{equation}
is surjective.  Splitting the source and the target as a direct sum of $r$ factors, we find that this follows immediately from the proof of Lemma  \ref{cap_transversality} given in the previous section.

To prove transversality even in the base case of mushrooms with $1$ input, it is important to note that surjectivity holds even if we are allowing only perturbations of the almost complex structure (but which are of course allowed to vary along the caps, this is is exactly the setting in which transversality is proved in \cite{MS}).
\end{proof}

\subsubsection{Mushroom gluing: the rigid case} \label{sec:gluing_mushrooms}

Let us now consider the situation considered in Proposition \ref{prop:1_dim_moduli_space_manifold}: a critical point $x_0$ and a sequence of time-$1$ Hamiltonian chords $\vp$ such that the moduli space 
\begin{equation} \Champ(x_0, \vp) \end{equation}
has virtual dimension $1$.  By Lemma \ref{lem:mushrooms_expected_dimension} which is proved in the previous section, all strata of $\Champbar(x_0, \vp)$ are empty except those of virtual dimension $0$ and $1$.  In order to prove that $\Champbar(x_0, \vp)$ is a $1$ dimensional manifold with boundary, we must prove that points in the $0$ dimensional strata have neighbourhoods, in the Gromov topology, which are homeomorphic to an interval.

There are four distinct geometric situations to consider.
\begin{enumerate}
\item \label{item:tree_break} $\psi$ is a rigid gradient tree in $\Stasheff(x_0, \vx)$, and $\vPsi = ( \Psi_1, \ldots, \Psi_{k})$ is a sequence of rigid mushrooms with outputs $x_j$.  We write $(\vp[1], \ldots, \vp[r])$ for the collections of inputs of all mushrooms $\Psi_{j}$, and $\vp$ for the sequence obtained by concatenating the sequences $\vp[j]$. The pair $(\psi, \vPsi)$ defines an element of $\Champbar(x_0, \vp)$ obtained geometrically by attaching $\Psi_{j}$ to the $j$\th input of $\psi$.  Consider the stratum $\Champbar_{\sD}$ determined by the datum of $d_j = |\vp[j]|$ and a tree $S$ with exactly $r$ inputs.  If the perturbation data on the caps were independent of the stem, the result would follow from Lemma \ref{lem:gluing_shrubs}, and the observation that the moduli space we are studying has an alternative description as
\begin{equation} \Shrubbar(x_0, \Pil(\vp[1], \ldots, \vp[r]))  \end{equation}
where $\Pil(\vp[1], \ldots, \vp[r])$ is the product of the moduli spaces $\Pil(\vp[j])$, each which defines a chain on $Q$ by the evaluation map at the outgoing segment.  Using only Lemma \ref{lem:gluing_shrubs} as analytic input one can prove the general case by imitating the strategy for passing from Equation \eqref{eq:naive_mushroom} to Lemma \ref{lem:mushroom_fibre_product} (i.e. work fiberwise over  $\Champbar_{\sD}$).

\item \label{item:disc_bubble} $v$ is a rigid holomorphic curve in $\cR(p_0, \vp[2])$, and $\Psi$ a rigid mushroom in $\Champbar(x_0, \vp[1])$, such that $p_0$ agrees with the $k$\th input in $\vp[1]$.  In this case, we focus attention on the cap with inputs $\vp[1,j]$ wherein this $k$\th input lies.  Writing $\Psi_{j}$ for the map of caps along which we are attaching $v$, we assume again for simplicity that the perturbation datum on that cap is independent of the remaining caps and the stem, and use Lemma \ref{lem:gluing_caps} to find a neighbourhood of $(\Psi_{j},v)$ in $\Pilbar(\vp[1,j] \cup_{p_0} \vp[2])$ foliated by smooth codimension $1$ submanifolds which are the images of the boundary for gluing parameter $\epsilon$.   Note that the immersive assumption in Lemma \ref{lem:gluing_caps} is implied by the rigidity and regularity requirements on $\Psi$.  In addition,  Lemma \ref{lem:mushroom_fibre_product} gives a description of $\Champbar(x_0, \vp \#_{p_0} \vp[2])$ as a fibre product one of whose factors is  $\Pilbar(\vp[1] \#_{p_0} \vp[2])$.   The assumption that $\Champbar(x_0, \vp)$ is rigid and the $C^1$ convergence  of the evaluation maps to $Q$ on the image of the gluing map  implies that there is a unique solution for each such sufficiently small parameter.  The general case can be proved by working fibrewise over $\Champ_{\sD}$, or, alternatively,  by noticing that the smoothness in the choice of perturbation data implies that there is a neighbourhood of the singular mushroom we are considering where all moduli spaces together with their evaluation maps are perturbed by $C^{\infty}$ small amounts from moduli spaces with perturbation data independent of the stem.  After such a perturbation, a unique ray converging to the boundary persists.

\item \label{item:mushroom_collapse} $\Psi$ is a mushroom with $r$ caps, such that two successive incoming vertices (say $k-1$ and $k$) of the stem are collapsed to one point, which is moreover rigid among such mushrooms.  In particular, the boundary conditions on the outgoing arcs of the two caps $P_{k-1}(C)$ and $P_{k}(C)$ agree.  In this case, there is a unique $1$-parameter family of mushrooms $\Psi^{t}$, whose domains have stems for which the distance between the incoming vertex labelled $k$ and $k+1$ does not vanish, and which converge to $\Psi$.  The argument given in Lemma \ref{prop:cup_product_agrees_perturbation}  proves this.

\item \label{item:cap_break} In the same situation as before, we glue $\Psi_{P_{k-1}}$ and  $\Psi_{P_{k}}$ to obtain a cap whose inputs are given by concatenating those for  $\Psi_{P_{k-1}}$ and  $\Psi_{P_{k}}$.  Lemma  \ref{lem:gluing_caps} again applies here to show that this moduli space of caps is foliated by the images of the gluing map for sufficiently small gluing parameter $\epsilon$, and that the evaluation map converges in the $C^1$ topology as $\epsilon$ goes to $0$.  The argument of Case (2) yields the desired conclusion.
\end{enumerate}

\section{Orientations and signs} \label{sec:orientations}
\subsection{The differential in Morse and Floer theory}
We begin by describing the differential in the Morse complex $CM^{*}(f_{i,j})$.  We define  $\Stasheff(x_0,x_1)$ to be the fibre product of  $W^{s}(x_0) $ and $W^{u}(x_1)$ over their common inclusion in  $Q_{i,j}$, yielding a natural short exact sequence of vector spaces
\begin{equation}  0 \to T \Stasheff(x_0,x_1) \to T (W^{s}(x_0) \times W^{u}(x_1) ) \to T Q_{i,j} \to 0, \end{equation}
and hence an isomorphism of top exterior powers
\begin{equation} \label{eq:isomorphism_rigid_line}   \lambda(\Stasheff(x_0,x_1 )) \otimes \lambda(Q_{i,j}) \cong   \lambda(W^{s}(x_0)) \otimes \lambda(W^{u}(x_1)) . \end{equation} 
 Whenever $\deg(x_0) = \deg(x_1) +1$, the intersection between $W^{s}(x_0)$ and $W^{u}(x_1)$ is a segment, mapping to $\bR$ immersively under $f_{i,j}$, so the tangent space of 
 \begin{equation}  \psi \in  \Stasheff(x_0,x_1) \end{equation}
is canonically oriented, with our conventions dictating that a rigid gradient flow line is given the opposite orientation of the one induced by projection to $\bR$ (i.e. the orientation we pick identifies $x_1$ with $-\infty$).  Using the decomposition of the tangent space of $Q_{i,j}$ at $x_1$, we have an isomorphism 
\begin{equation} \label{eq:decom_tangent_space_crit_points}
\lambda(Q_{i,j}) \cong \lambda(W^{s}(x_1)) \otimes \lambda(W^{u}(x_1)),
\end{equation}
which, together with the map in \eqref{eq:isomorphism_rigid_line} determines an isomorphism
\begin{equation} \lambda(W^{s}(x_1)) \cong  \lambda(W^{s}(x_0)), \end{equation} 
and hence an induced map on orientation lines
\begin{equation} \mu_{\psi}^{\M}\co |\ro_{x_1}| \to |\ro_{x_0}|. \end{equation}
The differential is defined to be the sum of all such terms multiplied by a global sign:
\begin{equation} \label{eq:formula_differential_morse}   \mu_{1}^{\M}([x_1]) = \sum_{\stackrel{\deg(x_0) = \deg(x_1) +1}{\psi \in \Stasheff(x_0,x_1)}} (-1)^{n}  \mu_{\psi}^{\M}([x_0])\end{equation}

\subsection{Orienting abstract moduli spaces}
As noted in our discussion of the Fukaya and Morse categories, the construction of $A_{\infty}$ categories from counts of rigid discs or trees requires first a choice of orientations on the abstract moduli spaces controlling these categories.
\subsubsection{Stasheff trees}
Let $T$ be a trivalent Stasheff tree with $d$ inputs.  Given an integer $k$, let $A_{k}$ denote the descending arc starting at the $k$\th incoming leaf, and let $b_{k}$ denote the vertex of $T$ where $A_{k}$ and $A_{k-1}$ meet.  We shall label the edges of this trivalent vertex $e_{k}^{\ell}$, $e_{k}^{r}$ and $e_{k}^{d}$ as in Figure \ref{fig:right_turn}.  

\begin{figure}[h] 

 \input{right-turn.pstex_t}
   \caption{}
   \label{fig:right_turn}
\end{figure}

We  order the (finite) edges of $T$ using the following inductive procedure:
\begin{enumerate}
\item Set $T_{2}$ to be the union of all the external edges of $T$
\item If $e_{k}^{r}$ lies in $T_{k}$, then define $e_{k} = e_{k}^{d}$.  Otherwise, set $e_{k} = e_{k}^{r}$.
\item Set $T_{k}$ to be the union of $T_{k-1}$ with $e_{k}$.
\end{enumerate}

\begin{lem}
If $j \neq k$, $e_{j} \neq e_{k}$.
\end{lem}
\begin{proof}
We must prove that it is impossible for both $e_{k}^{d}$ and $e_{k}^{r}$ to lie in $T_{k-1}$.  The edge $e_{k}^{d}$ lies in $T_{k-1}$ only if it is the outgoing external edge.   In that case, a simple check shows that $e_{k}^{r}$ cannot lie in $T_{k-1}$.
\end{proof}

We define the orientation on $\Stasheff_{d}$ to be given by
\begin{equation}  \label{eq:orientation_stasheff} (-1)^{r(T)}  dt_{e_{3}} \wedge \cdots  \wedge dt_{e_{d}}\end{equation}
where $t_{e}$ is the length of $e$, and $r(T)$ is the number of integers between $3$ and $d$ for which $e_{k} = e_{k}^{r}$.
\begin{lem} \label{lem:orientation_stasheff_good}
The orientation given by formula \eqref{eq:orientation_stasheff} is independent of the topological type of $T$.
\end{lem}
\begin{proof}
It suffices to check that this is the case in a neighbourhood of the codimension $1$ walls.  As these  consist of trees with a unique $4$-valent vertex, we start with a trivalent tree, and pass through the wall by collapsing an edge $e$.  There is a natural way to identify the edges on trivalent trees on either side of the wall, and the reader may easily check that, under this identification, the edges $e_{j}$ are unchanged whenever $b_{k}$ is not one of the endpoints of $e$. There are two cases to consider, depending on whether $e_{k}$ agrees with $e_{k}^{r}$, whose verification is left to the reader.
\end{proof}
\begin{figure}[t] 
 \input{stasheff_breaking.pstex_t}
   \caption{}
   \label{fig:stasheff_breaking}
\end{figure}
We now consider the boundary stratum of $\Stasheffbar_{d}$ which is the image of the map
\begin{equation}  \Stasheffbar_{d_1} \times \Stasheffbar_{d_2} \to \Stasheffbar_{d} \end{equation}
given by attaching a tree with $d_2$ inputs to the $k+1$\st incoming leaf of a tree with $d_1$ inputs  (here, $d_1 + d_2 = d +1$).   Since we know that our choices of orientations are consistent throughout the moduli space, it suffices to compare these orientations in a special case.  The author found the case shown in Figure \ref{fig:stasheff_breaking} to be particularly convenient.  The resulting computation implies:
\begin{lem} \label{lem:boundary_orientation_stasheff}
The product orientation on
\begin{equation}  (-1, 0]\times \Stasheffbar_{d_1} \times \Stasheffbar_{d_2} \into \Stasheffbar_{d}  \end{equation}
differs from the boundary orientation by a sign of
\begin{equation}  (d_1 -k) d_2  + d_2 + k .\end{equation} 
\noproof
\end{lem}

\subsubsection{Shrubs}

The moduli space of shrubs can be most easily oriented using the map
\begin{equation}  \Shrub_{d} \to \Stasheff_{d} \end{equation}
obtained by replacing every incoming leaf by one which has infinite length.  The fibre is isomorphic to a half-infinite line, which we trivialize using the distance from the outgoing edge to an incoming leaf.  We write $\ell$ for this function, and use the orientation
\begin{equation}  (-1)^{r(T)}  d\ell \wedge dt_{e_{3}} \wedge \cdots \wedge dt_{e_{d}}\end{equation}
for shrubs of homeomorphism type $T$ (note that the homeomorphism type is preserved by projection to $\Stasheff_d$).  Lemma \ref{lem:orientation_stasheff_good} implies immediately that this formula gives orientations on the top strata of $\Shrub_{d}$ compatible along their common boundary facets.  It remains to understand the inductive behaviour at the boundary of $\Shrubbar_{d}$:

\begin{figure} 

 \input{shrub_breaking.pstex_t}
   \caption{}
   \label{fig:shrub_breaking}
\end{figure}

\begin{lem} \label{lem:boundary_orientation_shrubs}
The product orientation on 
\begin{equation}  \label{eq:boundary_shrub_break}  (-1, 0] \times \Stasheffbar_{r} \times \Shrubbar_{d_1} \times \cdots \times \Shrubbar_{d_r} \to \Shrubbar_{d}  \end{equation}
differs from the boundary orientation by a sign of
\begin{equation} \label{eq:boundary_shrub_break_sign}  1 + \sum_{k=1}^{r} (r-k) (d_k+1) \end{equation}
while the orientation on the stratum
\begin{equation}  \label{eq:boundary_shrub_collapse} (-1, 0] \times \Shrubbar_{d-1} \to \Shrubbar_{d}  \end{equation}
whose image consists of shrubs whose $k$\th and $k-1$\st inputs are collapsed differs by a sign of
\begin{equation} k +1\end{equation}
from its orientation as a subset of the boundary.
\end{lem}
\begin{proof}
Given the fact that the orientations are consistent throughout the moduli space, it suffices to compute the difference in sign orientation for one topological type of trees.  The case \eqref{eq:boundary_shrub_collapse} is essentially trivial, while that of \eqref{eq:boundary_shrub_break} can be conveniently checked using the singular shrub appearing in Figure \ref{fig:shrub_breaking}. 
\end{proof}

\subsection{Orienting moduli spaces of shrubs} \label{sec:orienting_maps}
We shall fix the following conventions:  whenever a space $X$ is given as a transverse fibre product of $X_1$ and $X_2$ over $X_0$, the short exact sequence
\begin{equation} 0 \to TX \to TX_1 \oplus TX_2 \to T X_0 \to 0 \end{equation}
determines a fixed isomorphism
\begin{equation} \lambda(X) \otimes \lambda(X_0) \cong \lambda(X_1) \otimes \lambda(X_2) .\end{equation}

For example, the isomorphism of Equation \eqref{eq:iso_det_bunldes_stasheff}  comes from the choices
\begin{equation}
\begin{array}{ll}
 X = \Stasheff(x_0, \vx)   & X_0 = Q^{d+1}  \\
X_1= Q  & X_{2} = \Stasheff_{d} \times W^{s}(x_0) \times W^{u}(x_1) \times \cdots \times  W^{u}(x_d) ,\end{array}
\end{equation}
 while Equation \eqref{eq:isomorphism_orientation_bundles_shrubs} comes from the choices
\begin{equation}
\begin{array}{ll}
 X = \Shrub(x_0, \vsig)   & X_0 = Q^{d+1}  \\
X_1= Q  & X_{2} = \Stasheff_{d} \times W^{s}(x_0) \times \check{\sigma}_{1} \times \cdots \times \check{\sigma}_{d}.\end{array}
\end{equation}

We first use this fixed choice in order to prove the validity of the signs in Proposition \ref{prop:cup_product_agrees_perturbation}.  More precisely, we shall prove the validity of the signs in the right hand side, leaving the left hand side to the reader.  For clarity of the argument, we introduce a change of sign of
\begin{equation} \label{eq:change_orient_1_dim_shrub} (n+1) ( r + \dagger(\vsig)) \end{equation}
which we use to change the orientation of all moduli spaces of shrubs with inputs $\vsig$ and dimension $1$.

The starting point is an element of $\Shrubbar(x_0, \vsig)$ which consists of a rigid gradient tree in $\Stasheff(x_0, \vx)$ with $\vx = (x_1, \ldots, x_r)$, and a collection of $r$ rigid shrubs in $\Shrub(x_r, \vsig[r])$, where $(\vsig[1], \ldots, \vsig[r])$ is a partition of $\vsig$. The orientation of the rigid tree is given by Equation \eqref{eq:iso_det_bunldes_stasheff} which can be simplified to:
\begin{equation*}  \lambda(Q^{r}) \cong \lambda( \Stasheff_{r}) \otimes \lambda(W^{s}(x_0)) \otimes \lambda (W^{u}(\vx))  \end{equation*}
by using the splitting of the inclusion of the tangent space of the $r+1$-fold diagonal whose kernel are the last $r$ factors.  Next, we decompose the $k$\th copy of $\lambda(Q^{r})$ on the left as a tensor product:
\begin{equation*} \lambda(W^{s}(x_{k})) \otimes  \lambda(W^{u}(x_{k})),  \end{equation*}
and rearrange the terms in the left hand side as
\begin{equation*}   \lambda(W^{s}(\vx)) \otimes \lambda(W^{u}(\vx)), \end{equation*}
where the first factor is the tensor product of $\lambda(W^{s}(x_{k})$ in their usual order, introducing a Koszul sign of 
\begin{equation} \label{eq:first_sign_shrubs_break}  \sum_{k=1}^{r} \deg(x_{k})  \left( n (k-1) + \sum_{j = 1}^{k-1} \deg(x_j) \right).  \end{equation}
Cancelling $\lambda(W^{u}(\vx))$ from both sides of the equation, we arrive at the isomorphism
\begin{equation} \label{eq:simplified_tree} \lambda(W^{s}(\vx)) \cong \lambda( \Stasheff_{r}) \otimes \lambda(W^{s}(x_0)) . \end{equation}

Rigid shrubs  are oriented via Equation \eqref{eq:iso_det_bunldes_stasheff}, which can be simplified (using the same splitting of the inclusion of the tangent space of the diagonal) as
\begin{equation*}  \lambda(Q^{d_k}) \cong  \lambda(\Shrub_{d_k}) \otimes \lambda(W^{s}(x_k)) \otimes \lambda( \vsig[k]) .\end{equation*}
Introducing a sign of
\begin{equation} \label{eq:second_sign_shrubs_break} \sum_{k=1}^{r}  \deg(x_k) ( d_k +1 ) \end{equation}
we obtain an isomorphism
\begin{equation} \label{eq:simplified_shrub} \lambda^{-1}(W^{s}(x_k)) \otimes  \lambda(Q^{d_k}) \cong  \lambda(\Shrub_{d_k}) \otimes \lambda( \vsig[k]) .\end{equation}

The product orientation of the singular shrub we are considering is given by taking the tensor product of Equation \eqref{eq:simplified_tree} with a copy of Equation \eqref{eq:simplified_shrub} for each shrub:
\begin{multline*}   \lambda(W^{s}(\vx)) \otimes  \lambda^{-1}(W^{s}(x_1)) \otimes  \lambda(Q^{d_1}) \otimes \cdots \otimes  \lambda^{-1}(W^{s}(x_r)) \otimes  \lambda(Q^{d_r})  \cong  \\
\lambda( \Stasheff_{r}) \otimes \lambda(W^{s}(x_0)) \otimes  \lambda(\Shrub_{d_1}) \otimes \lambda( \vsig[1]) \otimes \cdots \otimes  \lambda(\Shrub_{d_r}) \otimes \lambda( \vsig[r]).  \end{multline*}

We shall compare this with the induced orientation as a boundary.  On the left hand side, we permute factors so that we arrive at
\begin{equation*} \left( \bigotimes_{k=1}^{r}  \lambda^{-1}(W^{s}(x_k)) \otimes  \lambda(W^{s}(x_k)) \right)  \otimes \lambda(Q^{d}) \end{equation*}
which introduces a sign of
\begin{equation}  \label{eq:third_sign_shrubs_break} \sum_{k=1}^{r} \left( \deg(x_{k}) \left( n(d_1 + \cdots + d_{k-1}) +  \sum_{j=k}^{r} \deg(x_{j}) \right)  \right). \end{equation}
 
On the right hand side, we rearrange the terms so that they appear in the order
\begin{equation*} \lambda(\Stasheff_{r}) \otimes \lambda(\Shrubbar_{d_1}) \otimes  \cdots \otimes \lambda(\Shrubbar_{d_r}) \otimes \lambda(W^{s}(x_0)) \otimes \lambda( \vsig) .\end{equation*} The Koszul signs  amount to:
\begin{equation*}  \sum_{k=1}^{r} (d_k+1)\left(\deg(x_0)  + \sum_{j=1}^{k-1} (\deg(\vsig[j]) + n d_j) \right) ,\end{equation*}
which we prefer to rewrite as
\begin{equation}  \label{eq:fourth_sign_shrubs_break} (r+ d) \deg(x_0) + \sum_{j=1}^{k-1} nd_j(d_k+1) +  \sum_{k=1}^{r} \deg(\vsig[k]) \left(r-k-1+ \sum_{j = k+1}^{r} d_j \right).\end{equation}
The difference between the boundary and product orientations is given by the sum of Equations \eqref{eq:first_sign_shrubs_break},  \eqref{eq:second_sign_shrubs_break}, \eqref{eq:third_sign_shrubs_break}, and \eqref{eq:fourth_sign_shrubs_break}. The first three give
 \begin{equation*} \sum_{k=1}^{r} \deg(x_{k})  {\Big{(}} (k+1)(n+1) + k + d_k + n(d_1 + \cdots + d_{k-1}) +  r + \deg(x_0)   {\Big{)}}. \end{equation*}
The term $\deg(x_k) (k+1) (n+1)$ cancels with a term appearing in the definition of the Morse category, which introduces an additional sign of $(n+1)r$ . This can be see by writing
\begin{equation*} \deg(x_0)  + \sum_{k=1}^{r} k \deg(x_k) = r + \sum_{k=1}^{r} (k-1) \deg(x_k) \end{equation*}
using the fact that the gradient tree with inputs $x_k$ and output $x_0$ is rigid.  Incorporating these signs coming from the Morse category yields
\begin{equation*}   (n+1) r +   \sum_{k=1}^{r} (\deg(\vsig[k])+ d_k +1)  \left( 1 + k + d_k  + n(d_1 + \cdots + d_{k-1})   \right), \end{equation*}
and the sum with \eqref{eq:fourth_sign_shrubs_break} equals
\begin{multline*}  (r+ d) \deg(x_0) + (n+1) r +   \sum_{k=1}^{r} (d_k+1) (d_k+1+ k) \\
+  \sum_{k=1}^{r} \deg(\vsig[k]) \left( r + d + (n+1) \sum_{j =1}^{k-1} d_j  \right).
\end{multline*}

There are additional signs coming from the definition of the functor $\cF$ (see Equation \eqref{eq:functor_simp_morse_formula}), and the change in orientations  of $1$-dimensional moduli spaces given by Equation \eqref{eq:change_orient_1_dim_shrub}.    Together, these terms contribute
\begin{equation} (n+1) \left(r+  \sum_{k=1}^{r} \left( \deg(\sigma_{k}) \sum_{j = 1}^{k-1} d_k \right) \right), \end{equation}
after which we use the fact that $\sum \deg(\vsig)= \deg(x_0) + d$ to simplify the total remaining sign to
\begin{equation*}  (r+d)d + \sum_{k=1}^{r} (d_k+1) (1+d_k + k) =  (r+1)(d+r) + \sum_{k=1}^{r} (d_k+1) (1+ k) = \sum_{k=1}^{r} (d_k+1) (r+ k) . \end{equation*}
Adding the sign coming from Equation \eqref{eq:boundary_shrub_break_sign}, we conclude that the boundary and product orientations are exactly opposite with the convention we have chosen.

\subsection{Orientations for mushroom maps} \label{sec:orienting_mush_maps}
Throughout this section, we shall work with the simplified setup in which moduli spaces of mushroom maps are given by the fibre product in \eqref{eq:naive_mushroom}.  More precisely, the moduli space of mushrooms $\Champ(x_0, \vp)$ is stratified with top strata labelled by a choice of a partition of $\vp$ into sequences $(\vp[1], \ldots, \vp[r])$. The stratum determined by this data is the moduli space of shrubs with $r$ inputs lying on the images of the projection maps from the moduli spaces of caps $\{ \Pil(\vp[k]) \}_{k=1}^{r}$, and is therefore given by a fibre product as in the previous section with
\begin{equation}
\begin{array}{ll}
 X = \Shrub(x_0,\Pil(\vp[1]), \ldots , \Pil(\vp[r]) )  & X_0 = Q^{r+1}  \\
X_1= Q  & X_{2} = \Shrub_{r} \times W^{s}(x_0) \times \Pil(\vp[1]) \times \cdots \times \Pil(\vp[r]) \end{array}
\end{equation}
Every such mushroom $\Psi$ therefore determines an isomorphism
\begin{multline} 
\label{eq:orientation_iso_shrub}
\lambda(\Champ(x_0, \vp)) \otimes \lambda( Q^{r+1} )  \cong \\ \lambda(Q) \otimes \lambda(\Shrub_{r}) \otimes \lambda(W^{s}(x_0)) \times \lambda(\Pil(\vp[1])) \times \cdots \times \lambda(\Pil(\vp[r])) .\end{multline}

The description of $\lambda(\Pil(p_1)$ given in Equation \eqref{eq:orientation_cap_one_input} extends to an isomorphism
\begin{equation} \label{eq:orientation_caps}  \lambda(\Pil(\vp[k]))  \cong \lambda(\Pil_{d_k}) \otimes \ro^{\vee}_{p_1^{k}} \otimes \cdots \otimes  \ro^{\vee}_{p_{d_k}^{k}} \otimes \lambda(Q) .\end{equation}

When $\Psi$ is rigid, our choice of orientations on $\Shrub_{r}$, on $\cR_{d_k+1}$ (using the conventions from \cite{seidel-book}) and the identification $\Pil_{d_k} \cong  \cR_{d_k+1} $  fixed in Remark \ref{caps_are_discs}, therefore determines a unique homomorphism
\begin{equation} \ro_{p_1} \otimes \cdots \otimes  \ro_{p_d} \to \ro_{x_0}  .\end{equation} 

We write $\cF^{\Psi}$ for the induced map on orientation bundles.  In order to get the correct sign, we first introduce the notation
\begin{equation} \dagger(\vPil) \equiv \sum_{j=1}^{r} k \deg( \Pil(\vp[k])) = \sum_{j=1}^{k} \left( 1 - d_k + \sum_{\ell=1}^{d_{k}} \deg(\vp[k]) \right), \end{equation}
where $\deg(\vp[k])$ is the sum of the degrees of the elements of $\vp[k]$.   The sign used to define $\cG$ is given by 
\begin{equation} \label{eq:twisting_orient_mushroom} \ddagger(\Psi) = (n+1) \dagger(\vPil) + \sum_{k=1}^{r}  \dagger(\vp[k]) \end{equation}

Having fixed the signs in the functor $\cF$, we have the necessary ingredients to prove Proposition \ref{caps_are_discs}.  The proof considers the $4$ types of boundary strata discussed in Section \ref{sec:gluing_mushrooms}, and shows that each such stratum contributes with the correct sign to Equation \eqref{a_oo-floer-morse},  after taking into account the sign \eqref{eq:twisting_orient_mushroom}, as well as the signs that go into the construction of $\Fuk(Q_1,Q_2)$ and $\Morse(Q_1,Q_2)$.  More precisely, it is enough to show that, after taking these artificial signs into account, the difference between the boundary and product orientation on each stratum of $\Champbar(x_0, \vp)$ is given by the sign claimed in the $A_{\infty}$ equation \eqref{eq:simp_to_morse_A_infty_holds}, up to a universal constant, independent of the stratum (but that is allowed to depend on $\vp$).

The proof is a tedious computation, with the cases \eqref{item:tree_break} and \eqref{item:disc_bubble}  from Section \ref{sec:gluing_mushrooms} respectively following from the computations performed in \cite{seidel-book} and the previous section.  Indeed, whenever a rigid disc bubbles off the $k$\th cap, the only affected term in Equation \eqref{eq:twisting_orient_mushroom} is $\dagger(\vp[k])$.  This is exactly the sign convention for holomorphic discs we have used in constructing the Fukaya category, see \cite{seidel-book}*{Equation (12.24)}.   Similarly, when the stem breaks into a gradient tree together with several shrubs, only the term   $(n+1) \dagger(\vPil)$ changes.   The integer $\deg(\Pil(\vp[k])$ agrees with the codimension of $\Pil(\vp[k])$ as a chain on $Q$, so these are exactly the conventions used in the previous section while orienting moduli space of shrubs.

Therefore, we shall only discuss cases \eqref{item:mushroom_collapse} and \eqref{item:cap_break}, which correspond to terms which are supposed to cancel pairwise.  In terms of orientations, we must therefore show that, whenever given a singular mushroom map lying in such an ``internal'' boundary, there is no difference between the orientations obtained by considering it as a boundary stratum of the two components described in cases \eqref{item:mushroom_collapse} and \eqref{item:cap_break}.  As in the previous section, it is convenient to change the orientation of the $1$-dimensional moduli spaces as well.  The sign we use is given by
\begin{equation} \label{eq:twisting_orient_mushroom_1_dim}  (n+1)\left( r + \dagger(\vPil) \right) + \sum_{k=1}^{r}  \left( 1 + \dagger(\vp[k]) + d_{k} \deg( \vp[k]) \right) \end{equation}

We will start with case \eqref{item:mushroom_collapse}, for which the orientation is induced by the isomorphism of Equation \eqref{eq:orientation_iso_shrub}.  For simplicity, we begin by choosing a splitting of the inclusion of the tangent space $(r+1)$-diagonal of $Q$, corresponding to the inclusion of all but the first factor.  Introducing a sign of $n$, we can simplify to an isomorphism
\begin{equation} \label{eq:simplified_iso_shrubs} \lambda(\Champ(x_0, \vp)) \otimes \lambda(Q^{r}) \cong \lambda(\Shrub_{r}) \otimes \lambda(W^{s}(x_0)) \otimes \lambda( \Pil(\vp[1])) \otimes \cdots \otimes \lambda (\Pil(\vp[r])) .\end{equation}
The next step is to use the isomorphism \eqref{eq:orientation_caps} to decompose  $\lambda(\Pil(\vp[k-1]))$ into its constituent factors, then cancel the resulting copy of $\lambda(Q)$ with its matching factor on the right hand side (this is the $k-1$\st factor starting at the left).  The resulting Koszul sign is
\begin{multline} \label{eq:mushroom_internal_boundary_first_sign}  n (k +1) + n \left(r+1 + \deg(x_0) + n(k-2) +  \sum_{j=1}^{k-1}\deg(\Pil(\vp[j]))   \right) \\ = n \left(r + \deg(x_0) +  \sum_{j=1}^{k-1} ( 1 + d_j + \deg(\vp[j])) \right)  \end{multline}
Decomposing $\lambda(\Pil(\vp[k]))$ as well, we find the expression
\begin{equation*}  \lambda(\Pil_{d_{k-1}}) \otimes \ro^{\vee}_{p_{1}^{k-1}} \otimes \cdots \otimes \ro^{\vee}_{p_{d_{k-1}}^{k-1}} \otimes  \lambda(\Pil_{d_{k}}) \otimes \ro^{\vee}_{p_{1}^{k}} \otimes \cdots \otimes  \ro^{\vee}_{p_{d_{k}}^{k}} \otimes \lambda(Q)  \end{equation*}
in the right hand side.  Using the identification of $\Pil_{d}$ with $\cR_{d+1}$ discussed in Remark \eqref{caps_are_discs}, and Lemma \ref{lem:boundary_orientation_stasheff} we may replace this by
\begin{equation*}  \lambda(\bR) \otimes  \lambda(\Pil(\vp[k-1] \cup \vp[k])) \end{equation*}
at the cost of a sign given by
\begin{equation}  \label{eq:mushroom_internal_boundary_second_sign} d_{k-1} +  (d_k+1) \deg(\vp[k-1]). \end{equation}
Note that in Lemma \ref{lem:boundary_orientation_stasheff}, $d_1$ plays the role of $d_{k-1}+1$, $d_2$ that of $d_{k}+1$, and $k$ is our $d_{k-1}$ (making for incompatible notations).  We now move the normal co-vector $\lambda(\bR)$ to the beginning of the right hand side, using Lemma \ref{lem:boundary_orientation_shrubs} to cancel with the normal vector of the inclusion of $\Shrub_{r-1}$ into $\Shrub_{r}$ in which the $k$\th and $k-1$\st incoming leaves are collapsed.  The resulting sign is:
\begin{equation}\label{eq:mushroom_internal_boundary_third_sign} k + r + \deg(x_0) + \sum_{j=1}^{k-2}\left( 1 + d_j + \deg(\vp[j]) +n \right)   \end{equation}
The final output of this procedure is an isomorphism
\begin{multline}   \lambda(\Champ(x_0, \vp)) \otimes \lambda(Q^{r-1}) \cong \\ \lambda(\Shrub_{r-1}) \otimes \lambda(W^{s}(x_0)) \otimes \lambda( \Pil(\vp[1])) \otimes \cdots  \otimes \lambda(\Pil(\vp[k-1] \cup \vp[k]))  \otimes \cdots \otimes \lambda (\Pil(\vp[r])) , \end{multline} 
which is exactly the analogue of the isomorphism of Equation \eqref{eq:simplified_iso_shrubs} for the case \eqref{item:cap_break}.    In particular, the difference between the orientations of a singular shrub as a boundary in cases \eqref{item:cap_break} and \eqref{item:mushroom_collapse} is given by the sum of the expressions appearing in Equations \eqref{eq:mushroom_internal_boundary_first_sign}, \eqref{eq:mushroom_internal_boundary_second_sign}, and \eqref{eq:mushroom_internal_boundary_third_sign}:
\begin{equation}  \label{eq:sign_difference_internal_boundary} 1  +  d_k \deg(\vp[k-1])  + (n+1) \left(k + r + \deg(x_0) +  \sum_{j=1}^{k-1} ( 1 + d_j + \deg(\vp[j])) \right)\end{equation}
We now consider the effect of the sign changes encoded in Equation \eqref{eq:twisting_orient_mushroom_1_dim}.  First, we note that the difference between $\dagger(\vPil)$ on either side is given by
\begin{equation} (k-1) +  \sum_{j = k}^{r} \deg(\Pil(\vp[j])) , \end{equation}
from which we conclude that the second term changes by
\begin{equation} \label{eq:sign_contribution_caps_to_mushroom} (n+1)\left( k + \sum_{j = k}^{r} (1 + d_j + \deg(\vp[j]) ) \right) \end{equation}
Since the moduli space we are studying is $1$-dimensional, we have an equality (modulo $2$)
\begin{equation} \deg(x_0) = d + \sum_{j=1}^{r} \deg(\vp[j]) ,  \end{equation}
which readily implies that \eqref{eq:sign_contribution_caps_to_mushroom} cancels the third term of Equation \eqref{eq:sign_difference_internal_boundary}.  The reader may easily check that the remaining terms of Equation \eqref{eq:twisting_orient_mushroom_1_dim} exactly cancel with the first two terms of Equation \eqref{eq:sign_difference_internal_boundary}, completing our argument.

\appendix

\section{The case of clean intersections} \label{sec:case-clean-inters}
In this section, we shall show that our construction, which so far yields Theorem \ref{thm:general_result} whenever $B$ is a point, can be extended to the general case.

The starting point is the observation that, whenever two Lagrangian manifolds $Q_1$ and $Q_2$ meet cleanly along a submanifold $B$, we have a perfect pairing on the normal bundles $ N_{Q_1} B$ and $ N_{Q_2}B $; in particular these normal bundles are dual, hence isomorphic as unoriented vector bundles.  Conversely,  given smooth embeddings of a compact manifold $B$ into $Q_1$ and $Q_2$ with normal bundles isomorphic to a given vector bundle $N B$, the choice of such an isomorphism determines a plumbing of $T^*Q_1$ and $T^*Q_2$ along a neighbourhood of $B$.

Explicitly, we fix a Riemannian metric on $B$ and on the vector bundle $N B  $, which induces a metric on the total space.   We may assume that we are given a Riemannian metric on $Q_i$ whose restriction to a neighbourhood of $B$ is isometric to the disc bundle of radius $4$ in $NB$, with $U_i$ corresponding to the disc bundle of radius $2$.  By Weinstein's theorem, there is a neighbourhood of $B$
\begin{equation}
  \label{eq:M_B}
  M_{B} \subset T^*Q_{i}
\end{equation}
which is symplectomorphic to the disc bundle of a symplectic vector bundle over a neighbourhood of the zero section in $T^*B$; this vector bundle is isomorphic to the pullback of $N B \otimes \bC$ to the cotangent bundle of $B$.  By construction, $N B \otimes \bC$ carries two Lagrangian sub-bundles obtained as the inclusion of the real part $N B  $ and the imaginary part $\sqrt{-1} NB $.  The Weinstein chart may be chosen so that the image of $Q_{i}$ is the union (over all points in $B$) of the real parts.  We construct a manifold
\begin{equation}
M = D^* Q_1 \#_{D^*B} D^* Q_2,
\end{equation}
 by gluing disc cotangent bundles of $Q_1$ and $Q_2$ along this common subset after complex multiplication in $N B \otimes \bC$.  As this map is an exact symplectomorphism for the appropriate Liouville form, $M$ is again  an exact symplectic manifold with exact Lagrangian embeddings of $Q_1$ and $Q_2$ intersecting cleanly along $B$.  In fact, we can choose a Liouville form on $M$ which away from $B$ agrees with the canonical forms on $D^* Q_i$, and which near $B$ can be written as a sum
 \begin{equation} \label{eq:form_near_B}
   \theta_{T^*B} \oplus \theta_{NB}
 \end{equation}
 where the first form the pullback of the canonical form on the cotangent bundle, and the second defines a flow which radially rescales each fibre.  
\subsection{Fukaya category}  \label{add:generalise_floer}
The fibrewise squared norm of a point in $M_{B}$ (considered as a vector in either $ T^*Q_{1} $ or $ T^*Q_{2} $), defines a map
\begin{equation*}
  \rho \co M_{B} \to [0,+\infty)^{2}.
\end{equation*}
As in Equation \eqref{eq:inside_manifold_near_0}, we define
\begin{equation*}
  M_{B}^{\ins} = (\chi \circ \rho)^{-1}(\epsilon).
\end{equation*}
This is a smooth manifold whose boundary is transverse to the Liouville flow of \eqref{eq:form_near_B}, and which away from a neighbourhood of $B$ agrees with the $\epsilon$-neighbourhood of the zero section in $T^* Q_i$.  As before, we write
\begin{equation*}
  M^{\ins} \subset M
\end{equation*}
for the Liouville subdomain obtained by extending this submanifold to all of $M$.  Condition \eqref{eq:allowed_J-H} now makes sense in the clean intersection case as well, and the construction of the Fukaya category proceeds as in the transverse case.

The analogue of Condition \eqref{eq:normalisation_floer_complex} is the requirement that the Floer complexes $CF^{*}(Q_1,Q_2)$ vanish in degrees smaller than $0$ or larger than $\dim(B)$.  This is easy to achieve if $B$ is connected by shifting the grading on $Q_1$.  Otherwise, one needs to pick an appropriate complex volume form on the plumbing in order for this result to hold (using the fact that the space of complex volume forms is an affine space over $H^1(M, \bZ)$, and that $H^{0}(B, \bZ)$ controls the difference between this space and $H^1(Q_1, \bZ) \oplus H^1(Q_2, \bZ)$).

However, if we were considering the situation of Theorem \ref{thm:general_result}, the total space $W$ may not admit any complex volume form  for which the Floer cohomology groups are supported in such a degree.  In this case, there are integers $\{ m_{k} \}_{k=1}^{r}$ labelled by the components of $B$ such that the Floer cohomology $HF^{*}(Q_1,Q_2)$ splits as a direct sum, with the summand corresponding to $B^{k}$ supported in degrees lying in the interval $[m_k, m_k + \dim(B^{k})]$.  With appropriate choices of perturbation data, we may ensure that the Floer complexes also admit such a splitting.

\subsection{Morse category}

Let  $Q_{1,2}$ and $Q_{2,1}$ respectively denote the closures of $U_1$ and $U_2$ in $Q_1$ and $Q_2$; these are diffeomorphic to the disc bundle associated to $NB$, and the plumbing construction fixes such a diffeomorphism.  Note that the construction of the Morse category in Section \ref{sec:stasheff} did not make any assumption about $U_i$ (other than the fact that their closures are smooth manifolds with boundary).  In particular, the construction may be extended to the case of clean intersections.

However, whenever $Q_{1,2}$ is disconnected, we need a generalisation  which accounts for the fact Floer cochains are not necessarily supported in degree $[0, k]$ and $[n-k, n]$ for a fixed integer $k$.  We choose integers $\{ m_{k} \}_{k=1}^{r}$ indexed by the component of $Q_{1,2}$  and define morphism spaces
\begin{align} \label{eq:morphisms_morse_general}
Hom_{*}^{\M}(Q_1,Q_2) & = \bigoplus_{k=1}^{r} CM^{*}(f^{k}_{1,2})[m_k] \\\
Hom_{*}^{\M}(Q_2,Q_1) & = \bigoplus_{k=1}^{r} CM^{*}(f^{k}_{2,1})[-m_k]
\end{align}
where $f^{k}_{i,j}$ is the restriction of $f_{i,j}$ to the $k$\th component of $Q_{i,j}$ and we are using the standard notation from homological algebra for a shift in degree; the grading on $Hom_{*}^{\M}(Q_i,Q_i)$ is unchanged. The shifts in gradings do not affect the degree of the contribution of each gradient tree to the higher products, so that these groups are morphisms spaces in a $\bZ$-graded $A_{\infty}$ category.  To see this, assume for simplicity that we are considering a gradient tree with $i_0 = i_d$.  An argument analogous to the one given in the proof of Lemma \ref{lem:compactness_stasheff} implies that, in this case, every labelled gradient tree that contributes to the Morse category has the property that the numbers of external edges labelled by $f^{k}_{1,2}$ and  $f^{k}_{2,1}$ agree for every $1 \leq k \leq r$, so that the formula for expected dimension still holds.  The generalisation of this argument to the case $i_0 \neq  i_d$ is left to the reader.

\subsection{Simplicial category}
Let us now choose simplicial triangulations $\sQ_{i}$ of $Q_i$ and assume that there is a pair  of codimension $0$ submanifolds with boundary $N_{i} \subset Q_{i}$ which are the closures of top dimensional simplices, each properly including $U_i$ as a deformation retract.  We assume that the diffeomorphism $U_{1} \cong U_{2}$ determined by the plumbing construction extends to a cellular homeomorphism
\begin{equation*} N_{1} \cong N_{2} .\end{equation*} 

Writing $\sN$ for the induced subdivision on an abstract manifold $N$, and $\{ \sN^{k} \}_{k=1}^{r}$ for the triangulations on the components of this space, we define a category $\Simp_{B}(Q_1,Q_2)$ with objects $Q_1$ and $Q_2$, and morphism spaces
\begin{align*}
\Hom_{*}^{\S}(Q_i,Q_i) & = C^{*}(\sQ_i) \\
\Hom_{*}^{\S}(Q_1,Q_2) & = \bigoplus_{k=1}^{r} C^{*}(\sN^{k})[m_k] \\
\Hom_{*}^{\S}(Q_2, Q_1) & =\bigoplus_{k=1}^{r}  C^{*}(\sN^{k}, \partial \sN^{k})[-m_k],
\end{align*} 
The construction of composition maps and the proofs that the appropriate diagrams commute follows word by word the arguments given in the case $N_{i}$ is the closure of a simplex.  We simply observe that the composition maps still have degree $0$ because the different components of $\sN$ do not interact, and we have chosen changes in gradings which differ by a sign for the morphisms going in opposite directions.

\subsection{From simplicial to Morse}
The construction of a functor from the simplicial to the Morse category relied on an appropriate choice of dual subdivision.  In the clean intersection case, we choose $\check{\sQ}_{i}$, and write $\check{\sN}_{i}$ for those cells dual to interior cells of $\sN$,  replacing Condition  \eqref{cond:morse_simp_compatible} by the assumption that the cells of $\check{\sQ}_{i} $ intersect the boundary of $\bar{U}_{i}$ transversely, and that we have nested weak homotopy equivalences 
\begin{equation} \check{\sN}_{i} \subset \bar{U}_{i}  \subset N_{i} .\end{equation} 
Whenever $B$ is disconnected, we assume that the choices of shifts in grading in the simplicial and Morse categories agree.  All arguments extend with only minor modifications; e.g. in Lemma \ref{lem:compactness_shrubs}, we should require the appropriate neighbourhood of the boundary of $Q_{2,1}$ not to intersect any cell of $\sQ_{2,1}$ which is dual to a cell whose interior lies entirely in $N$.

\subsection{From Floer to Morse}
In order to construct a functor from the Fukaya category to the Morse category, we start by choosing the integers $m_{k}$ in Equation \eqref{eq:morphisms_morse_general},   whenever $B$ is disconnected, to agree with those coming from Floer theory as discussed in Addendum \ref{add:generalise_floer}.

The next step is to construct a foliation of $M_B$ which we shall extend to $M$.  Near each point $b \in B$,  $M_B$ splits as the product of the disc cotangent bundle of $B$ with a domain in $\bC^{n-\dim(B)}$.  We equip this region with the product of the cotangent foliation of $T^*B$ with the model foliation $\sL$ on $D^*L_1 \# D^* L_2$ (with $\bC^{n}$ replaced by $\bC^{n-\dim(B)}$) constructed in Section \ref{sec:floer_to_morse}.  

We claim that this foliation of  $\bC^{n-\dim(B)}$ can be chosen to be invariant under the action of the orthogonal group $O(n-\dim(B))$ embedded in  unitary group (in Figure \ref{fig:foliation_model_dim_1}, this corresponds to the invariance under rotation by $\pi$), which is the structure group for the normal bundle of $T^*B$ in $M$.  To see this, recall that we first constructed a foliation by affine Lagrangian planes in Equation \eqref{eq:linear_foliation}, which we then deformed, using a Hamiltonian flow.  Since only the norm of a point in $\bR^{n- \dim(B)}$ enters in the construction in Equation \eqref{eq:linear_foliation}, the corresponding foliation is tautologically invariant under $ O(n-\dim(B)) $; one can then easily check that the construction of the deforming Hamiltonian flow can be made invariant, say by averaging it over the action of the orthogonal group.  

We conclude that the foliations defined locally on neighbourhoods of points of $B$ are compatible with each other, and define a foliation on a neighbourhood of $B$ in $M$ which agrees with the cotangent bundle foliations of $T^*Q_i$ away from a neighbourhood of $B$, and hence extends to all of $M$.  To construct a functor in this case, note that Condition \eqref{eq:morse-floer-compatible-conditions} makes sense as stated, but that we cannot choose $f_{i,j}$  to be given by the squared length of a normal vector as this function is only Morse-Bott.  However, choosing $f_{i,j}$ to be a $C^1$-small generic perturbations of this fibrewise quadratic function, one concludes that, away from a small neighbourhood of $B$, the distance to $B$ increases (or decreases) along the appropriate gradient flow, to that the analogue of Lemma \ref{lem:morse_function_grows} holds.

With the choice of Hamiltonian perturbations coming from \eqref{eq:morse-floer-compatible-conditions}, we obtain an identification between generators of Floer and Morse complexes, and we may carry through the remainder of the construction of an $A_{\infty}$ functors.  The discussion of which moduli spaces are counted, as well as issues of transversality require no modification.  As to compactness, note that
\begin{equation*}
 L_{B, - \Delta} = \cup_{b \in B}  (1 - \sqrt{-1}) N_{b} B \subset  \cup_{b \in B}  N_{b} B \otimes \bC
\end{equation*}
defines a Lagrangian submanifold of the total space of $   NB \otimes \bC $, and hence a Lagrangian in $M$ locally near $B$.   In analogy with Equation \eqref{eq:M_mid_corners}, we consider the submanifold of $M$
\begin{equation*}
  D_{1/2}Q_1 \cup   D_{1/2}Q_2 \cup \left( D_{\sqrt{2}} B  \cap  D_{\sqrt{4 - \delta}} L_{B, - \Delta} \right),
\end{equation*}
where the subscripts stand for the radii of various neighbourhoods.  Working in charts near each point $b \in B$, the discussion of Section \ref{sec:comp-curv-leaves} shows that the boundary of this manifold with corners is transverse to the Liouville flow, and hence can be smoothed to a Liouville subdomain $ M^{\mid} \subset M$.  As long as the parameter $\epsilon$ in the construction of the foliation is chosen much smaller than $\delta$, Lemma \ref{lem:cap_cpactness}, proved in Section \ref{sec:comp-curv-leaves} by using the contact type condition near $\partial M^{\mid}$ and $\partial M^{\ins}$ applies in the clean intersection case as well.

\end{document}